\documentclass[12pt]{amsart}
\usepackage{amssymb,latexsym,mathtools}
\usepackage[margin=1.25in]{geometry}
\usepackage{mathrsfs}
\usepackage{mathtools}
\usepackage{graphicx}
\usepackage{esint}
\usepackage{braket}
\usepackage{url, hyperref}
\newtheorem{theorem}{Theorem}
\newtheorem{lemma}{Lemma}
\newtheorem{definition}{Definition}

\newtheorem{prop}{Proposition}

\newtheorem*{prob}{Problem}
\newtheorem{remark}{Remark}
\newtheorem{corollary}{Corollary}

\begin{document}

\title[Analysis of a class of RDE's]{Analysis of a class of Recursive Distributional Equations Including the Resistance of the Series-parallel graph}
\author[P.S.\ Morfe]{Peter S.\ Morfe}
\address{%
  Department of Mathematics,
  Pennsylvania State University,
  219B McAllister Building,
 University Park, State College, PA 16802
}
\email{pmorfe@psu.edu}

\maketitle

\begin{abstract}  This paper analyzes a class of recursive distributional equations (RDE's) proposed by Gurel-Gurevich \cite{gurel-gurevich} and involving a bias parameter $p$, which includes the logarithm of the resistance of the series-parallel graph.  A discrete-time evolution equation resembling a quasilinear Fisher-KPP equation is derived to describe the CDF's of solutions.  When the bias parameter $p = \frac{1}{2}$, this equation is shown to have a PDE scaling limit, from which distributional limit theorems for the RDE are derived.  Applied to the series-parallel graph, the results imply that $N^{-1/3} \log R^{(N)}$ has a nondegenerate limit when $p = \frac{1}{2}$, as conjectured by Addario-Berry, Cairns, Devroye, Kerriou, and Mitchell \cite{addario-berry_cairns_devroye_kerriou_mitchell}.   \end{abstract}\




\section{Introduction}  

\subsection{Resistance of the Series-Parallel Graph} This paper seeks to address an old problem, namely, the asymptotic analysis of the resistance of the series-parallel graph.  This is the random hierarchical lattice introduced by Hambly and Jordan \cite{hambly_jordan} and constructed as follows: Start with a graph $G^{(0)}$ consisting of one edge connecting two nodes, $A$ and $B$.  At step $n$ of the construction, obtain $G^{(n)}$ from $G^{(n-1)}$ by replacing every edge by two edges, either in series or in parallel, the result being decided by independent coin flips.   Interpreting $G^{(n)}$ as a resistor network where each edge has resistance one, one can study the effective resistance $R^{(n)}$ between $A$ and $B$.

The asymptotic analysis of $R^{(n)}$ as $n \to +\infty$ was initiated in \cite{hambly_jordan}.  Since $G^{(n)}$ can be realized by combining  two independent copies $G^{(n-1)}_{1}$ and $G^{(n-1)}_{2}$ of $G^{(n-1)}$, say, in series with probability $p$ and in parallel with probability $1 - p$, the law of the resistance $R^{(n)}$ is a solution of the recursive distributional equation (RDE)
	\begin{align*}
		R^{(n)} \overset{d}{=} \left\{ \begin{array}{r l}
								R^{(n-1)}_{1} + R^{(n-1)}_{2}, & \text{with probability} \, \, p, \\
								\left( \frac{1}{R_{1}^{(n-1)}} + \frac{1}{R_{2}^{(n-1)}} \right)^{-1}, & \text{with probability} \, \, 1 - p.
							\end{array} \right.
	\end{align*}
Using this fact, a number of results were proved in \cite{hambly_jordan}: When $p > 1/2$, the logarithm $\log R^{(n)}$ grows linearly in $n$.  More precisely, there is a constant $\lambda ( p ) > 0$ such that
	\begin{align} \label{E: linear growth} 
		\mathbb{P} \left\{ (2 p - 1 ) \log 2 \leq \frac{1}{N} \log R^{(N)} \leq \lambda(p) \right\} \to 1 \quad \text{as} \, \, N \to +\infty.
	\end{align}
Due to symmetry, this implies that $\log R^{(n)} \to -\infty$ linearly when $p < 1/2$.  

At the critical point $p = 1/2$, Hambly and Jordan instead proved that $\log R^{(n)}$ exhibits sublinear growth.  The main results of this paper (see Theorem \ref{T: main subdiffusive} and Corollary \ref{C: distributional limit} below) show that, in fact, $\log R^{(N)}$ grows like $N^{1/3}$.  Specifially,
	\begin{align*}
		\frac{1}{ (72 \zeta(3) )^{ \frac{1}{3} } N^{\frac{1}{3}} } \log R^{(N)} + \frac{1}{2} \overset{d}{\to} \text{Beta}(2,2) \quad \text{as} \, \, N \to +\infty,
	\end{align*}
where $\zeta$ is the Riemann zeta function.  This confirms a conjecture of Addario-Berry, Cairns, Devroye, Kerriou, and Mitchell \cite{addario-berry_cairns_devroye_kerriou_mitchell}.

The starting point for this work is an observation made by Gurel-Gurevich during a presentation at the open problem session at the Workshop for Disordered Media at the R\'{e}nyi Institute in Budapest in January 2025 \cite{gurel-gurevich}.  He noted that the transformation $X^{(n)} = \log R^{(n)}$ leads to an RDE of the form
	\begin{align}
		X^{(n)} =  \Phi(X^{(n-1)}_{1},X^{(n-1)}_{2}),  \label{E: gurel-gurevich RDE}
	\end{align}
where $\Phi$ is the random function defined by
	\begin{align}
		\Phi(x,y) = \left\{ \begin{array}{r l}
					\max\{ x,y \} + f ( |y - x|), & \text{with probability} \, \, p, \\
					\min\{ x,y \} - f ( |y - x|), & \text{with probability} \, \, 1 - p,
				\end{array} \right. \label{E: Gurel-Gurevich force}
	\end{align}
and $f(u) = \log ( 1 + e^{-u} )$.  This led him to pose the following problem:

	\begin{prob} Consider the RDE \eqref{E: gurel-gurevich RDE} with forcing $\Phi$ given by \eqref{E: Gurel-Gurevich force} under general assumptions on the nonincreasing function $f : [0,+\infty) \to [0,+\infty)$.  What can be said about the long-time behavior of the law of solutions? \end{prob}
	
This paper endeavors to address this problem in some generality when $p = 1/2$, focusing on the distributional limit of $n^{-1/3} X^{(n)}$.  The general framework developed here applies not only to the resistance of the critical series-parallel graph, but also to some other RDE's considered recently in the literature.  For instance, as is shown in Section \ref{S: examples} below, the class of RDE's determined by \eqref{E: gurel-gurevich RDE} and \eqref{E: Gurel-Gurevich force} includes the symmetric hipster random walk from \cite{addario-berry_cairns_devroye_kerriou_mitchell} and two-player symmetric cooperative motion, one of the RDE's analyzed very recently by Addario-Berry, Beckman, and Lin in \cite{addario-berry_beckman_lin-symmetric}.  

While the analysis of the asymptotic behavior in the case $p \neq 1/2$ is not touched on here, the framework developed in this paper suggests that it may be amenable to techniques developed in the study of reaction-diffusion equations.  This is discussed further in Section \ref{S: KPP} below.

A few days after this paper first appeared as a preprint, a work of Chen, Duquesne, and Shi \cite{chen_duquesne_shi} appeared, which considers the same class of RDE's and independently obtains the limit theorem for $n^{-1/3} X^{(n)}$.

\subsection{General Setting} Consider the following general setting, which generalizes the forcing $\Phi$ in  \eqref{E: gurel-gurevich RDE}.    Given a Bernoulli random variable $\Theta$ and random nonincreasing functions $f_{+}, f_{-} : \mathbb{R} \to [0,+\infty)$, define $\Phi$ by
	\begin{align} \label{E: main forcing}
		\Phi(x,y) = \left\{ \begin{array}{r l}
						\max\{x,y\} + f_{+}(|x - y|), \quad \text{if} \, \, \Theta = 1, \\
						\min\{x,y\} - f_{-}(|x-y|), \quad \text{if} \, \, \Theta = 0.
					\end{array} \right.
	\end{align}
Without loss of generality, the pair $(f_{+},f_{-})$ will be assumed to be independent of $\Theta$ henceforth.  

The object of interest in this paper is the RDE determining the law of random variables $(X^{(n)})_{n \in \mathbb{N}}$ through the recursion
	\begin{align} \label{E: main RDE}
		X^{(n)} \overset{d}{=} \Phi( X^{(n-1)}_{1},X^{(n-1)}_{2}),
	\end{align}
where $X^{(n-1)}_{1},X^{(n-1)}_{2}$ are i.i.d. copies of $X^{(n-1)}$, which are independent of $\Phi$.  Throughout the paper, equality is only required to hold in law; questions of couplings or constructions of $\{X^{(n)}\}$ will not be treated.

The form \eqref{E: main forcing} of $\Phi$ is not just a repackaging of \eqref{E: Gurel-Gurevich force}.  In the asymmetric setting where $f_{+} \neq f_{-}$, this class of RDE's includes the one associated with Pemantle's min-plus binary tree, the asymptotics of which were analyzed by Auffinger and Cable \cite{auffinger_cable}, as well as certain versions of the asymmetric hipster random walk from \cite{addario-berry_cairns_devroye_kerriou_mitchell} and two-player totally asymmetric, $q$-lazy cooperative motion, one of the RDE's analyzed by Addario-Berry, Beckman, and Lin in \cite{addario-berry_beckman_lin-asymmetric}.  In all of these examples, as well as the symmetric models from \cite{addario-berry_cairns_devroye_kerriou_mitchell,addario-berry_beckman_lin-symmetric} already mentioned above, the asymptotic analysis carried out here generalizes the limit theorems previously obtained within a unified setting.

Henceforth define the bias parameter $p$ by
	\begin{equation*}
		p= \mathbb{P} \{ \Theta = 1\}.  
	\end{equation*}

The joint law of the functions $(f_{+},f_{-})$ will be denoted by $\mathbf{P}$ in what follows.  $\mathbf{P}$ is assumed to be a Borel probability measure on the space $BC([0,+\infty)) \times BC([0,+\infty))$ of pairs of bounded continuous functions in $[0,+\infty)$ with the uniform norm topology.  In most of the examples of interest here, however, $f_{+}$ and $f_{-}$ are deterministic; see the discussion in Section \ref{S: examples}.

\subsection{Evolution of the CDF} The main results of this paper are obtained by treating the evolution of the law of $X^{(n)}$ as if it were governed by a parabolic partial differential equation (PDE).  Following \cite{addario-berry_beckman_lin-asymmetric,addario-berry_beckman_lin-symmetric}, this is done at the level of cumulative distribution functions (CDF's) rather than probability measures.  

In particular, in what follows, define the operator $T$ by taking a CDF $F$, letting $X_{1},X_{2}$ be i.i.d.\ random variables distributed according to $F$, and defining $T F$ to be the CDF of the random variable $Y$ determined by
	\begin{align} \label{E: T definition}
		Y = \Phi(X_{1},X_{2}),
	\end{align}
where $\Phi$ is sampled independently of $(X_{1},X_{2})$. With this notation, it follows that the random variables $\{ X^{(n)} \}$ solve the RDE \eqref{E: main RDE}  if and only if the corresponding CDF's $\{ F_{n} \}$ solve the recursive equation
	\begin{equation*}
		F_{n} = TF_{n-1}.
	\end{equation*}

In the first step of this work, this recursion is reformulated as a discrete-time evolution equation resembling a parabolic PDE.  Namely, there is a nonlinear operator $\mathscr{L}$ defined on the space of CDF's such that
		\begin{align} \label{E: KPP-like equation}
			F_{n} - F_{n-1} = \mathscr{L}F_{n-1} + (1 - 2p) F_{n-1} ( 1 - F_{n-1} ).
		\end{align}
	As will be argued below, $\mathscr{L}$ has the character of an advection-diffusion operator, albeit a nonlinear one, while the ``reaction term" $F (1 - F)$ is familiar from the Fisher-KPP equation.  The equation thus resembles a reaction-diffusion equation.
	
A key property of the equation \eqref{E: KPP-like equation} is its monotonicity: If $F \leq G$ pointwise, then $TF \leq TG$ also holds.  This means that \eqref{E: KPP-like equation} is amenable to techniques developed in the theory of parabolic PDE's.  The monotonicity of $T$ also plays a fundamental role in the aforementioned work of Chen, Duquesne, and Shi \cite{chen_duquesne_shi}.  Their proof involves controlling CDF's through the delicate construction of explicit sub- and supersolutions of \eqref{E: KPP-like equation}.  By contrast, the proof here is more conceptual: It starts by establishing that the evolution equation \eqref{E: KPP-like equation} ``looks" like a PDE under a certain rescaling of the variables $(x,n)$, then invokes abstract convergence results for monotone semigroups and parabolic PDE's.  In addition to the clear intuitive interpretation, the method has the advantage that it applies very generally, including to cases in which explicit formulae are not available.

\subsection{Scaling Limits at the Critical Point $p = 1/2$} As far as convergence in distribution is concerned, deriving a limit for $N^{-1/\alpha} X^{(N)}$ for some exponent $\alpha$ is equivalent to proving a scaling limit for the equation \eqref{E: KPP-like equation}.  In the main result of this paper, which treats the critical point $p = 1/2$, this is done by adapting the classical approach to numerical approximations of parabolic PDE's developed by Barles and Souganidis \cite{barles_souganidis}.

To have a sense of why it is reasonable to expect a PDE scaling limit, consider what happens in \eqref{E: main forcing} when $f_{+} = f_{-} = f_{\mathbb{Z}}$, where $f_{\mathbb{Z}}(u) = (1 - u)_{+}$.  In this case, if the initial datum $X^{(0)}$ is integer-valued, then that remains true of $X^{(n)}$ for all $n$, and thus the CDF $F_{n}$ can be regarded as a function on $\mathbb{Z}$.  In fact, these $\mathbb{Z}$-valued solutions of the RDE coincide with two-player symmetric cooperative motion, one of the RDE's introduced in \cite{addario-berry_beckman_lin-symmetric}.  As shown in that work, the evolution equation governing the CDF now takes a very concrete form, namely,
	\begin{align*}
		F_{n} - F_{n-1} = | \nabla_{\mathbb{Z}} F_{n-1} | \Delta_{\mathbb{Z}} F_{n-1} \quad \text{in} \, \, \mathbb{Z},
	\end{align*}
where $\nabla_{\mathbb{Z}}$ and $\Delta_{\mathbb{Z}}$ are certain discrete derivative operators.  This observation was then used in \cite{addario-berry_beckman_lin-symmetric} in conjunction with the method of \cite{barles_souganidis} and a regularization effect of the equation to show that, after a suitable rescaling, the CDF converges to the solution of the initial-value problem
	\begin{align*}
		\partial_{t} F = | \partial_{x} F | \partial_{x}^{2} F, \quad F(x,0) = \mathbf{1}_{[0,\infty)}(x).
	\end{align*}
After differentiation, this PDE becomes the porous medium equation satisfied by the PDF $\rho = \partial_{x} F$ with Dirac initial data, the solution of which is known to be a self-similarly growing parabolic cap (the so-called Barenblatt solution).  In probabilistic terms, this implies the convergence of $N^{-1/3} X^{(N)}$ to a shifted and dilated Beta(2,2) random variable.  

In this work, this general strategy --- derivation of a discrete-time evolution equation and application of a convergence result in the spirit of \cite{barles_souganidis} --- is shown to still apply in the setting of the RDE's defined above, although its execution is complicated somewhat by the fact that the operator $\mathscr{L}$ is generally nonlocal.

\subsubsection{Main Results} \label{S: main results} The main result holds under some assumptions on the law $\mathbf{P}$ of $(f_{+},f_{-})$.  First, define the class of functions $\mathcal{S}$ so that $f \in \mathcal{S}$ if and only if
		\begin{gather}
			f : \mathbb{R} \to \mathbb{R} \quad \text{is continuous and nonincreasing, and} \label{E: nonincreasing}\\
			u \mapsto u + f(u) \, \, \text{is a nondecreasing function in} \, \, [0,+\infty). \label{E: map} 
		\end{gather}
Note that the function $f(u) = \log ( 1 + e^{-u} )$ relevant to the series-parallel graph belongs to $\mathcal{S}$ since $-1 \leq f' \leq 0$.

Throughout the paper, $f_{+}$ and $f_{-}$ are assumed to lie in $\mathcal{S}$.  For $\rho \in \{+,-\}$, let $g_{\rho} : [0,+\infty) \to [0,+\infty)$ be the function
	\begin{align} \label{E: right inverse}
		g_{\rho}(s) = \sup \left\{ u \geq 0 \, \mid \, u + f_{\rho}(u) \leq s \right\} \quad \text{if} \, \, s \geq f_{\rho}(0), \quad g_{\rho}(s) = 0, \quad \text{otherwise.}
	\end{align}
Notice that the restriction of $g_{\rho}$ to the set $[f_{\rho}(0),+\infty)$ is the right-continuous right-inverse of the function $u \mapsto u + f_{\rho}(u)$.

By definition, if $f_{\pm}(+\infty) = 0$, then $g_{\pm}(s) - s \to 0$ as $s \to +\infty$.  The next assumptions ask that $g_{\pm}(s) - s$ vanishes in a quantitative manner: First, under the assumption that
	\begin{align} \label{E: decay one}
		\int_{0}^{+\infty} \mathbf{E}[ |g_{+}(s) - s| + |g_{-}(s) - s| ] \, ds < +\infty,
	\end{align}
define the constant $\sigma$ by 
	\begin{align*}
		\sigma = - \int_{0}^{+\infty} \mathbf{E}[ |g_{+}(s) - s| ] \, ds +  \int_{0}^{+\infty} \mathbf{E} [ |g_{-}(s) - s| ] \, ds.
	\end{align*}
If in addition the following stronger assumption holds
	\begin{align} \label{E: decay two}
		\int_{0}^{+\infty} (1 + s ) \mathbf{E}[ |g_{+}(s) - s| + |g_{-}(s) - s| ] \, ds < +\infty,
	\end{align}
let $a \geq 0$ be the constant defined by
	\begin{align*}
		a = \frac{1}{2} \int_{0}^{+\infty} \mathbf{E} [ (3 s - g_{+}(s) ) ( s - g_{+}(s) ) ] \, ds + \frac{1}{2} \int_{0}^{+\infty} \mathbf{E}[ (3 s - g_{-}(s)) ( s - g_{-}(s) ) ] \, ds.
	\end{align*}
The nonnegativity of $a$ follows from the fact that $g_{\pm}(s) \leq s$ for any $s \geq 0$.

The next two theorems characterize the scaling limit of the CDF at the critical point $p = 1/2$ in terms of a family of parabolic PDE's parametrized by $\sigma$ and $a$.  The relevant background from PDE theory, particularly the theory of viscosity solutions, is reviewed in Appendix \ref{A: viscosity solutions} below.  Motivated by recent work of Chen, Derrida, Duquesne, and Shi \cite{chen_derrida_duquesne_shi}, the theorem allows for the possibility that $p = p^{(N)}$ converges to $1/2$ as $N \to +\infty$.

	\begin{theorem} \label{T: main diffusive} Assume that $\mathbf{P}\{ f_{+},f_{-} \in \mathcal{S} \} = 1$,  \eqref{E: decay one} holds, and $\sigma \neq 0$.  Fix a $\theta \in \mathbb{R}$ and let $p^{(N)} = 1/2 + \theta N^{-1}$.  For each $N \in \mathbb{N}$, let $\{X^{(N,n)}\}$ be a solution of the RDE 
	\begin{equation*}
		X^{(N,n)} \overset{d}{=} \Phi^{(N)} ( X^{(N,n-1)}_{1}, X^{(N,n-1)}_{2} ),
	\end{equation*}
where $(X^{(N,n-1)}_{1},X^{(N,n-1)}_{2})$ are i.i.d.\ copies of $X^{(N,n-1)}$ and $\Phi^{(N)}$ is given by \eqref{E: main forcing} with bias parameter $p = p^{(N)}$.

	Assume that there is a function $F_{\text{in}}$ such that, at the initial time $n = 0$, the following limit holds:
		\begin{align*}
			\lim_{N \to +\infty} \mathbb{P} \{ N^{-1/2} X^{(N,0)} \leq x\} = F_{\text{in}}(x) \quad \text{for almost every} \, \, x \in \mathbb{R}.
		\end{align*}
	If $\{F_{N}\}$ is the sequence of rescaled CDF's defined in $\mathbb{R} \times [0,+\infty)$ by
		\begin{align*}
			F_{N}(x,t) = \mathbb{P} \{ N^{-1/2} X^{(N, [ N t ] )} \leq x \},
		\end{align*}
	then $F_{N} \to F$ locally uniformly in $\mathbb{R} \times (0,+\infty)$ as $N \to +\infty$, where $F$ is the unique bounded discontinuous viscosity solution of the initial value problem
				\begin{align} \label{E: HJ IVP}
					\left\{ \begin{array}{r l}
						\partial_{t} F - \sigma | \partial_{x} F |^{2} + 2 \theta F ( 1 - F ) = 0 & \text{in} \, \, \mathbb{R} \times (0,+\infty), \\
						F(x,0) = F_{\text{in}}(x).
					\end{array} \right.
				\end{align}
	\end{theorem}
	
See Section \ref{S: counterexample} for an example in which \eqref{E: decay one} fails to hold and $N^{-1/2} X^{(N)} \overset{d}{\to} +\infty$.
	
If $\theta = \sigma = 0$ above, then the proof of the theorem instead implies that $F_{N}$ converges to the initial datum $F_{\text{in}}$, meaning that $\{X^{(N,n)}\}$ exhibits no nontrivial motion in the diffusive scaling limit.  Under the stronger assumption \eqref{E: decay two}, one finds that instead the appropriate scaling is subdiffusive.

\begin{theorem} \label{T: main subdiffusive} Assume that $\mathbf{P}\{ f_{+},f_{-} \in \mathcal{S} \} = 1$, \eqref{E: decay two} holds, $\sigma = 0$, and $a > 0$.  Fix a $\theta \in \mathbb{R}$ and let $p^{(N)} = 1/2 + \theta N^{-1}$.  For each $N \in \mathbb{N}$, let $\{X^{(N,n)}\}$ be a solution of the RDE 
	\begin{equation*}
		X^{(N,n)} \overset{d}{=} \Phi^{(N)} ( X^{(N,n-1)}_{1}, X^{(N,n-1)}_{2} ),
	\end{equation*}
where $(X^{(N,n-1)}_{1},X^{(N,n-1)}_{2})$ are i.i.d.\ copies of $X^{(N,n-1)}$ and $\Phi^{(N)}$ is given by \eqref{E: main forcing} with bias parameter $p = p^{(N)}$.

Assume that there is a function $F_{\text{in}}$ such that, at the initial time $n = 0$, the following limit holds:
		\begin{align*}
			\lim_{N \to +\infty} \mathbb{P} \{ N^{-1/3} X^{(N,0)} \leq x\} = F_{\text{in}}(x) \quad \text{for almost every} \, \, x \in \mathbb{R}.
		\end{align*}
	If $\{F_{N}\}$ is the sequence of rescaled CDF's defined in $\mathbb{R} \times [0,+\infty)$ by
		\begin{align*}
			F_{N}(x,t) = \mathbb{P} \{ N^{-1/3} X^{(N, [ N t ] )} \leq x \},
		\end{align*}
	then $F_{N} \to F$ locally uniformly in $\mathbb{R} \times (0,+\infty)$ as $N \to +\infty$, where $F$ is the unique bounded discontinuous viscosity solution of the initial-value problem
				\begin{align} \label{E: porous medium IVP}
					\left\{ \begin{array}{r l}
						\partial_{t} F - a | \partial_{x} F | \partial_{x}^{2} F + 2 \theta F ( 1 - F ) = 0 & \text{in} \, \, \mathbb{R} \times (0,+\infty), \\
						F(x,0) = F_{\text{in}}(x).
					\end{array} \right.
				\end{align}
	\end{theorem}

		
The PDE's derived above become more familiar when passing to PDF's and setting $\theta = 0$: The spatial derivative $\rho = \partial_{x} F$ is a solution of Burger's equation $\partial_{t} \rho = \sigma \partial_{x} ( \rho^{2} )$ in Theorem \ref{T: main diffusive} and a solution of the porous medium equation $\partial_{t} \rho = \frac{1}{2} a \partial_{x}^{2} ( \rho^{2} )$ in Theorem \ref{T: main subdiffusive}.  

As observed in \cite{addario-berry_beckman_lin-asymmetric,addario-berry_beckman_lin-symmetric}, when $\theta = 0$ and $F_{\text{in}} = \mathbf{1}_{[0,\infty)}$, the solutions of the PDE's above are explicitly known and closely related to the $\text{Beta}(2,1)$ and $\text{Beta}(2,2)$ distributions.  This leads to limit theorems for the RDE \eqref{E: main RDE} started from an arbitrary (fixed) initial distribution.

	\begin{corollary} \label{C: distributional limit} Assume that $p = 1/2$, $\mathbf{P}\{ f_{+},f_{-} \in \mathcal{S} \} = 1$, and \eqref{E: decay one} holds, and let $\{X^{(n)}\}$ be a solution of the RDE \eqref{E: main RDE} with an arbitrary initial distribution. 
		\begin{itemize}
			\item[(i)] The sequence $\{ N^{-1/2} X^{(N)}\}$ has the following distributional limit:
				\begin{align*}
					-\frac{ \text{sgn} ( \sigma ) }{ 2 \sqrt{ N } } X^{(N)} \overset{d}{\to} \sqrt{|\sigma|} \text{Beta}(2,1),
				\end{align*}
			where $\sqrt{|\sigma|} \text{Beta}(2,1)$ denotes a $\text{Beta}(2,1)$ random variable dilated by a factor $\sqrt{|\sigma|}$.  In particular, if $\sigma = 0$, then the limit is zero.
			\item[(ii)] Under the stronger assumption \eqref{E: decay two}, if $\sigma = 0$ and $a > 0$, then $\{ N^{-1/3} X^{(N)} \}$ instead has a nontrivial distributional limit:
				\begin{align*}
					\frac{ 1 }{ (36 a)^{\frac{1}{3}} N^{\frac{1}{3}} } X^{(N)} + \frac{1}{2} \overset{d}{\to} \text{Beta}(2,2)
				\end{align*}
			\item[(iii)] Under \eqref{E: decay two}, if $\sigma = a = 0$, then $\mathbf{P} \{ f_{+} = f_{-} = 0\} = 1$ and $\{ X^{(n)} \}$ is the constant sequence $X^{(n)} \overset{d}{=} X^{(0)}$.
		\end{itemize}
	\end{corollary}

It is not hard to show, through manipulations of the definitions above, that $\sigma > 0$ holds, for instance, if 
$f_{-} > f_{+}$ $\mathbf{P}$-almost surely, while $\sigma = 0$ if $f_{+} = f_{-}$.  In the case of the resistance $R^{(n)}$ of the critical series-parallel graph, $\sigma =0$ and $a = 2 \zeta(3)$.  For more on this and further discussion of examples, see Section \ref{S: examples}.  

In addition to Corollary \ref{C: distributional limit}, Theorems \ref{T: main diffusive} and \ref{T: main subdiffusive} have two applications that are worth highlighting now.  In Section \ref{S: resistance} below, Theorem \ref{T: main subdiffusive} is used to characterize the asymptotics of the resistance of the critical series-parallel graph in the general case when the edges have i.i.d.\ resistances, which could equal zero or $+\infty$.  Similarly, Section \ref{S: distance} gives a characterization of the limiting behavior of Pemantle's min-plus binary tree (or first passage percolation on the critical series-parallel graph) when the weights on the leaves are i.i.d.\ with an arbitrary distribution, answering a question posed in \cite{auffinger_cable}.

\subsubsection{Idea of the Proof} The PDE scaling limit of the CDF is obtained using an extension of the approach of \cite{barles_souganidis}, which is broadly applicable to elliptic and parabolic PDE's studied through the lens of Crandall and Lions' theory of viscosity solutions.  There are two key ingredients in that approach, namely, \emph{monotonicity} and \emph{consistency}.  

As mentioned already above, monotonicity refers to the fact that the equation \eqref{E: KPP-like equation} is order-preserving: If $\{F_{n}\}$ and $\{G_{n}\}$ are two solutions, and if $F_{0} \leq G_{0}$ pointwise, then the inequality $F_{n} \leq G_{n}$ remains true for any subsequent $n$.

Consistency, on the other hand, refers to the behavior of the equation under rescaling.  Specifically, Proposition \ref{P: consistency condition} below shows that if $G$ is a smooth CDF, rescaled according to $
G_{\delta}(x) = G(\delta x)$, then
	\begin{align} \label{E: asymptotic expansion intro}
		(\mathscr{L}G_{\delta})( \delta^{-1} x) = \delta^{2} \sigma | \partial_{x} G(x) |^{2} + \delta^{3} a | \partial_{x} G(x) | \partial_{x}^{2} G(x) + \cdots \quad \text{as} \, \, \delta \downarrow 0.
	\end{align}
Rescaling the sequence of CDF's $\{F_{n}\}$ according to $F_{N}(x,t) = F_{[N t]}( N^{1/\alpha} x)$, na\"{i}vely inserting the above expansion into the equation \eqref{E: KPP-like equation}, and setting $p = 1/2 + \theta N^{-1}$, one obtains the formal asymptotic expansion
	\begin{align*}
		\partial_{t} F_{N} = N^{1 - \frac{2}{\alpha} } \sigma | \partial_{x} F_{N}|^{2} + N^{1 - \frac{3}{\alpha} } a | \partial_{x} F_{N}|^{2} - 2 \theta F_{N} ( 1 - F_{N} ) + \cdots \quad \text{as} \, \, N \to + \infty,
	\end{align*}
which motivates the choices $\alpha = 2$ and $\alpha = 3$ used above and the appearance of the PDE's \eqref{E: HJ IVP} and \eqref{E: porous medium IVP}.  The approach of \cite{barles_souganidis} provides a robust framework for making such a formal argument rigorous.

\subsection{Variants of the Fisher-KPP Equation when $p \neq 1/2$} \label{S: KPP} The results of this work suggest interesting new questions away from the critical point $p = 1/2$.  In view of the asymptotic expansion \eqref{E: asymptotic expansion intro} and the PDE's obtained in Theorems \ref{T: main diffusive} and \ref{T: main subdiffusive}, the equation \eqref{E: KPP-like equation} shares some similarities with the reaction-diffusion equation
	\begin{align} \label{E: KPP}
		\partial_{t} G = a | \partial_{x} G | \partial_{x}^{2} G + \sigma | \partial_{x} G |^{2} - 2 \theta G ( 1 - G ).
	\end{align}
In fact, \eqref{E: KPP-like equation} becomes the discrete-in-space-time, finite-difference version of \eqref{E: KPP} for suitable choices of the forcing $(f_{+},f_{-})$; see Section \ref{S: finite difference} for more details.

A number of works in the PDE literature considered the equation \eqref{E: KPP} in the case when $\sigma = 0$ and $a > 0$.  Using an auxiliary ODE, Engui\c{c}a, Gavioli, and Sanchez \cite{enguica_gavioli_sanchez} established the existence of traveling waves propagating at any large-enough speed, as in the classical Fisher-KPP equation.  The work of Audrito and V\'{a}zquez \cite{audrito_vazquez} establishes existence, uniqueness, and positivity properties of traveling waves and also analyzes the asymptotic behavior of solutions, proving that sufficiently rapidly decaying perturbations of the unstable state (say, $0$ if $\theta < 0$) converge to the stable state ($1$ if $\theta < 0$) inside a ballistically growing region, with the rate of growth determined by the minimal wave speed.  It would be interesting to adapt the techniques of \cite{audrito_vazquez} to the study of \eqref{E: KPP-like equation} away from the critical point.

\subsection{Related Literature} There is considerable interest in RDE's in the probability literature.  The reader is referred to the survey article of Aldous and Bandyopadhyay \cite{aldous_bandyopadhyay}, which discusses applications in the study of branching processes, percolation, and mean-field combinatorial optimization problems. 

There has been a recent burst of activity concerning RDE's in which the relevant asymptotic behavior is not convergence to a fixed point as in the examples in \cite{aldous_bandyopadhyay}, but instead a distributional scaling limit  in the same spirit as the central limit theorem.  

Auffinger and Cable \cite{auffinger_cable} analyzed Pemantle's min-plus binary tree (or, equivalently, first passage percolation on the critical series-parallel graph, see \cite[Section 3]{hambly_jordan}), proving that solutions of the associated RDE converge to a $\text{Beta}(2,1)$ distribution after a diffusive rescaling provided the initial distribution is concentrated at one.  This is improved to the case of an arbitrary initial distribution in Corollary \ref{C: critical distance} below, answering one of the open problems posed therein.

The proof in \cite{auffinger_cable} is based on the explicit construction of sub- and supersolutions of \eqref{E: KPP-like equation}, which, by monotonicity, can be used to control solutions, or what is referred to as a barrier argument in the PDE literature.  The same strategy was employed by Chen, Duquesne, and Shi in \cite{chen_duquesne_shi}, who independently derived the $\text{Beta}(2,2)$ limit theorem of Corollary \ref{C: distributional limit} above, albeit under slightly more restrictive assumptions on $(f_{+},f_{-})$.  

Addario-Berry and coauthors recently introduced two classes of RDE's, called hipster random walks and cooperative motion, in \cite{addario-berry_cairns_devroye_kerriou_mitchell, addario-berry_beckman_lin-asymmetric, addario-berry_beckman_lin-symmetric}.  In \cite{addario-berry_cairns_devroye_kerriou_mitchell}, $\text{Beta}(2,1)$ and $\text{Beta}(2,2)$ limit theorems for some hipster random walks were proved by treating the evolution of the PMF as a numerical approximation of the solution of a divergence-form PDE.  As discussed therein, some of the motivation to study these RDE's came from interest in the resistance of the series-parallel graph and Pemantle's min-plus binary tree.  

Cooperative motion was analyzed in \cite{addario-berry_beckman_lin-asymmetric, addario-berry_beckman_lin-symmetric}, this time through the analysis of the CDF and using techniques developed for PDE's in nondivergence form, specifically the theory of viscosity solutions.    In the present paper, the idea to focus attention on the CDF came from these last two works.

The approach used here is similar in spirit to that in \cite{addario-berry_beckman_lin-asymmetric,addario-berry_beckman_lin-symmetric}.  A few differences are worth highlighting.  Many of the RDE's treated in \cite{addario-berry_beckman_lin-asymmetric,addario-berry_beckman_lin-symmetric} are not monotone, hence the method of \cite{barles_souganidis} does not apply out of the box.  Instead, those works showed that, after a finite time, the law of $X^{(n)}$ is pulled into a region in which the RDE \emph{is} monotone; at the level of the CDF, the key observation is that the CDF obtains a uniform Lipschitz bound in a (universal) finite time, which mimics the properties of the PDE's that emerge in the scaling limit.  After this finite waiting period, monotonicity can be used as in \cite{barles_souganidis} (or, in the case of \cite{addario-berry_beckman_lin-asymmetric}, following the earlier work \cite{crandall_lions}).  

By contrast, here the RDE's of interest are all monotone.  There is a technical issue, namely, that the result of \cite{barles_souganidis} does not apply directly as the evolution \eqref{E: KPP-like equation} only makes sense as an equation posed in the space of CDF's, whereas \cite{barles_souganidis} works with evolution equations posed in the (vector) space of bounded functions.  In \cite{addario-berry_beckman_lin-symmetric}, this was arguably less of an issue as the equations all involve finite differences (although the relevant details were not discussed there).  Here $\mathscr{L}$ behaves like a nonlocal operator with a kernel with infinite range so at first blush one would like to know that \cite{barles_souganidis} can be carried out with globally nondecreasing test functions, which is not obvious.  It is, in fact, possible to adapt the method to monotone semigroups in the space of CDF's.  In the hope that it may be useful in other contexts, the result is formulated in an abstract way in Appendix \ref{A: monotone schemes} below.


When $p \neq \frac{1}{2}$, the evolution equation obtained here and the analogy with reaction-diffusion equations suggests a strategy for tackling the asymptotics of the RDE.  This is of interest both in the analysis of the resistance and the distance on the series-parallel graph.  The latter was recently considered by Chen, Derrida, Duquesne, and Shi in \cite{chen_derrida_duquesne_shi}.   They proved that when $p > \frac{1}{2}$, the logarithm of the expected value of the distance grows at a linear rate with slope $\alpha(p)$, and, again using barrier arguments, they proved that $\alpha( 1/2 + \epsilon ) \approx \sqrt{ \zeta(2) \epsilon }$ as $\epsilon \to 0$.  Their proof is motivated through the heuristic derivation of a PDE that can be shown to be closely related to \eqref{E: HJ IVP}.

As mentioned above, convergence of the rescaled CDF's to the solution of a PDE is proved here by adapting the classical approach of Barles and Souganidis from \cite{barles_souganidis}.  At a high level, as soon as the monotonicity and scaling behavior (``consistency") of $\mathscr{L}$ is understood, convergence follows.  This is not the only setting where some extra tailoring of the method is needed to furnish a proof.  For instance, in \cite{barles_souganidis-interface}, the same authors adapted their approach to the setting of geometric flows and interface motions.  

\subsection{Organization of the Paper} Section \ref{S: prelim} discusses some preliminaries used throughout the paper.  The key monotonicity result is proved there.  The equation describing the evolution of CDF's is derived in Section \ref{S: equation}. Section \ref{S: examples} discusses examples that fit into the class described by \eqref{E: main RDE} and \eqref{E: main forcing}.  The main results are proved in Section \ref{S: scaling limit} conditional on an extension of \cite{barles_souganidis} and some uniqueness results for discontinuous viscosity solutions, which are presented in the appendix.

There are three appendices.  Necessary terminology and results from the theory of viscosity solutions are presented in Appendix \ref{A: viscosity solutions}.  A variant of the convergence result of \cite{barles_souganidis}, tailored to monotone semigroups in the space of CDF's, is stated and proved in Appendix \ref{A: monotone schemes}.  Finally, Appendix \ref{A: polylog} includes computations of the coefficients $\sigma$ and $a$ arising respectively in Pemantle's min-plus binary tree and the resistance of the critical series-parallel graph  

\section*{Acknowledgements} 

This work was made possible in part by the hospitality and support of the R\'{e}nyi Institute in Budapest, Hungary.  The initial inspiration came from the talk \cite{gurel-gurevich} delivered during the Workshop on Disordered Media hosted by the Erd\H{o}s Center during the Simons Semester on Probability and Statistical Physics in the spring of 2025.




\section{Preliminaries} \label{S: prelim}

\subsection{Extended Real-Valued Random Variables} Note that the definition \eqref{E: main RDE} still makes sense if the input random variables are permitted to take the values $+\infty$ and $-\infty$.  It will be convenient to extend the definition to include such extended real-valued random variables.  To that end, for an extended real-valued random variable $X$, define the CDF $F : \mathbb{R} \to [0,1]$ via the rule
	\begin{align*}
		\mathbb{P}\{X = -\infty\} = \lim_{y \to -\infty} F(y), \quad \mathbb{P}\{X \leq x\} = F(x), \quad \mathbb{P}\{X = +\infty\} = 1 - \lim_{y \to +\infty} F(y).
	\end{align*}
Notice that the space $CDF(\overline{\mathbb{R}})$ of all such CDF's is simply
	\begin{align*}
		CDF(\overline{\mathbb{R}}) = \{ F : \mathbb{R} \to [0,1] \, \, \text{nondecreasing, right-continuous}  \}.
	\end{align*}
Using standard extended real number arithmetic, from now on, the formula \eqref{E: T definition} defining $T$ will be understood to hold in the domain $CDF(\overline{\mathbb{R}})$. 

\subsection{Continuity under Vague Convergence.} It is useful to note that $T$ is continuous with respect to vague convergence.  Recall that $\{F_{n}\}$ converges vaguely to $F$ if $F_{n}(x) \to F(x)$ at each point of continuity $x \in \mathbb{R}$ of $F$.

	\begin{prop} \label{P: vague convergence} If $\{F_{n}\}$ is a sequence in $CDF(\overline{\mathbb{R}})$ converging vaguely to some $F$, then $\{ TF_{n} \}$ converges vaguely to $TF$. \end{prop}
	
		\begin{proof} It suffices to recall that this mode of convergence for $\{ F_{n} \}$ is equivalent to the existence of a probability space $\mathbb{P}$ supporting a sequence $\{ (X_{1,n},X_{2,n}) \}$ and a vector $(X_{1},X_{2})$ such that, for each $n$, $X_{1,n}$ and $X_{2,n}$ are i.i.d.\ with CDF $F_{n}$; $(X_{1},X_{2})$ are i.i.d.\ with CDF $F$; and $(X_{1,n},X_{2,n}) \to (X_{1},X_{2})$ $\mathbb{P}$-almost surely.  (Note that the possibility that $X_{1}$ or $X_{2}$ are infinite poses no difficulties here.)  Letting $Y_{n} = \Phi ( X_{1,n}, X_{2,n})$ and $Y = \Phi ( X_{1}, X_{2} )$, where $\Phi$ is sampled independently of $\{ (X_{1,n},X_{2,n}) \}$ and $(X_{1},X_{2})$, the boundedness and continuity of $f_{+}$ and $f_{-}$ are enough to deduce that $Y_{n} \to Y$ almost surely as $n \to \infty$.  By definition of $T$, this implies $TF_{n} \to TF$.         \end{proof}

\subsection{Symmetry} \label{S: symmetry} In order to reduce redundancy in the proofs that follow, it is worth observing that if $Y = \Phi(X_{1},X_{2})$, where $\Phi$ is given by \eqref{E: main forcing}, then 
	\begin{align*}
		- Y = \Psi( - X_{1}, - X_{2} ),
	\end{align*}
provided $\Psi$ is defined by
	\begin{align*}
		\Psi(x,y) = \left\{ \begin{array}{r l}
						\max\{x,y\} + f_{-}(|x-y|), & \text{if} \, \, 1 - \Theta = 1, \\
						\min\{x,y\} - f_{+}(|x-y|), & \text{if} \, \, 1 - \Theta = 0.
					\end{array} \right.
	\end{align*}
Thus, in proofs where the events $\{\Theta = 1\}$ and $\{\Theta = 0\}$ would in principle need to be treated independently, by symmetry, there is no loss of generality in assuming $\Theta = 1$.

The symmetry property above has a counterpart at the level of the CDF.  If $\{F_{n}\}$ is the sequence of CDF's associated with a solution $\{X^{(n)}\}$ of the RDE $X^{(n)} = \Phi( X^{(n-1)}_{1}, X^{(n-1)}_{2} )$, then the functions $\{G_{n}\}$ defined by
	\begin{align*}
		G_{n}(x) = 1 - \lim_{\delta \downarrow 0} F_{n}(-x - \delta)
	\end{align*}
are the CDF's associated with the solution $\{ Y^{(n)} \}$ of the RDE $Y^{(n)} = \Psi( Y^{(n-1)}_{1}, Y^{(n-1)}_{2} )$ obtained by setting $Y^{(n)} = - X^{(n)}$.

\subsection{Monotonicity} In this section, the monotonicity of $T$ is proved.  Here \emph{monotone} means that 
	\begin{align*}
		\text{if} \, \, F \leq G \, \, \text{pointwise in} \, \, \mathbb{R}, \quad \text{then} \, \, TF \leq TG \, \, \text{pointwise in} \, \, \mathbb{R}.
	\end{align*}To this end, it is convenient to exploit the quadratic structure inherent in the definition.  Specifically, it will be useful to consider a certain natural operator 
		\begin{equation*}
			S : CDF(\overline{\mathbb{R}}) \times CDF(\overline{\mathbb{R}}) \to CDF(\overline{\mathbb{R}})
		\end{equation*} 
	with the property that $S(F,F) = TF$ for any CDF $F$.

$S$ is defined analogously to $T$: Given two CDF's $F,G \in CDF(\overline{\mathbb{R}})$, let $X_{F}$ and $X_{G}$ be independent extended real-valued random variables distributed according to $F$ and $G$, respectively, and sample the data $(\Theta,f_{+},f_{-})$ independently of $(X_{F},X_{G})$.  Let $S(F,G)$ denote the CDF of the random variable $Y$ defined as follows:
	\begin{align} \label{E: definition of S}
		Y = \Phi(X_{F},X_{G}) = \left\{ \begin{array}{r l} 
			\max\{X_{F},X_{G}\} + f_{+} ( |X_{F} - X_{G}| ), & \text{if} \, \, \Theta = 1, \\
			\min\{X_{F},X_{G}\} - f_{-} ( |X_{F} - X_{G}| ), & \text{if} \, \, \Theta = 0.
		\end{array} \right.
	\end{align}
Clearly $S(F,F) = TF$ as desired.  Further, $S(F,G) = S(G,F)$ by symmetry.

The next result asserts that $S$ is monotone provided $f_{+},f_{-} \in \mathcal{S}$.  

	\begin{prop} \label{P: monotone} Assume that $\mathbf{P} \{ f_{+},f_{-} \in \mathcal{S} \} = 1$. Given $F_{+},F_{-}, G \in CDF(\overline{\mathbb{R}})$, if $F_{+} \leq F_{-}$ holds pointwise in $\mathbb{R}$, then $S(F_{+},G) \leq S(F_{-},G)$ also holds.  Therefore, $T$ is monotone. \end{prop}
	
	\begin{proof} Since $S(F,F) = TF$ and $S$ is symmetric, the desired monotonicity of $T$ readily follows from that of $S$.
	
	It only remains to prove the monotonicity of $S$.  Fix $F_{+},F_{-},G \in CDF(\overline{\mathbb{R}})$ such that $F_{+} \leq F_{-}$ pointwise in $\mathbb{R}$.  Recall that, by Strassen's Theorem, this is equivalent to the existence of a Borel probability measure $\Pi$ on $\overline{\mathbb{R}} \times \overline{\mathbb{R}}$ such that $\Pi \{ (x,y) \, \mid \, x \leq y \}= 1$ and the first and second marginals of $\Pi$ are distributed according to $F_{-}$ and $F_{+}$, respectively.
	
	Let $(X_{1}^{-},X_{1}^{+},X_{2})$ be a random vector such that $(X_{1}^{-},X_{1}^{+})$ is distributed according to $\Pi$, independent of $X_{2}$, and $X_{2}$ is distributed according to $G$.  Note that $X_{1}^{-} \leq X_{2}^{+}$ almost surely.  Define $Y^{+} = \Phi(X_{1}^{+},X_{2})$ and $Y^{-} = \Phi(X_{1}^{-},X_{2})$, where $\Phi$ is independent of $(X_{1}^{-},X_{1}^{+},X_{2})$.  By definition, $S(F_{\rho},G)$ is the CDF of $Y^{\rho}$ for $\rho \in \{ +, - \}$.  Thus, to conclude, it only remains to show that $Y^{-} \leq Y^{+}$ holds almost surely.
	
	This can be checked via a case analysis.  The details for the event $\{\Theta = 1\}$ are provided below; the case when $\{\Theta = 0\}$ follows by symmetry as in Section \ref{S: symmetry}.  Assume henceforth that $\Theta = 1$.
	
	If $X_{2} \leq X_{1}^{-} \leq X_{1}^{+}$, then the assumption that the map $u \mapsto u + f(u)$ is nondecreasing (see \eqref{E: map}) implies
	\begin{align*}
		Y^{-} = X^{-}_{1} + f_{+}(X^{-}_{1} - X_{2}) \leq X_{1}^{+} + f_{+}(X^{+}_{1} - X_{2}) = Y^{+}.
	\end{align*}
On the other hand, if $X^{-}_{1} \leq X_{2} \leq X^{+}_{1}$, then invoking first the fact that $f_{+}$ is nonincreasing (see \eqref{E: nonincreasing}) and then \eqref{E: map}, one finds
	\begin{align*}
		Y^{-} = X_{2} + f_{+}(X_{2} - X^{-}_{1}) &\leq X_{2} + f_{+}(0) \\
			&= X_{2} + f_{+}(X_{2} - X_{2}) \leq X^{+}_{1} + f_{+}(X^{+}_{1} - X_{2}) = Y^{+}.
	\end{align*}
Finally, if $X^{-}_{1} \leq X^{+}_{1} \leq X_{2}$, then, again, the fact that $f_{+}$ is nonincreasing implies
	\begin{align*}
		Y^{-} = X_{2} + f_{+}(X_{2} - X^{-}_{1}) \leq X_{2} + f_{+}(X_{2} - X^{+}_{1}) = Y^{+}.
	\end{align*}
\end{proof}

Using coupling arguments very similar to the one employed above, one readily deduces that the map $T = T(p,\mathbf{P})$ is nonincreasing with respect to $p$ and $f_{+}$ and nondecreasing with respect to $f_{-}$.  The result is stated next for the sake of precision.

	\begin{prop} \label{P: monotone in data} (i) For any $p_{1} \leq p_{2}$, $T(p_{1},\mathbf{P}) F \geq T(p_{2},\mathbf{P})F$ for all $F \in CDF(\overline{\mathbb{R}})$. 
	
	(ii) If $\mathbf{P}$ and $\mathbf{\tilde{P}}$ are measures on $BC([0,\infty) \times BC([0,\infty)$ for which it is possible to find a coupling of the corresponding random variables $(f_{+},f_{-})$ and $(\tilde{f}_{+},\tilde{f}_{-})$ such that $f_{+} \leq \tilde{f}_{+}$ and $f_{-} \geq \tilde{f}_{-}$, then $T(p,\mathbf{P}) F \geq T(p,\mathbf{\tilde{P}})F$ for each $F \in CDF(\overline{\mathbb{R}})$. \end{prop}

%
%

\subsection{Continuity} For the sake of developing intuition, it may be useful to note that $T$ maps the space of continuous CDF's to itself.

	\begin{prop} Assume that $\mathbf{P} \{ f_{+},f_{-} \in \mathcal{S} \} = 1$.  If $F \in CDF(\overline{\mathbb{R}})$ is continuous, then $TF$ is also continuous. \end{prop}
	
		\begin{proof} Let $(X_{1},X_{2})$ be i.i.d.\ random variables distributed according to $F$ and let $Y = \Phi(X_{1},X_{2})$, where $\Phi$ is sampled independently of $(X_{1},X_{2})$.  Since $TF$ is the CDF of $Y$, to see that it is continuous, it suffices to prove that $\mathbb{P} \{ Y = x_{0} \} = 0$ for each $x_{0} \in \mathbb{R}$.  
		
		Consider the event $\Theta = 1$: Since the law of $(X_{1},X_{2})$ is invariant under permutation of the coordinates, 
			\begin{align*}
				\mathbb{P}\{ Y = x_{0} \, \mid \, \Theta = 1\} &= 2 \int_{-\infty}^{\infty} \mathbb{P} \{ X_{2} - x_{1} + f(X_{2} - x_{1}) = x_{0} - x_{1}, \, \, X_{2} > x_{1} \} \mu(dx_{1}),
			\end{align*}
		where $\mu$ is the law of $X_{1}$.  Since $u \mapsto u + f(u)$ is nondecreasing by \eqref{E: right inverse}, there are two possibilities: Either the set $\{ u \geq 0 \, \mid \, u + f(u) = x_{0} - x_{1}\}$ is an interval, or it is a point.  If it is a point, then the corresponding probability in the integral above vanishes since $X_{2}$ is a continuous random variable.  On the other hand, there are at most countably many disjoint intervals on which $u + f(u)$ is constant, hence countably many values of $x_{0} - x_{1}$.  Since $\mu$ has no atoms, these values of $x_{0} - x_{1}$ do not contribute to the integral either.  Therefore, the probability above is zero.  
		
		A similar analysis shows that $\mathbb{P} \{ Y = x_{0} \, \mid \, \Theta = 0\} = 0$.     \end{proof}




\section{Evolution of the CDF} \label{S: equation}

This section concerns the evolution equation \eqref{E: KPP-like equation} satisfied by CDF's.  To this end, it will be useful to decompose the random variable $Y$ in the definition of $T$ (see \eqref{E: T definition}) in a specific way.  Given $F \in CDF(\overline{\mathbb{R}})$, let $X_{1},X_{2}$ be i.i.d.\ extended real-valued random variables distributed according to $F$, and sample the data $(\Theta,f_{+},f_{-})$ independently of $(X_{1},X_{2})$.  Define the extended real-valued random variable $Z$ by 
	\begin{align} \label{E: eqn for Z}
		Z &= \left\{ \begin{array}{r l}
				\max\{X_{1},X_{2}\}, & \text{if} \, \, \Theta = 1, \\
				\min\{X_{1},X_{2}\}, & \text{if} \, \, \Theta = 0.
			\end{array} \right.
	\end{align}
so that the random variable $Y$ of \eqref{E: T definition} can be written in the form
	\begin{align} \label{E: eqn for Y}
		Y &=  Z + \Theta f_{+} ( |X_{2} - X_{1}|) + (1 - \Theta) f_{-}(|X_{2} - X_{1}|).
	\end{align}

The next result describes the difference $TF - F$  in terms of $(Y,Z,\Theta)$.  To that end, let $\mathscr{L}F$ be the function defined by
		\begin{align} \label{E: definition of L}
			\mathscr{L}F(x) = - p \mathbb{P} \{ Z \leq x < Y \mid \Theta = 1 \} + ( 1 - p ) \mathbb{P} \{ Y \leq x < Z \mid \Theta = 0 \}.
		\end{align}

	\begin{prop} \label{P: evolution eqn}   For any $F \in CDF(\overline{\mathbb{R}})$.
		\begin{equation*}
			TF - F = \mathscr{L}F + (1 - 2p) F ( 1 - F).
		\end{equation*}
	\end{prop}
	
Using the explicit expressions \eqref{E: eqn for Z} and \eqref{E: eqn for Y}, one obtains very concrete expressions for the action of $\mathscr{L}$ on smooth CDF's, as stated in the next result.

\begin{prop} \label{P: formula for L} Assume that $\mathbf{P}\{f_{+},f_{-} \in \mathcal{S}\} = \mathbf{P} \{ f_{+}(+\infty) = f_{-}(+\infty) = 0 \} = 1$.  If $F \in CDF(\overline{\mathbb{R}})$ is absolutely continuous, then
	\begin{align} \label{E: jump process}
		\mathscr{L}F(x) &= 2 p \int_{-\infty}^{x}  \mathbf{E} [ F(x_{1} + g_{+}(x - x_{1})) - F(x) ] F'(x_{1}) \, dx_{1} \\
			&\quad +  2 ( 1 - p )  \int_{x}^{+\infty} \mathbf{E} [ F(x_{1} - g_{-}(x_{1} - x) ) - F(x) ] F'(x_{1}) \, dx_{1}. \nonumber
	\end{align}
\end{prop}

Propositions \ref{P: evolution eqn} and \ref{P: formula for L} immediately imply two properties of $\mathscr{L}$ that will be useful in the sequel and are stated in the next proposition.

	\begin{prop} \label{P: properties of L} Under the assumption that $\mathbf{P}\{f_{+},f_{-} \in \mathcal{S}\} = \mathbf{P} \{ f_{+}(+\infty) = f_{-}(+\infty) = 0 \} = 1$, the function $\mathscr{L}$ has the following two properties:
		\begin{itemize}
			\item[(i)] $(\mathscr{L}F)(+\infty) = (\mathscr{L}F)(-\infty) = 0$ for each $F \in CDF(\overline{\mathbb{R}})$.
			\item[(ii)] Given $F,G \in CDF(\overline{\mathbb{R}})$, if the difference $F - G$ is a constant function, then $\mathscr{L}F = \mathscr{L}G$.
		\end{itemize}
	In particular, $T$ maps constant functions to constant functions: Its action on constants is determined by the dynamical system 
		\begin{align*}
			T q - q =  (1 - 2p) q ( 1 - q) \quad \text{for each} \, \, q \in [0,1].
		\end{align*}
	\end{prop}
	
As a consequence of property (ii) and the monotonicity of $T$, the following continuity property of $T$ follows, namely,
	\begin{equation*}
		\| TF - TG \|_{\sup} \leq 2 ( 1 - p ) \| F - G \|_{\sup} 
	\end{equation*}
In particular, $T$ is a contraction when $p = 1/2$.  See Appendix \ref{A: contractivity} for the proof.

\subsection{Proof of Proposition \ref{P: evolution eqn}} \label{S: rep formula} Fix a CDF $F$ and define $(X_{1},X_{2},Z,Y,\Theta,f_{+},f_{-})$ as in the discussion preceding the statement of the proposition.  Start by decomposing the difference $TF - F$ as follows
	\begin{align} 
		TF(x) - F(x) &= \mathbb{P} \{ Y \leq x \} - F(x) \nonumber \\
			&= \Big( \mathbb{P}\{Y \leq x\} - \mathbb{P} \{ Z \leq x\} \Big) + \Big( \mathbb{P} \{ Z \leq x \} - F(x) \Big). \label{E: decomp}
	\end{align}
Observe that, since $f_{+}, f_{-} \geq 0$,
	\begin{align*}
		Y - Z \geq 0 \quad \text{if} \, \, \Theta = 1, \quad \text{and} \quad Y - Z  \leq 0 \quad \text{if} \, \, \Theta = 0.
	\end{align*}
Thus, since $p = \mathbb{P} \{ \Theta = 1\}$ by definition,
	\begin{align*}
		\mathbb{P}\{Y \leq x\} - \mathbb{P}\{Z \leq x\} &= -p \mathbb{P} \{ Z \leq x < Y \mid \Theta = 1\} + (1 - p) \mathbb{P} \{ Y \leq x < Z \mid \Theta = 0\} \\
			&= \mathscr{L}F(x).
	\end{align*}
	
It remains to treat the remaining summand, $\mathbb{P}\{Z \leq x\} - F(x)$.  Notice that 
	\begin{align*}
		\{ Z \leq x \} = \{ \Theta = 1, X_{1} \leq x \, \, \text{and} \, \, X_{2} \leq x \} \cup \{ \Theta = 0, X_{1} \leq x \, \, \text{or} \, \, X_{2} \leq x \}.
	\end{align*}
Thus, since $X_{1}$ and $X_{2}$ are i.i.d.\ with CDF $F$, one obtains, exactly as in \cite[Lemma 1.1]{hambly_jordan}, 
	\begin{align*}
		\mathbb{P} \{ Z \leq x \} = p F(x)^{2} + (1 - p) ( 1 - (1 - F(x))^{2} ).
	\end{align*}
Upon subtracting $F(x)$ and simplifying, this becomes
	\begin{align*}
		\mathbb{P} \{Z \leq x\} - F(x) = (1 - 2p) F(x) ( 1 - F(x) ).
	\end{align*}

\begin{remark} In \cite[Lemma 1.1]{hambly_jordan}, it is observed that the fixed points of the map $q \mapsto p q^{2} + (1 - p) (1 - (1 - q)^{2} )$ are fundamental in the study of the asymptotics of the series-parallel graph.  Yet, as in the derivation above, a fixed point of this map is nothing but a zero of the polynomial $(1 - 2p ) q ( 1 - q )$, the ``reaction term" appearing in the evolution equation \eqref{E: KPP-like equation}.  \end{remark}

\subsection{Proof of Proposition \ref{P: formula for L}} \label{S: smooth case} Suppose that $F \in CDF(\overline{\mathbb{R}})$ is absolutely continuous and fix $x \in \mathbb{R}$.  To obtain the formula \eqref{E: jump process} for $\mathscr{L}F(x)$ in this case, it suffices to establish that 
	\begin{align*}
		\mathbb{P} \{ Z \leq x < Y \mid \Theta = 1\} = -2 \int_{-\infty}^{x} \mathbf{E} [ F(x_{1} + g_{+}(x - x_{1})) - F(x) ] F'(x_{1}) \, dx_{1}, \\
		\mathbb{P} \{ Y \leq x < Z \mid \Theta = 0\} = 2 \int_{x}^{\infty} \mathbf{E} [ F(x_{1} - g_{-}(x_{1} - x)) - F(x) ] F'(x_{1}) \, dx_{1}.
	\end{align*}
Since the law of $(X_{1},X_{2})$ is invariant under permutation of the coordinates, this is equivalent to
	\begin{align}
		\mathbb{P} \{ Z \leq x < Y, X_{1} \leq X_{2} \mid \Theta = 1\} = -\int_{-\infty}^{x} \mathbf{E} [ F(x_{1} + g_{+}(x - x_{1})) - F(x) ] F'(x_{1}) \, dx_{1}, \nonumber \\
		\mathbb{P} \{Y \leq x < Z, X_{2} \leq X_{1} \mid \Theta = 0\} = \int_{x}^{\infty} \mathbf{E} [ F(x_{1} - g_{-}(x_{1} - x)) - F(x) ] F'(x_{1}) \, dx_{1}. \nonumber
	\end{align}
In particular, in the analysis that follows, only the case where $Z = X_{2}$ needs to be treated.  


\subsubsection{Case: $\Theta = 1$} Condition first on the event $\{\Theta = 1\}$. Before going further, it is worth noting that the assumption that $f_{+}(+\infty) = 0$ implies that $X_{1} > -\infty$ almost surely on the event $\{ Z \leq x < Y, \, \,  X_{1} \leq X_{2}, \, \, \Theta = 1\}$ for $x \in \mathbb{R}$.  This is immediate from the fact that $Y = X_{2} = Z$ on the event $\{X_{1} = -\infty, \, \, \Theta = 1, \, \, f_{+}(+\infty) = 0\}$.  

Subtracting $X_{1}$ from both sides of the equation defining $Y$ while assuming that $X_{2} \geq X_{1}$ leads to
	\begin{align*}
		Y - X_{1} = (X_{2} - X_{1}) + f_{+}(X_{2} - X_{1}).
	\end{align*}
At the same time, notice that $g_{+}$ is defined precisely so that, for any $u > 0$ and $s \geq 0$, 
	\begin{align*}
		g_{+}(s) < u \quad \text{if and only if} \quad s < u + f_{+}(u).
	\end{align*}
Thus,
	\begin{align*}
		&\{ Z \leq x < Y, \, \, X_{1} < X_{2}, \, \, \Theta = 1 \} \\
		&\qquad= \{ -\infty < X_{1} < X_{2} \leq x, \, \, g_{+}(x - X_{1}) < X_{2} - X_{1}, \, \, \Theta = 1  \} \\
			&\qquad= \{ -\infty < X_{1} < x, \, \, X_{1} + g_{+}(x - X_{1}) < X_{2} \leq x, \, \, \Theta = 1 \}.
	\end{align*}
Therefore, since $(\Theta,f_{+},X_{1},X_{2})$ are mutually independent and $\mathbb{P} \{ X_{1} \leq X_{2} \} = \mathbb{P} \{ X_{1} < X_{2} \}$ by the continuity of $F$,
	\begin{align*}
		\mathbb{P} \{ Z \leq x < Y, X_{1} \leq X_{2} \mid \Theta = 1\} &= -\int_{-\infty}^{x} \mathbf{E} [ F(x_{1} + g_{+}(x - x_{1})) - F(x) ] F'(x_{1}) \, dx_{1}.
	\end{align*}

\subsubsection{Case: $\Theta = 0$} The analysis of $\mathbb{P} \{Y \leq x < Z, X_{2} \leq X_{1} \mid \Theta = 0\}$ follows from arguments very similar to those used when $\Theta = 1$.  In fact, the desired formula can be quickly derived from the fact that the transformation $(X_{1},X_{2},Z,Y) \mapsto (-X_{1},-X_{2},-Z,-Y)$ preserves the structure of the problem (see Section \ref{S: symmetry}). \qed

\subsection{Proof of Proposition \ref{P: properties of L}} Here is the proof of (i): Fix $F \in CDF(\overline{\mathbb{R}})$.  Observe that, by continuity of measure,
	\begin{align*}
		\lim_{x \to +\infty} \mathbb{P} \{ Z \leq x < Y \, \mid \, \Theta = 1 \} &= \lim_{x \to + \infty} \left( \mathbb{P} \{ Y > x \, \mid \, \Theta = 1 \} - \mathbb{P} \{ Z > x \, \mid \, \Theta = 1 \} \right) \\
			&= \mathbb{P} \{ Y = +\infty \, \mid \, \Theta = 1 \} - \mathbb{P} \{ Z = +\infty \, \mid \, \Theta = 1\}  = 0
	\end{align*}
since $\{ Y = +\infty \} = \{Z = +\infty\}$ by the boundedness of $f_{+}$ and $f_{-}$.   Similarly, 
	\begin{align*}
		\lim_{x \to -\infty} \mathbb{P} \{ Z \leq x < Y \, \mid \, \Theta = 1 \} = \mathbb{P} \{ Z = - \infty \, \mid \, \Theta = 1 \} - \mathbb{P} \{ Y = - \infty \, \mid \, \Theta = 1 \} = 0
	\end{align*}
as $\{ Y = - \infty \} = \{ Z = - \infty\}$.  By symmetry,
	\begin{align*}
		\lim_{x \to +\infty} \mathbb{P} \{ Y \leq x < Z \, \mid \, \Theta = 0\} = \lim_{x \to -\infty} \mathbb{P} \{ Y \leq x < Z \, \mid \, \Theta = 0 \} = 0.
	\end{align*}
Therefore, by \eqref{E: definition of L}, $(\mathscr{L}F)(+\infty) = (\mathscr{L}F)(-\infty) = 0$ as claimed.

Next, the proof of (ii): Suppose that $F,G \in CDF(\overline{\mathbb{R}})$ differ by a constant, hence $G = F + c$ for some $c \in \mathbb{R}$.  First, if $F$ is smooth, then equation \eqref{E: jump process} applied to $F$ and $F+c$ implies $ \mathscr{L}(F + c) = \mathscr{L}F$.  In general, there is a sequence $\{F_{n}\}$ of smooth CDF's that converges vaguely to $F$ and such that $F_{n}(\pm \infty) = F(\pm \infty)$ for any $n$.  By Proposition \ref{P: vague convergence}, $\{ T F_{n} \}$ and $\{ T(F_{n} + c)\}$ converge vaguely to respective limits $TF$ and $T(F + c)$.  At the same time, since the polynomial $P(q) = q + (1 - 2p) q ( 1 - q )$ is smooth, the functions $\{ P(F_{n}) \}$ converge vaguely to $\{ P(F) \}$.  Therefore, for a dense set of points $x \in \mathbb{R}$,
	\begin{align*}
		T(F + c)(x) - P(F + c)(x) = \lim_{n \to \infty} \mathscr{L}(F_{n} + c)(x) = \lim_{n \to \infty} \mathscr{L}F_{n}(x) = TF(x) - PF(x).
	\end{align*}
Rearranging, this implies $T(F + c) + PF = TF + P(F + c)$ on a dense set.  Since right- and lefthand sides are right-continuous (being sums of such functions), this actually implies they are equal everywhere, and, thus, 
	\begin{align*}
		\mathscr{L}(F + c) = T(F + c) - P(F + c) = TF - PF = \mathscr{L}F.
	\end{align*}

Finally, consider a constant CDF $F \equiv q$ for some $q \in [0,1]$.  Writing $TF = \mathcal{L} F + P ( F )$ as in the first part of the proof, observe that
	\begin{align*}
		(TF)(+ \infty ) = ( \mathcal{L} F )( + \infty ) + P(q) = P(q),
	\end{align*}
and, similarly, $TF( - \infty ) = P(q)$.  Therefore, being nondecreasing, $TF$ must be a constant function, equal to $P(q)$.  Identifying $F$ with $q$, this proves that $T$ maps $q$ to the constant $P(q)$, as claimed. \qed




\section{Examples} \label{S: examples} 

This section describes a few relevant examples and consequences of the main theorems, Theorems \ref{T: main diffusive} and \ref{T: main subdiffusive}.

\subsection{Resistance of the series-parallel graph} \label{S: resistance} As mentioned already in the introduction, if $X^{(n)} = \log R^{(n)}$ is the logarithm of the resistance of the series-parallel graph, then $(X^{(n)})_{n \in \mathbb{N}}$ is a solution of the RDE \eqref{E: main RDE} provided $f_{+} = f_{-} = f$ $\mathbf{P}$-almost surely, where $f$ is given by the formula
	\begin{align*}
		f(u) = \log(1 + e^{-u}).
	\end{align*}
In this case, the right-continuous, right-inverse $g$ of the function $u \mapsto u + f(u)$ can be computed to be
	\begin{align*}
		g(s) = \log ( e^{s} - 1 ) \quad  \text{for each} \, \, s \geq \log 2.
	\end{align*}
It is immediate to check that $f \in \mathcal{S}$ and $s - g(s) = f(g(s))$ for $s \geq \log 2$ so the stronger assumption \eqref{E: decay two} of Theorems \ref{T: main diffusive} and \ref{T: main subdiffusive} is satisfied.

By symmery, $\sigma = 0$ in this setting and $a$ can be explicitly computed to be $a = 2 \zeta(3)$.  See Appendix \ref{A: polylog} for the proof.

The approach of this paper readily implies the following characterization of the asymptotics of the resistance $R^{(n)}$ of the critical series-parallel graph, in which resistors have i.i.d.\ resistances with any given law on $[0,+\infty]$.  The result allows for the possibility that $\mathbb{P} \{ R^{(0)} = 0 \} > 0$ or $\mathbb{P} \{ R^{(0)} = + \infty \} > 0$, that is, any given resistor of the graph could be open or shorted.

	\begin{corollary} \label{C: open or short circuit} Assume that $p = 1/2$ and let $R^{(0)}$ be any random variable taking values in $[0,+\infty]$.  Let $R^{(n)}$ be the resistance of the series-parallel graph at stage $n$, in which each edge has i.i.d.\ resistances with the same law as $R^{(0)}$, and let $q = \mathbb{P} \{ 0 < R^{(0)} < + \infty \}$.  Then $\mathbb{P} \{ 0 < R^{(n)} < + \infty \} = q$ for each $n$ and the law of $R^{(n)}$ conditional on this event has the following limit:
		\begin{align*}
			\lim_{N \to \infty} \mathbb{P} \{ (36 a)^{-\frac{1}{3}} N^{-\frac{1}{3}} \log R^{(N)}  \leq x \, \mid \, 0 < R^{(N)} < + \infty \} = F_{\text{Beta}(2,2)} ( q^{-\frac{1}{3}}x + \frac{1}{2}  ).
		\end{align*}
	Here $F_{\text{Beta}(2,2)}$ is the CDF of the $\text{Beta}(2,2)$ distribution and $a = 2 \zeta(3)$.\end{corollary}

Recall that the $\text{Beta}(2,2)$ distribution is determined by the PDF $6 y ( 1 - y ) \mathbf{1}_{[0,1]}(y)$. 
	
		\begin{proof} Let $X^{(n)} = \log R^{(n)}$.  As in the discussion above, $\{X^{(n)}\}$ is a solution of the RDE \eqref{E: main RDE} with $f_{+}(u) = f_{-}(u) = \log ( 1 + e^{-u})$.  Let $q_{-} = \mathbb{P} \{ X^{(0)} = - \infty \}$ and $q_{+} = \mathbb{P} \{ X^{(0)} = + \infty\}$ so that $1 - q  = q_{+} + q_{-}$.  Notice that 
			\begin{align*}
				\lim_{N \to \infty} \mathbb{P} \{ N^{-\frac{1}{3}} X^{(0)} \leq x \} = q_{-} + q \mathbf{1}_{[0,+\infty)}(x) \quad \text{for each} \, \, x \in \mathbb{R} \setminus \{0\}.
			\end{align*}
		Let $F_{N} : \mathbb{R} \times [0,+\infty) \to [0,1]$ be the rescaled CDF of $\{X^{(n)}\}$ defined as in Theorem \ref{T: main subdiffusive}.  By that theorem (which applies even though $\{X^{(n)}\}$ are extended real-valued, see Remark \ref{R: extended real valued} below), $F_{N} \to F$ locally uniformly in $\mathbb{R} \times (0,+\infty)$ as $N \to +\infty$, where $F$ is the bounded discontinuous viscosity solution of the PDE
			\begin{align*}
				\partial_{t} F - a | \partial_{x} F | \partial_{x}^{2} F = 0 \quad \text{in} \, \, \mathbb{R} \times (0,+\infty), \quad F(x,0) = q_{-} + q \mathbf{1}_{[0,+\infty)}(x).
			\end{align*}
		
	$F$ can be readily related to the solution of the same PDE, but with initial datum $\mathbf{1}_{[0,+\infty)}$ and a different coefficient $a$: Notice that if $G$ is defined by $G(x,t) = q^{-1} ( F(x,t) - q_{-} )$, then $G(x,0) = \mathbf{1}_{[0,+\infty)}(x)$ and
		\begin{align*}
			\partial_{t} G - q a | \partial_{x} G | \partial_{x}^{2} G = 0 \quad \text{in} \, \, \mathbb{R} \times (0,+\infty), \quad G(x,0) = \mathbf{1}_{[0,+\infty)}(x).
		\end{align*}
	As in the proof of Corollary \ref{C: distributional limit} (see Section \ref{S: proof of corollary}) below, $G$ can be explicitly computed to be
		\begin{align*}
			G(x,t) = F_{\text{Beta}(2,2)}( U(x,t) ), \quad \text{where} \quad U(x,t) = \frac{ x }{ (36 q a t )^{\frac{1}{3}} } + \frac{1}{2}.
		\end{align*}
	At the same time, since $\mathbb{P} \{ 0 < R^{(N)} < +\infty\} = q$ and $\mathbb{P} \{ R^{(N)} > 0\} = q_{-}$ for any $N$ by Proposition \ref{P: properties of L}, 
		\begin{align*}
			\lim_{N \to +\infty} \mathbb{P} \{ N^{-\frac{1}{3}} \log R^{(N)}  \leq x \, \mid \, 0 < R^{(N)} < + \infty \} = \lim_{N \to + \infty } \frac{1}{q} ( F_{N}(x,1) - q_{-} ) = G(x,1).
		\end{align*}\end{proof}
	
\subsection{Distance on the series-parallel graph} \label{S: distance} Let $D^{(n)}$ be the distance between the two terminal nodes in the series-parallel graph at stage $n$.  Due to the recursive structure inherent in the graph, $D^{(n)}$ is a solution of the RDE
	\begin{align*}
		D^{(n)} = \left\{ \begin{array}{r l}
						D^{(n-1)}_{1} + D^{(n-1)}_{2}, & \text{with probability} \, \, p, \\
						\min \{ D^{(n-1)}_{1}, D^{(n-1)}_{2} \}, & \text{with probability} \, \, 1 - p.
					\end{array} \right.
	\end{align*}
This RDE also arises in Pemantle's min-plus binary tree, and its asymptotics were analyzed in \cite{auffinger_cable}.  Analogous to the case of the resistance, the logarithm $X^{(n)} = \log D^{(n)}$ gives a solution of the RDE \eqref{E: main RDE} with $\Phi$ given as in \eqref{E: main forcing} provided $f_{+}(u) = \log ( 1 + e^{-u} )$ and $f_{-} = 0$.

A complete analysis of the asymptotics of this RDE in the subcritical case ($p < 1/2$) was given in \cite{hambly_jordan}:  There is a random variable $K$ such that if $\lambda$ denotes  the essential infimum of $D^{(0)}$, then $D^{(n)} \to \lambda K$ as $n \to \infty$.  This implies that $K$ is the unique (up to multiplication by a constant) fixed point of the RDE in this case. 

At the critical point $p = 1/2$, the main result of \cite{auffinger_cable} obtains a $\text{Beta}(2,1)$ limit for $N^{-1/2} \log D^{(N)}$ provided $D^{(0)} = 1$ almost surely.  The next corollary of Theorem \ref{T: main diffusive} describes the asymptotics for an arbitrary initial distribution, allowing for the possibility that the i.i.d.\ edge weights can be zero or $+\infty$.  Notice the contrast with the subcritical case: In the subcritical case, if $D^{(0)}$ can be arbitrarily small with positive probability, then $D^{(n)} \to 0$ as $n \to +\infty$, whereas here, conditionally on the event $\{ 0 < D^{(n)} < +\infty \}$, mass at zero only slows the rate of growth.

	\begin{corollary} \label{C: critical distance} Assume that $p = \frac{1}{2}$ and $\mathbb{P} \{ 0 \leq D^{(0)} \leq +\infty \} = 1$ and let $q = \mathbb{P} \{ 0 < D^{(0)} < +\infty\}$.  Then $\mathbb{P} \{ 0 < D^{(n)} < +\infty \} = q$ for any $n$ and the law of $D^{(n)}$ conditional on this event has the following limit:
		\begin{align*}
			\lim_{N \to \infty} \mathbb{P} \{ N^{-\frac{1}{2}} \log D^{(N)} \leq 2 x \sqrt{ |\sigma| }  \, \mid \, 0 < D^{(N)} < + \infty \} = \left\{ \begin{array}{r l}
						0, & \text{if} \, \, x < 0, \\
						\frac{x^{2}}{q}, & \text{if} \, \, 0 \leq x \leq \sqrt{q}, \\
						1, & \text{otherwise.}
					\end{array} \right.
			\end{align*}
	Here the constant $\sigma = - \frac{1}{2} \zeta(2) = - \frac{\pi^{2}}{12}$.
	\end{corollary}  
	
		\begin{proof}   The claimed value of $\sigma$ is computed in Appendix \ref{A: polylog} and is consistent with \cite{auffinger_cable}.  The limiting behavior of $D^{(N)}$ can be derived following the same strategy as in Corollary \ref{C: open or short circuit} above, once again using the fact that $\mathbb{P} \{ 0 < D^{(n)} < +\infty \} = q$ for any $n$ by Proposition \ref{P: properties of L}.  \end{proof}

\subsection{$\mathbb{Z}$-valued Solutions and Finite-Difference Equations}  \label{S: finite difference} This section discusses the subclass of RDE's described by \eqref{E: main RDE} and \eqref{E: main forcing} in the case when $f_{+}$ and $f_{-}$ take values in the set $\{0,f_{\mathbb{Z}}\}$, where $f_{\mathbb{Z}}(u) = (1 - u)_{+}$.  For instance, if the law of the data is chosen so that $f_{+} = f_{-} = f_{\mathbb{Z}}$ surely, then $X^{(n)}$ takes a jump only if $X^{(n-1)}_{1}$ and $X^{(n-1)}_{2}$ are at most distance one apart, and the size of the jump decays linearly with the distance.

What is nice about this particular case of the RDE is a $\mathbb{Z}$-valued RDE is embedded within it: If the initial data $X^{(0)}$ takes values in $\mathbb{Z}$, then $X^{(n)}$ remains integer-valued for all $n$ since $f_{\mathbb{Z}}$ maps integers to integers.  For this $\mathbb{Z}$-valued RDE, the one-step evolution is described by 
	\begin{align} \label{E: symmetrized cooperative}
		Y \overset{d}{=} \left\{ \begin{array}{r l}
							\max\{ X_{1}, X_{2} \} + \mathbf{1}_{\{0\}} ( | X_{2} - X_{1} | ), & \text{if} \, \, \Theta = 1, \\
							\min \{ X_{1}, X_{2} \} - \mathbf{1}_{\{0\}} ( |X_{2} - X_{1} |), & \text{if} \, \, \Theta = 0.
			\end{array} \right.
	\end{align}
The RDE associated with this equation provides a very natural distribution-dependent random walk on $\mathbb{Z}$, and it has the nice feature that the equations describing the CDF are nothing other than finite-difference discretizations of the quasilinear Fisher-KPP equations discussed in the introduction above.

For the rest of this subsection, assume that the joint law $\mathbf{P}$ of $(f_{+},f_{-})$ is concentrated on the set $\{0,f_{\mathbb{Z}}\} \times \{0,f_{\mathbb{Z}}\}$.  Notice that
	\begin{align*}
		f_{\mathbb{Z}} \in \mathcal{S} \quad \text{and} \quad g_{\mathbb{Z}}(s) = s \quad \text{for each} \, \, s \geq 1,
	\end{align*}
where, as in Section \ref{S: main results}, $g_{\mathbb{Z}}$ denotes the right-continuous, right-inverse of the function $u \mapsto u + f_{\mathbb{Z}}(u)$.  Thus, such a law $\mathbf{P}$ satisfies the assumptions of Theorems \ref{T: main diffusive} and \ref{T: main subdiffusive}.  In particular, as long as $(f_{+},f_{-})$ are not identically zero, either $n^{-\frac{1}{2}} X^{(n)}$ or $n^{-\frac{1}{3}} X^{(n)}$ converges to a nondegenerate limit.

 Let $q_{+} = \mathbf{P}\{f_{+} = f_{\mathbb{Z}}\}$ and $q_{-} = \mathbf{P}\{f_{-}=f_{\mathbb{Z}}\}$ in the discussion that follows. 

For $\mathbb{Z}$-valued solutions of the RDE, there is no loss restricting the CDF to the domain $\mathbb{Z}$.  By invoking the formula \eqref{E: definition of L}, one obtains, in this case,
	\begin{align} 
		F_{n} - F_{n-1} &= - p q_{+} (\nabla_{\mathbb{Z}}^{-} F_{n-1})^{2} + ( 1 - p ) q_{-} ( \nabla_{\mathbb{Z}}^{+} F_{n-1} )^{2} + (1 - 2p) F_{n-1} (1 - F_{n-1}), \label{E: more asymmetric form} 
	\end{align}
where $\nabla_{\mathbb{Z}}^{+}$ and $\nabla_{\mathbb{Z}}^{-}$ are the forward- and backward-difference operators
	\begin{align*}
		\nabla_{\mathbb{Z}}^{+} F(x) = F(x +1) - F(x), \quad \nabla_{\mathbb{Z}}^{-} F(x) = F(x) - F(x-1).
	\end{align*}
Using the identity $\nabla_{\mathbb{Z}}^{+} - \nabla_{\mathbb{Z}}^{-} = \Delta_{\mathbb{Z}}$ for $\Delta_{\mathbb{Z}} F(x) = ( F(x+1) + F(x-1) ) - F(x)$ and simplifying the difference of the two squares, this becomes
	\begin{align} \label{E: finite difference} 
		F_{n} - F_{n-1} &= a \nabla^{\text{sym}}_{\mathbb{Z}} F_{n-1} \Delta_{\mathbb{Z}} F_{n-1} + \sigma (\nabla_{\mathbb{Z}}^{+} F_{n-1} )^{2} + (1 - 2p) F_{n-1} (1 - F_{n-1}).
	\end{align}
where $\nabla^{\text{sym}}_{\mathbb{Z}} = 2^{-1} ( \nabla_{\mathbb{Z}}^{+} + \nabla_{\mathbb{Z}}^{-} )$, $\sigma = (1 - p) q_{-} - p q_{+}$, and $a = 2 p q_{+}$.  

Equation \eqref{E: finite difference} looks like a discrete analogue of the quasilinear Fisher-KPP equation \eqref{E: KPP} discussed in Section \ref{S: KPP}.  Notice that here $a$ or $\sigma$ may vanish, irrespective of whether or not $p = 1/2$.  Since replacing $\mathbb{Z}$ by $\mathbb{R}$ and finite-differences by derivatives does not seem likely to fundamentally change the large-scale asymptotic behavior, equation \eqref{E: finite difference} motivates the general study of \eqref{E: KPP} for arbitrary $\sigma \in \mathbb{R}$ and $a \geq 0$.

It is worth emphasizing that there is no contradiction between the formulae \eqref{E: finite difference} and \eqref{E: jump process} since the former involves the CDF of a $\mathbb{Z}$-valued random variable while the latter describes smooth CDF's.

\subsection{Hipster Random Walks and Cooperative Motion} Inspection of \eqref{E: finite difference} also reveals a connection with the symmetric hipster random walk and two-player cooperative motion, examples considered in \cite{addario-berry_cairns_devroye_kerriou_mitchell} and \cite{addario-berry_beckman_lin-symmetric}, respectively.  For the rest of this section, let $p = 1/2$.

To understand the connection, consider again the random variable $Z$ defined by $Z = \Theta \max\{X_{1},X_{2}\} + (1 - \Theta ) \min\{X_{1},X_{2}\}$.  Recall that $Z$ and $X_{1}$ have the same distribution since $p = 1/2$.  At the same time, in view of the fact that the two coincide when $X_{1} = X_{2}$, a stronger statement is true: $Z$ and $X_{1}$ have the same conditional laws on the event $\{X_{1} \neq X_{2}\}$.  This reasoning also applies if $X_{1}$ is replaced by the more symmetric $\Theta X_{1} + ( 1 - \Theta )X_{2}$.  From this, it follows that the law of the output $Y$ from \eqref{E: symmetrized cooperative} is unchanged if the equation is instead changed to 
	\begin{align} \label{E: symmetric hipster random walk}
		Y \overset{d}{=} \left\{ \begin{array}{r l}
						X_{1} + \mathbf{1}_{\{0\}} ( |X_{2} - X_{1}| ), & \text{if} \, \, \Theta = 1, \\
						X_{2} - \mathbf{1}_{\{0\}} ( |X_{2} - X_{1}|), & \text{if} \, \, \Theta = 0,
					\end{array} \right.
	\end{align}
or
	\begin{align} \label{E: symmetric cooperative motion}
		Y \overset{d}{=} \left\{ \begin{array}{r l}
						X_{1} + \mathbf{1}_{\{0\}} ( | X_{2} - X_{1} | ), & \text{if} \, \, \Theta = 1, \\
						X_{1} - \mathbf{1}_{\{0\}} ( |X_{2} - X_{1} |), & \text{if} \, \, \Theta = 0.
			\end{array} \right.
	\end{align}
These are precisely the one-step updates used to define the symmetric hipster random walk and two-player symmetric cooperative motion, respectively.  Thus, although it is not immediately obvious from inspection of \eqref{E: symmetric hipster random walk} or \eqref{E: symmetric cooperative motion}, those two RDE's belong to the class considered here.

In fact, by this reasoning, it is also possible to draw a connection with the other RDE considered in \cite{addario-berry_cairns_devroye_kerriou_mitchell}, called the totally asymmetric hipster random walk.  This is the RDE in $\mathbb{Z}$ with the one-step update determined by two independent Bernoulli random variables $(\Theta, \Xi)$ via the rule
	\begin{align*}
		Y \overset{d}{=} \left\{ \begin{array}{r l}
			X_{1} + \Xi \mathbf{1}_{\{0\}}(|X_{2} - X_{1}|), & \text{if} \, \, \Theta = 1, \\
			X_{2} + \Xi \mathbf{1}_{\{0\}}(|X_{2} - X_{1}|), & \text{if} \, \, \Theta = 0.
		\end{array} \right.		
	\end{align*} 
where $\mathbb{P}\{\Theta = 1\} = \frac{1}{2}$.  This is a class of RDE's parametrized by $q := \mathbb{P} \{ \Xi = 1\}$.  By the same reasoning employed above, as long as $q \leq \frac{1}{2}$, the totally asymmetric hipster random walk belongs to the subclass of RDE's considered in the previous subsection.  This is made precise in the next result.

	\begin{prop} Assume that $p = \frac{1}{2}$ and $\mathbf{P} \{ f_{+}, f_{-} \in \{0,f_{\mathbb{Z}}\} \} = 1$, and define $(q_{+},q_{-})$ by $q_{\rho} = \mathbf{P}\{f_{\rho} = f_{\mathbb{Z}}\}$.  If $\{ X^{(n)} \}$ is the solution of the RDE \eqref{E: main RDE} with forcing $\Phi$ given by \eqref{E: main forcing} for some integer-valued initial condition $X^{(0)}$, then, for each $n$, $X^{(n)}$ has the same law as (i) the symmetric hipster random walk with the same initial datum if $q_{+} = q_{-} = 1$ and (ii) the totally asymmetric hipster random walk with the same initial datum if $q_{-} = 0$ provided the parameter $q$ is set to $q = \frac{1}{2} q_{+}$.  \end{prop} 
	
Note that the equation $q = \frac{1}{2} q_{+}$ introduces the constraint $q \leq \frac{1}{2}$, so that the previous result only covers the asymmetric hipster random walk for half of the possible values of $q$.  In fact, it is possible to show (by manipulation of the evolution equation satisfied by the CDF) that the asymmetric hipster random walk is a monotone RDE only if $q \leq 1/2$.  Therefore, it falls outside the class of RDE's considered in this paper when $q > 1/2$.

Similar arguments apply to the two-player version of the RDE's analyzed in \cite{addario-berry_beckman_lin-asymmetric}.  Specifically, two-player totally asymmetric, $q$-lazy cooperative motion is the RDE for $\mathbb{Z}$-valued random variables given by
	\begin{align*}
		X^{(n)} \overset{d}{=} \left\{ \begin{array}{r l}
								X_{1}^{(n-1)} + \mathbf{1}_{\{0\}} ( | X_{2}^{(n-1)} - X_{1}^{(n-1)} | ), & \text{with probability} \, \, r q, \\
								X_{1}^{(n-1)} - \mathbf{1}_{\{0\}} ( | X_{2}^{(n-1)} - X_{1}^{(n-1)} | ), & \text{with probability} \, \, (1 - r) q, \\
								X_{1}^{(n-1)}, & \text{with probability} \, \, 1 - q.
							\end{array} \right.
	\end{align*}  
The RDE is parameterized by the pair $(r,q)$.  The next result describes the connection with the RDE's considered here.  As for the asymmetric hipster random walk, only a subset of the parameter values are covered; the proof is left to the interested reader.

\begin{prop} Assume that $p = \frac{1}{2}$ and $\mathbf{P} \{ f_{+}, f_{-} \in \{0,f_{\mathbb{Z}}\} \} = 1$, and define $(q_{+},q_{-})$ by $q_{\rho} = \mathbb{P} \{ f_{\rho} = f_{\mathbb{Z}} \}$. If $\{ X^{(n)} \}$ is the solution of the RDE \eqref{E: main RDE} with forcing $\Phi$ given by \eqref{E: main forcing} for some integer-valued initial condition $X^{(0)}$, then, for each $n$, $X^{(n)}$ has the same law as (i) two-player symmetric cooperative motion with that initial condition if $q_{+} = q_{-} = 1$ or (ii) two-player asymmetric, $q$-lazy cooperative motion with that initial condition if $q_{+} \neq q_{-}$, provided the parameters $(r,q)$ are set to $q = \frac{1}{2} ( q_{+} + q_{-} )$ and $r = \frac{q_{+}}{q_{-} + q_{+}}$. \end{prop}

Notice that when $q_{+} \neq q_{-}$, the range of the map $(q_{+},q_{-}) \mapsto (r,q)$ is
	\begin{align} \label{E: parameter regime}
		\left\{ (r,q) \, \mid \, q \in [0,1], \, \, \left| r - \frac{1}{2} \right| \leq \frac{1}{2} ( q^{-1} - 1) \right\}.
	\end{align}
In general, this is a proper subset of the unit square.  However, for any fixed $q \leq \frac{1}{2}$, it covers all possible values of $r \in [0,1]$.

As in the case of hipster random walks, the restriction in \eqref{E: parameter regime} is fundamental.  Two-player totally asymmetric, $q$-lazy cooperative motion is monotone only for $(r,q)$ in the set given above; this can readily be proved via differentiation of the evolution equation satisfied by the CDF (see \cite[Equation (1.6)]{addario-berry_beckman_lin-asymmetric}). 

In view of the proposition above, the $\text{Beta}(2,2)$ limit theorems for the symmetric RDE's considered in \cite{addario-berry_cairns_devroye_kerriou_mitchell} and \cite{addario-berry_beckman_lin-symmetric} follow as a special case of this paper's main results.  The $\text{Beta}(2,1)$ limit theorems for asymmetric RDE's obtained in \cite{addario-berry_cairns_devroye_kerriou_mitchell} and \cite{addario-berry_beckman_lin-asymmetric} also follow, albeit only in the parameter regimes identified above.   Interestingly, in the asymmetric case, those works were able to prove limit theorems even in the nonmonotone regime.




\section{Scaling Limit} \label{S: scaling limit}

This section proves the main results, Theorems \ref{T: main diffusive} and \ref{T: main subdiffusive}, as well as Corollary \ref{C: distributional limit}.	

Fix data $(f_{+},f_{-})$ satisfying the assumptions of Theorems \ref{T: main diffusive} or \ref{T: main subdiffusive}.  Given $\theta \in \mathbb{R}$, define the sequence $\{p^{(N)}\}$ by
	 \begin{equation*}
	 	p^{(N)} = \frac{1}{2} + \theta N^{-1}.
	\end{equation*}  
Let $T^{(N)}$ be the operator associated with the RDE \eqref{E: main RDE} with $\Phi$ given by \eqref{E: main forcing} and bias parameter $p = p^{(N)}$.  In view of the results of Section \ref{S: equation}, the equation $F^{(N)}_{n} = T^{(N)} F^{(N)}_{n-1}$ can be written in the form
	\begin{align} \label{E: scheme}
		F^{(N)}_{n} - F^{(N)}_{n-1} = \mathcal{L}^{(N)} F^{(N)}_{n-1} - 2 \theta N^{-1} F^{(N)}_{n-1} ( 1 - F^{(N)}_{n-1} )
	\end{align}
for some function $\mathcal{L}^{(N)}$ defined on $CDF(\overline{\mathbb{R}})$.  This equation, together with the properties of $\mathcal{L}^{(N)}$ obtained in Section \ref{S: equation}, motivates the study of scaling limits of certain sequences of monotone operators defined on $CDF(\overline{\mathbb{R}})$ in an abstract setting.

\subsection{Abstract Framework} \label{S: abstract framework} Let $\{T^{(N)}\}$ be a family of operators on $CDF(\overline{\mathbb{R}})$ with the following properties: 
	\begin{itemize}
		\item[(i)] \textit{Monotonicity:} If $F \leq G$ pointwise in $\mathbb{R}$, then $T^{(N)}F\leq T^{(N)}G$ also holds.
		\item[(ii)] \textit{Discrete-Time Reaction-Diffusion Equation:} There are functions $\{ \mathcal{L}^{(N)} \}$ defined on $CDF(\overline{\mathbb{R}})$ and continuous functions $\{ Q^{(N)} \}$ defined on $[0,1]$ such that
			\begin{align*}
				T^{(N)}F - F = \mathcal{L}^{(N)} F - Q^{(N)}(F).
			\end{align*} 
		\item[(iii)] \textit{Commutation with Constants:} Given $F,G \in CDF(\overline{\mathbb{R}})$, if the difference $F - G$ is a constant function, then $\mathcal{L}^{(N)} F = \mathcal{L}^{(N)} G$.
		\item[(iv)] \textit{Action at Infinity:} If $F \in CDF(\overline{\mathbb{R}})$, then, for each $\bar{x} \in \{+\infty,-\infty\}$,
			\begin{align*}
				( T^{(N)} F )( \bar{x} ) = F ( \bar{x} ) - Q^{(N)} ( F ( \bar{x} ) ).
			\end{align*}
	\end{itemize}
Notice that, as in the proof of Proposition \ref{P: properties of L} above, assumption (iv) implies that $\mathcal{L}^{(N)}$ vanishes on constant functions and, thus, $T^{(N)}$ maps constant functions to constant functions via the dynamical system
	\begin{align*}
		T^{(N)} q = q - Q^{(N)}(q) \quad \text{for each} \, \, q \in [0,1].
	\end{align*}
Since $T^{(N)}$ is monotone and maps $CDF(\overline{\mathbb{R}})$ to itself, it follows that the function $q \mapsto q - Q^{(N)}(q)$ is a nondecreasing map that sends $[0,1]$ into itself.

In addition to (i)-(iv) above, assume that there is a sequence of positive numbers $\{ \delta^{(N)} \}$ such that $\delta^{(N)} \to 0$ as $N \to +\infty$ and a continuous function $\mathcal{F} : \mathbb{R} \times \mathbb{R} \to \mathbb{R}$, which is nondecreasing in the second variable, for which the following property holds:
	\begin{itemize}
		\item[(v)] \textit{Consistency of $\mathcal{L}^{(N)}$:} If $G \in CDF(\overline{\mathbb{R}})$ is smooth with bounded, uniformly continuous second derivatives, rescaled according to the formula $G_{\delta}(y) = G(\delta y)$, then
			\begin{align*}
				\lim_{ N \to + \infty } N (\mathscr{L}^{(N)} G_{ \delta^{(N)} } ) \left( \frac{x}{\delta^{(N)}} \right) = \mathcal{F}( \partial_{x} G(x), \partial_{x}^{2} G(x)),
			\end{align*}
		where convergence holds uniformly with respect to $x \in \mathbb{R}$. 
	\end{itemize}
Since $\mathcal{L}^{(N)}$ vanishes on constant functions for any $N$, assumption (v) implies
	\begin{equation} \label{E: operator vanishes at zero}
		\mathcal{F}(0,0) = 0.
	\end{equation}
	
An analogous assumption will be made on the reaction term $Q^{(N)}$, as follows:
	\begin{itemize}
		\item[(vi)] \textit{ Consistency of $Q^{(N)}$: } There is a function $Q : [0,1] \to \mathbb{R}$ such that
			\begin{align*}
				\lim_{N \to \infty} \sup \left\{ | N Q^{(N)}(q) - Q(q) | \, \mid \, q \in [0,1] \right\} = 0
			\end{align*}
		Further, there is a constant $L > 0$ such that, independently of $N$,
			\begin{align*}
				| Q^{(N)}(q) - Q^{(N)}(q') | \leq L N^{-1} |q - q'| \quad \text{for each} \, \, q, q' \in [0,1].
			\end{align*}
	\end{itemize}

Under these assumptions, sequences of solutions $\{ F^{(N)}_{n} \}$ of the recursion $F^{(N)}_{n} = T^{(N)}F^{(N)}_{n-1}$ undergo a scaling limit, whereby their asymptotic behavior is determined by the PDE $\partial_{t} F = \mathcal{F}(\partial_{x} F, \partial_{x}^{2} F ) - Q(F)$.  This is made precise in the next theorem.

	\begin{theorem} \label{T: monotone scheme} Assume that $\{ T^{(N)} \}$ is a family of operators on $CDF(\overline{\mathbb{R}})$ satisfying assumptions (i)-(vi) above.  For each $N$, let $\{ F^{(N)}_{n} \}$ be a sequence in $CDF(\overline{\mathbb{R}})$ generated by the recursion $F_{n}^{(N)} = T^{(N)}F^{(N)}_{n-1}$.  If $F_{N} : \mathbb{R} \times [0,\infty) \to \mathbb{R}$ is the rescaled CDF defined by
		\begin{align*}
			F_{N}(x,t) = F^{(N)}_{[ N t]} \left( \frac{ x }{ \delta^{(N)} } \right),
		\end{align*}
	and if there is a continuous $F_{\text{in}} \in CDF(\overline{\mathbb{R}})$ such that, at time $t = 0$,
		\begin{equation*}
			F_{N}(\cdot,0) \to F_{\text{in}} \quad \text{vaguely as} \, \, N \to +\infty,
		\end{equation*} 
	then
		\begin{align*}
			F_{N} \to F \quad \text{locally uniformly in} \, \, \mathbb{R} \times [0, +\infty) \, \, \text{as} \, \, N \to +\infty,
		\end{align*}
	where $F$ is the unique bounded (continuous) viscosity solution of the equation
		\begin{equation} \label{E: initial value problem}
			\left\{ \begin{array}{r l}
				\partial_{t} F - \mathcal{F}( \partial_{x} F, \partial^{2}_{x} F ) + Q(F) = 0 & \text{in} \, \, \mathbb{R} \times (0, +\infty), \\
				F(x,0) = F_{\text{in}}(x).
			\end{array} \right.
		\end{equation}
	\end{theorem}

The results of the previous sections imply that if $\mathbf{P} \{ f_{+},f_{-} \in \mathcal{S} \} = 1$ and \eqref{E: decay one} holds, as is assumed in Theorems \ref{T: main diffusive} and \ref{T: main subdiffusive}, then the one-parameter family $\{ T^{(N)} \}$ defined in the discussion at the beginning of this section satisfies assumptions (i)-(iv); see Section \ref{S: deduce main theorems} for the details. In particular, $\{ Q^{(N)}\}$ is given by $Q^{(N)}(q) = 2 \theta N^{-1} q ( 1 - q )$, which trivially satisfies (vi).  It remains to verify condition (v).  That is the subject of the next proposition.  Recall that the constants $\sigma$ and $a$ are defined in Section \ref{S: main results} above.

	\begin{prop} \label{P: consistency condition} Assume that $\mathbf{P} \{ f_{+},f_{-} \in \mathcal{S} \} = 1$ and \eqref{E: decay one} holds, and fix $\theta \in \mathbb{R}$.  Define $p^{(N)} = 1/2 + \theta N^{-1}$.  Consider the operator $\mathscr{L}^{(N)}$ given by \eqref{E: definition of L} with $p = p^{(N)}$.  If $G : \mathbb{R} \to [0,1]$ is a nondecreasing function with bounded, uniformly continuous first and second derivatives, rescaled according to the rule $G_{\delta}(y) = G(\delta y)$, and if $R_{G}^{(N)} : \mathbb{R} \to \mathbb{R}$ is the remainder defined so that
		\begin{align*}
			( \mathscr{L}^{(N)} G_{N^{-\frac{1}{2}}} )( N^{\frac{1}{2}} x ) = N^{-1} \sigma | \partial_{x} G ( x ) |^{2} + N^{-1} R_{G}^{(N)}(x),
		\end{align*}
	then $R^{(N)}_{G} \to 0$ uniformly in $\mathbb{R}$ as $N \to +\infty$.
	
	Further, if the stronger assumption \eqref{E: decay two} holds,  and if $R^{(N)}_{G}$ is instead defined by
		\begin{align*}
			( \mathscr{L}^{(N)} G_{N^{-\frac{1}{3}}} )( N^{\frac{1}{3}} x ) = N^{- \frac{2}{3} } \sigma | \partial_{x} G ( x ) |^{2} + N^{ -1 } a | \partial_{x} G ( x ) | \partial_{x}^{2} G ( x ) + N^{ -1 } R_{G}^{(N)}(x),
		\end{align*}
	then, as before, $R^{(N)}_{G} \to 0$ uniformly as $N \to +\infty$.
	\end{prop}

The proof of the proposition is relegated to the end of this section (Section \ref{S: consistency}).

There is one last technicality to address, namely, whereas Theorem \ref{T: monotone scheme} asks for the continuity of the initial datum $F_{\text{in}}$, Theorems \ref{T: main diffusive} and \ref{T: main subdiffusive} allow for discontinuities.  This gap is bridged using the fact that the equations of interest here are well-posed even for discontinuous data.  This is made precise in the next condition:
	\begin{itemize}
		\item[(vii)] \textit{Asymptotically regularizing:} If $\underline{F}$ and $\overline{F}$ denote the maximal and minimal viscosity sub- and supersolutions, respectively, of the initial-value problem \eqref{E: initial value problem} for a given $F_{\text{in}} \in CDF(\overline{\mathbb{R}})$ (see Appendix \ref{A: comparison}), then 
			\begin{align*}
				\underline{F} = \overline{F} \quad \text{in} \, \, \mathbb{R} \times (0,+\infty).
			\end{align*}
	\end{itemize}  
As explained in Appendix \ref{A: viscosity solutions}, from the fact that $\underline{F} = \overline{F}$ in $\mathbb{R} \times (0,+\infty)$, it follows that there is a unique solution of \eqref{E: initial value problem} started from any initial datum $F_{\text{in}} \in CDF(\overline{\mathbb{R}})$, which is continuous for positive times.  
	
The next corollary asserts that if the family $\{T^{(N)}\}$ satisfies assumption (vii) in addition to (i)-(vi), then Theorem \ref{T: monotone scheme} remains true even if $F_{\text{in}}$ has discontinuities.

	\begin{corollary} \label{C: arbitrary convergence} Assume that $\{ T^{(N)}\}$ is a family of operators on $CDF(\overline{\mathbb{R}})$ satisfying assumptions (i)-(vii) above.  Then the statement of Theorem \ref{T: monotone scheme} remains true for an arbitrary $F_{\text{in}} \in CDF(\overline{\mathbb{R}})$ provided $F$ is taken to be the unique bounded discontinuous viscosity solution of the initial-value problem \eqref{E: initial value problem}.  \end{corollary}
	
	See Appendix \ref{A: monotone schemes} for the proof.  
	
	It follows from classical results that, in the PDE's of interest here, assumption (vii) holds as long as either $\sigma \neq 0$ or $a > 0$.  (The degenerate PDE $\partial_{t} F + Q(F) = 0$ does \emph{not} have this property and therefore requires a slightly different treatment.)  This is made precise in the next proposition.
	
		\begin{prop} \label{P: uniqueness for the nonlinearities} Let $Q$ be an arbitrary function for which there is an $L > 0$ such that $|Q(q) - Q(q')| \leq L | q - q'|$ for each $q,q' \in [0,1]$ and satisfying both $Q(0) \leq 0$ and $Q(1) \geq 0$.  If $\mathcal{F}(v,w) = \sigma |v|^{2}$ for some $\sigma \neq 0$ or $\mathcal{F}(v,w) = a |v| w$ for some $a > 0$, then assumption (vii) is satisfied. \end{prop}
		
			\begin{proof} When $\theta = 0$, this is established in Propositions \ref{P: uniqueness HJ} and \ref{P: uniqueness PME} in Appendix \ref{A: viscosity solutions}.  The improvement to the case $\theta \neq 0$ is covered by Proposition \ref{P: unique with reaction}. \end{proof}
	
	\begin{remark} \label{R: left continuous} Theorem \ref{T: monotone scheme} is stated for semigroups defined in $CDF(\overline{\mathbb{R}})$, the space of nondecreasing, \emph{right-continuous} functions taking values in $[0,1]$.  There is nothing special about right-continuity, however.  In fact, let $F_{-}$ and $F^{+}$ denote the left- and right-continuous versions of a nondecreasing function $F$.  If operators $\{T^{(N)}_{-}\}$ are defined on the space $CDF_{\text{left}}(\overline{\mathbb{R}})$ of nondecreasing, left-continuous functions taking values in $[0,1]$ by the formula $T^{(N)}_{-}F = (T^{(N)}(F^{+}))_{-}$, then $\{T^{(N)}_{-}\}$ satisfies assumptions (i)-(vi) if and only if $\{T^{(N)}\}$ does.  Further, it is not hard to see that a PDE scaling limit of $\{T^{(N)}\}$ holds as in Theorem \ref{T: monotone scheme} if and only if the same conclusion applies to $\{T^{(N)}_{-}\}$.     \end{remark}
	
\subsection{Proof of Theorems \ref{T: main diffusive} and \ref{T: main subdiffusive}} \label{S: deduce main theorems} Both theorems are direct applications of Theorem \ref{T: monotone scheme}, or rather its corollary, Corollary \ref{C: arbitrary convergence}.  

As in the discussion at the beginning of this section, let $T^{(N)}$ be the operator associated with the RDE \eqref{E: main RDE} with data $(f_{+},f_{-})$ satisfying $\mathbf{P} \{ f_{+},f_{-} \in \mathcal{S} \} = 1$ and \eqref{E: decay one} and bias parameter $p^{(N)} = 1/2 + \theta N^{-1}$ for some fixed constant $\theta \in \mathbb{R}$.  Assumptions (i) and (ii) follow from Propositions \ref{P: monotone} and \ref{P: evolution eqn}, respectively.  

The assumptions (iii) and (iv) hold by Proposition \ref{P: properties of L}.  To see that this proposition is applicable, it only remains to check that $\mathbf{P} \{ f_{+}(+\infty) = f_{-}(+\infty) = 0 \} = 1$.  Fix $\rho \in \{+,-\}$.  Since $s - g_{\rho}(s) = f_{\rho} ( g_{\rho}(s) )$ for $s \geq f_{\rho}(0)$ and $g_{\rho}(s) \to + \infty$ as $s \to + \infty$, the integral $\int_{ f_{\rho}(0) }^{ + \infty} |g_{\rho}(s) - s| \, ds$ is finite only if $f_{\rho}(+\infty) = 0$.  At the same time, \eqref{E: decay one} implies that this integral is finite almost surely, hence $\mathbf{P} \{ f_{+}(+\infty) = f_{-}(+\infty) = 0\} = 1$ as claimed.  Thus, the hypotheses of Proposition \ref{P: properties of L} are satisfied, and it follows that assumptions (iii) and (iv) both apply to the sequence $\{T^{(N)}\}$.

In the context of Theorem \ref{T: main diffusive} (i.e., assuming both \eqref{E: decay one} and $\sigma \neq 0$), Proposition \ref{P: consistency condition} implies that assumption (v) holds with 
	\begin{align*}
		\delta^{(N)} = N^{- \frac{1}{2} } \quad \text{and} \quad \mathcal{F}(v,w) = \sigma | v |^{2}.
	\end{align*}
According to Proposition \ref{P: uniqueness for the nonlinearities} above, from the fact that $\sigma \neq 0$, assumption (vii) also holds.  Therefore, Corollary \ref{C: arbitrary convergence} applies, establishing convergence of the rescaled CDF $F_{N}$ to the solution of the initial-value problem associated to the PDE $\partial_{t} F - \sigma |\partial_{x} F|^{2} + 2 \theta F ( 1 - F ) = 0$, as desired.

Similarly, in Theorem \ref{T: main subdiffusive} (that is, assuming that \eqref{E: decay two} holds, $\sigma = 0$, and $a > 0$), Proposition \ref{P: consistency condition} instead establishes that assumption (v) holds with
	\begin{align*}
		\delta^{(N)} = N^{- \frac{1}{3} } \quad \text{and} \quad \mathcal{F}(v,w) = a |v| w.
	\end{align*}
Once again, because $a > 0$, Proposition \ref{P: uniqueness for the nonlinearities} asserts that assumption (vii) is satisfied.  Therefore, by Corollary \ref{C: arbitrary convergence}, the rescaled CDF $F_{N}$ instead converges to the solution of the initial-value problem associated to $\partial_{t} F - a |\partial_{x} F| \partial_{x}^{2} F + 2 \theta F ( 1 - F ) = 0$. \qed

\begin{remark} \label{R: extended real valued} Notice that, in the proof of Theorems \ref{T: main diffusive} and \ref{T: main subdiffusive} above, nowhere was it necessary to assume that $\{X^{(N,n)}\}$ was real-valued.  Therefore, since Theorem \ref{T: monotone scheme} and Corollary \ref{C: arbitrary convergence} in particular do not require that $F^{(N)}_{n}(+\infty) = 1$ or $F^{(N)}_{n}(-\infty) = 0$, the conclusions of the theorems remain true if $\{X^{(N,n)}\}$ are extended real-valued. \end{remark}
	
\subsection{Proof of Corollary \ref{C: distributional limit}} \label{S: proof of corollary} It only remains to reinterpret Theorems \ref{T: main diffusive} and \ref{T: main subdiffusive} in terms of the $\text{Beta}(2,1)$ and $\text{Beta}(2,2)$ distributions.  

In what follows, consider a solution $\{ X^{(n)} \}$ of the RDE \eqref{E: main RDE} with $\mathbf{P}\{ f_{+},f_{-} \in \mathcal{S} \} = 1$,  (constant) bias parameter $p = 1/2$, and started from an arbitrary real-valued initial datum $X^{(0)}$.  Recall that assumption \eqref{E: decay one} is in effect.  

\medskip

\noindent\textit{Case (i): Diffusive Asymptotics.}  Let $\{F_{n}\}$ be the sequence of CDF's associated with $\{ X^{(n)} \}$.  Define $F_{N}(x,t) = F_{[Nt]} ( N^{1/2} x )$ as in Theorem \ref{T: main diffusive}.  Since $X^{(0)}$ is real-valued,
	\begin{align*}
		\lim_{N \to \infty} F_{N}(x,0) = \lim_{N \to \infty} \mathbb{P} \{ N^{-1/2} X^{(0)} \leq x \} = \mathbf{1}_{[0,+\infty)}(x) \quad \text{for each} \, \, x \in \mathbb{R} \setminus \{0\}.
	\end{align*}

First, consider the case when $\sigma \neq 0$.  By Theorem \ref{T: main diffusive}, 
	\begin{align} \label{E: convergence}
		\lim_{N \to \infty} \mathbb{P} \{ N^{- \frac{1}{2} } X^{(N)} \leq x \} = F(x,1),
	\end{align}
where $F$ is the unique bounded discontinuous viscosity solution of \eqref{E: initial value problem} with initial datum $F_{\text{in}} = \mathbf{1}_{[0,+\infty)}$.  

If $\sigma < 0$, then $F$ is determined by the formula
	\begin{align*}
		F(x,t) = \frac{ x^{2} }{ 4 |\sigma| t } \quad \text{for each} \, \, x \in [0, 2 \sqrt{ |\sigma| } t^{1/2}],
	\end{align*}
as can be proved directly or by using the Hopf-Lax formula (see Proposition \ref{P: uniqueness HJ} below).  If $\sigma > 0$, then one can instead study the asymptotics of $-X^{(n)}$ reasoning as in Section \ref{S: symmetry}, which has the effect of flipping the sign of $\sigma$.  In any event, \eqref{E: convergence} implies
	\begin{align*}
		\lim_{N \to \infty} \mathbb{P} \{ - \text{sgn}(\sigma) (2 \sqrt{ |\sigma| } )^{-1}  N^{-\frac{1}{2}} X^{(N)} \leq x \} = F( 2 x \sqrt{ |\sigma| }, 1) = \left\{ \begin{array}{r l} 
							0, & \text{if} \, \, x < 0, \\
							x^{2}, & \text{if} \, \, 0 \leq x \leq 1, \\
							1, & \text{otherwise.}
						\end{array} \right.
	\end{align*}
The righthand side is precisely the CDF of the $\text{Beta}(2,1)$ distribution.  Therefore, $2^{-1} N^{-1/2} X^{(N)} \overset{d}{\to} -\text{sgn}(\sigma) \sqrt{ |\sigma| } \text{Beta}(2,1)$.

It only remains to consider the case when $\sigma = 0$.  Fix a continuous function $G_{\text{in}} \in CDF(\overline{\mathbb{R}})$ such that $G_{\text{in}}(x) \leq \mathbf{1}_{[0+\infty)}(x)$ for each $x \in \mathbb{R}$.  Consider the sequence $\{ Y^{(N,n)} \}$ obtained by solving the RDE \eqref{E: main RDE} with an $N$-dependent initial datum $Y^{(N,0)}$ chosen in such a way that, for every $x \in \mathbb{R}$,
	\begin{gather}
		\mathbb{P} \{ Y^{(N,0)} \leq x \} \leq \mathbb{P} \{ X^{(0)} \leq x \}  \quad \text{for each} \, \, N \in \mathbb{N}, \label{E: simple comparison} \\
		\lim_{N \to \infty} \mathbb{P} \{ N^{-1/2} Y^{(N,0)} \leq x \} = G_{\text{in}}(x). \nonumber
	\end{gather}
This is equivalent to asking that $Y^{(N,0)} \overset{d}{=} X^{(0)} + N^{1/2} Z^{(N)}$ for some sequence of nonnegative random variables $\{ Z^{(N)} \}$ with law converging to $G_{\text{in}}$.

Since $G_{\text{in}}$ is continuous, $\sigma = 0$, and $p = 1/2$ independently of $N$, applying Theorem \ref{T: monotone scheme} with $T^{(N)} = T$, $\delta^{(N)} = N^{-1/2}$, $\mathcal{F} = 0$, and $Q^{(N)} = Q = 0$ yields
	\begin{align*}
		\lim_{ N \to \infty } \mathbb{P} \{ N^{-1/2} Y^{(N,N)} \leq x \} = G(x,1),
	\end{align*}
where $G$ is the unique bounded viscosity solution of the equation $\partial_{t} G = 0$ in $\mathbb{R} \times (0,+\infty)$ with initial condition $G(x,0) = G_{\text{in}}(x)$.  This is just the time-independent function $G(x,t) = G_{\text{in}}(x)$.  In particular, by \eqref{E: simple comparison} and monotonicity (Proposition \ref{P: monotone}),
	\begin{align*}
		G_{\text{in}}(x) = \lim_{ N \to \infty } \mathbb{P} \{ N^{-1/2} Y^{(N,N)} \leq x \} \leq \liminf_{N \to \infty} \mathbb{P} \{ N^{-1/2} X^{(N)} \leq x \}.
	\end{align*}
Since $G_{\text{in}}$ was an arbitrary continuous CDF lying below $\mathbf{1}_{[0,+\infty)}$, this implies
	\begin{align*}
		\mathbf{1}_{[0,+\infty)}(x) \leq \liminf_{N \to \infty} \mathbb{P} \{ N^{-1/2} X^{(N)} \leq x \} \quad \text{for each} \, \, x \in \mathbb{R} \setminus \{0\}.
	\end{align*}
	
Approximating the step function $\mathbf{1}_{[0,+\infty)}$ this time from above rather than below, one similarly finds
	\begin{align*}
		 \limsup_{N \to \infty} \mathbb{P} \{ N^{-1/2} X^{(N)} \leq x \} \leq \mathbf{1}_{[0,+\infty)}(x) \quad \text{for each} \, \, x \in \mathbb{R} \setminus \{0\}.
	\end{align*}
Therefore, $N^{-1/2} X^{(N)} \overset{d}{\to} 0$ as $N \to \infty$, and case (i) is complete.

\medskip

\noindent\textit{Case (ii): Subdiffusive Asymptotics.}  Assume the stronger condition \eqref{E: decay one} and that $\sigma = 0$ and $a > 0$.  Since as in the previous step $N^{-1/3} X^{(0)} \to 0$ in distribution as $N \to +\infty$, Theorem \ref{T: main subdiffusive} implies that 
	\begin{align} \label{E: convergence subdiffusive case}
		\lim_{N \to \infty} \mathbb{P} \{ N^{- \frac{1}{3} } X^{(N)} \leq x \} = F(x,1)
	\end{align}
provided $F$ is the unique bounded discontinuous viscosity solution of \eqref{E: initial value problem} with initial datum $F_{\text{in}} = \mathbf{1}_{[0,\infty)}$.  By Proposition \ref{P: uniqueness PME} below, $F(x,t) = \int_{-\infty}^{x} \rho(y,t) \, dy$, where $\rho$ is the unique continuous distributional solution of the porous medium equation $\partial_{t} \rho = \frac{1}{2} a\partial_{x}^{2} ( \rho^{2} )$ with initial datum $\rho(0) = \delta_{0}$.  According to \cite{vazquez}, this solution, called the Barenblatt (or ZKB) solution, is given by
	\begin{align*}
		\rho(x,t) = \frac{1}{ (6 a t)^{ \frac{1}{3} } } \left( \left( \frac{3}{4} \right)^{ \frac{2}{3} } - \frac{ x^{2} }{ (6 a t)^{ \frac{2}{3} } } \right) \mathbf{1}_{[-(9at/2)^{\frac{1}{3}}, (9at/2)^{\frac{1}{3}} ]}(x).
	\end{align*}  
By \eqref{E: convergence subdiffusive case}, for any bounded continuous function $h$, if $Y^{(N)} = (36a)^{ - \frac{1}{3} } N^{-1/3}  X^{(N)} + \frac{1}{2}$, then
	\begin{align*}
		\lim_{N \to \infty} \mathbb{E} [ h(  Y^{(N)} ) ] &= \int_{-(9a/2)^{\frac{1}{3}}}^{(9a/2)^{\frac{1}{3}}} h( (36a)^{ - \frac{1}{3} } x + 2^{-1} ) \rho(x,1) \, dx = 6 \int_{0}^{1} h(y) y ( 1 - y ) \, dy,
	\end{align*}
and the integrand $6 y ( 1 - y ) \mathbf{1}_{[0,1]}(y)$ is nothing but the PDF of the $\text{Beta}(2,2)$ distribution.  This completes the proof in case (ii).

\medskip

\noindent\textit{Case (iii): Triviality when $\sigma = a = 0$.} Finally, suppose that \eqref{E: decay two} holds and $\sigma = a = 0$.  Since $g_{\pm}(s) \leq s$ holds automatically by definition, the identity $a = 0$ implies that $ (3s - g_{\rho}(s) )(s - g_{\rho}(s) ) = 0$ for all $s \geq 0$ and $\rho \in \{+,-\}$ $\mathbf{P}$-almost surely.  In particular,
	\begin{align*}
		\mathbf{P}\{ f_{+} = f_{-} = 0 \} = \mathbf{P} \{ g_{+}(s) = g_{-}(s) = s \, \, \text{for all} \, \, s \geq 0 \} = 1,
	\end{align*} 
and thus $X^{(n)} \overset{d}{=} (1 - \Theta) \max\{X^{(n-1)}_{1},X^{(n-1)}_{2}\} + \Theta \min \{ X^{(n-1)}_{1}, X^{(n-1)}_{2}\} \overset{d}{=} X^{(n-1)}$ for all $n$ (since $\mathbb{P} \{ \Theta = 1 \} = p = 1/2$).  In other words, $\{X^{(n)}\}$ is the constant sequence $X^{(n)} \overset{d}{=} X^{(0)}$, completing the proof in case (iii). \qed


\subsection{Proof of Proposition \ref{P: consistency condition}} \label{S: consistency} It will be convenient to start by fixing some notation.  Throughout this proof, define $h_{+}, h_{-} : [0,+\infty) \to [0,+\infty)$ by 
	\begin{align*}
		h_{\pm}(s) = s - g_{\pm}(s).
	\end{align*}
Notice that $0 \leq h_{\pm}(s) \leq s$ for any $ s \geq 0$ by the definition of $g_{\pm}$ (Section \ref{S: main results}). 

Define coefficients $\sigma_{+}$ and $\sigma_{-}$ by
	\begin{align} \label{E: definition of sigma plus}
		\sigma_{+} = - \int_{0}^{+\infty} \mathbf{E} h_{+}(s) \, ds, \quad \sigma_{-} = \int_{0}^{+\infty} \mathbf{E} h_{-}(s) \, ds,
	\end{align}
which are finite by \eqref{E: decay one}, so that $\sigma = \sigma_{+} + \sigma_{-}$.  In the case of the stronger assumption \eqref{E: decay two}, also define coefficients $a_{+}$ and $a_{-}$ by
	\begin{align*}
		a_{\pm} = \int_{0}^{+\infty} s \mathbf{E} h_{\pm}(s) \, ds + \frac{1}{2} \int_{0}^{+\infty} \mathbf{E} h_{\pm}(s)^{2} \,ds = \frac{1}{2} \int_{0}^{+\infty} \mathbf{E} ( 2s + h_{\pm}(s) ) h_{\pm}(s) \, ds.
	\end{align*}
By definition, $a = a_{+} + a_{-}$.

Assumption \eqref{E: decay one} says that $\int_{0}^{+\infty} \mathbf{E} h_{\pm}(s) \, ds < +\infty$, while assumption \eqref{E: decay two} says that $\int_{0}^{+\infty} (1 + s ) \mathbf{E} h_{\pm}(s) \, ds < +\infty$.  Notice that, in the latter case, the bound $0 \leq h_{\pm}(s) \leq s$ implies, in addition, that
	\begin{align} \label{E: key integral inequality}
		\int_{0}^{+\infty} \mathbf{E} h_{\pm}(s)^{2} \, ds \leq \int_{0}^{+\infty} s \mathbf{E} h_{\pm}(s) \, ds < +\infty.
	\end{align}
In particular, $a_{+}$ and $a_{-}$ are finite.

With these preliminaries out of the way, fix a $G \in CDF(\overline{\mathbb{R}})$ that has bounded, uniformly continuous first and second derivatives, and define $G_{\delta}(y) = G(\delta y)$ for $\delta > 0$.  Let $\omega_{G}$ be the modulus of continuity
	\begin{align*}
		\omega_{G}(\delta) = \sup \left\{ | \partial_{x} G(x) - \partial_{x} G(y) | + | \partial_{x}^{2} G(x) - \partial_{x}^{2} G(y) | \, \mid \, x,y \in \mathbb{R}, \, \, |x - y| \leq \delta \right\}.
	\end{align*}

Recall from Section \ref{S: equation} that it is possible to write 
	\begin{equation} \label{E: decomposition of L}
		\mathscr{L}^{(N)} = 2 p^{(N)} \mathscr{L}_{+} + 2 ( 1 - p^{(N)} ) \mathscr{L}_{-}
	\end{equation}
for some $N$-independent functions $\mathscr{L}_{+}$ and $\mathscr{L}_{-}$ defined on $CDF(\overline{\mathbb{R}})$.  Thus, the scaling behavior of $\mathscr{L}^{(N)}$ is determined by that of $\mathscr{L}_{+}$ and $\mathscr{L}_{-}$.  Further, since the assumption \eqref{E: decay one} implies that $\mathbf{P} \{ f_{+}(+\infty) = f_{-}(+\infty) = 0\} = 1$ (as proved in Section \ref{S: deduce main theorems} above), the formula \eqref{E: jump process} can be invoked to find
	\begin{align*}
		(\mathscr{L}_{+} G_{\delta}) ( \delta^{-1} x ) &= \delta \int_{0}^{+\infty} \mathbf{E} [ G ( x - \delta h_{+}(s) ) - G(x) ] \partial_{x} G ( x - \delta s ) \, ds, \\
		(\mathscr{L}_{-} G_{\delta}) ( \delta^{-1} x ) &= \delta \int_{0}^{+\infty} \mathbf{E} [ G ( x + \delta h_{-}(s) ) - G(x) ] \partial_{x} G ( x + \delta s ) \, ds.
	\end{align*}

In view of \eqref{E: decomposition of L}, it only remains to prove the following two claims: First, under assumption \eqref{E: decay one}, for any $\rho \in \{+,-\}$,
	\begin{align} \label{E: square asymptotics}
		\lim_{ \delta \downarrow 0 } \sup \left\{ \left| \delta^{-2} (\mathscr{L}_{\rho} G_{\delta}) ( \delta^{-1} x ) - \sigma_{\rho} |\partial_{x} G(x)|^{2} \right| \, \mid \, x \in \mathbb{R} \right\} = 0.
	\end{align}
Further, if assumption \eqref{E: decay two} holds, then
	\begin{align} \label{E: cubic asymptotics}
		\lim_{ \delta \downarrow 0 } \sup \left\{ \left| \delta^{-3} (\mathscr{L}_{\rho} G_{\delta}) ( \delta^{-1} x ) - \delta^{-1} \sigma_{\rho} |\partial_{x} G(x)|^{2} - a_{\rho} | \partial_{x} G(x) | \partial_{x}^{2} G(x) \right| \, \mid \, x \in \mathbb{R} \right\} = 0.
	\end{align}

\medskip

\noindent\textit{$\delta^{2}$ Asymptotics.} Assume \eqref{E: decay one} holds and fix a $\rho \in \{+,-\}$.  First, observe that, by the definition \eqref{E: definition of sigma plus} of $\sigma_{\rho}$,
	\begin{align}
		&\delta^{-2} ( \mathscr{L}_{\rho} G_{\delta} ) ( \delta^{-1} x ) - \sigma_{\rho} | \partial_{x} G ( x ) |^{2} \label{E: key identity integral} \\
		&= \mathbf{E} \int_{0}^{+\infty} \left( \frac{1}{\delta h_{\rho}(s)} \{ G ( x - \rho \delta  h_{\rho}(s) ) - G(x) \}  \partial_{x} G(x -  \rho \delta s) + \rho |\partial_{x} G(x)|^{2} \right) h_{\rho}(s) \, ds. \nonumber
	\end{align}
By the assumptions on $G$ and the bound $h_{\rho}(s) \leq s$,
	\begin{align*}
		&\left| \mathbf{E} \int_{0}^{+\infty} \left( \frac{1}{\delta h_{\rho}(s)} \{ G ( x - \rho \delta h_{\rho}(s) ) - G(x) \}  \partial_{x} G(x - \rho \delta s) + \rho |\partial_{x} G(x)|^{2} \right) h_{\rho}(s) \, ds \right| \\
		&\qquad \leq  \| \partial_{x} G \|_{\text{sup}}  \mathbf{E} \int_{0}^{+\infty} \left( \omega_{G} ( \delta h_{\rho}(s) ) + \omega_{G}( \delta s) \right)  h_{\rho}(s) \, ds \\
		&\qquad \leq 2  \| \partial_{x} G \|_{\text{sup}} \int_{0}^{+\infty} \omega_{G}(\delta s) \mathbf{E} h_{\rho}(s) \, ds
	\end{align*}
Since assumption \eqref{E: decay one} implies $\mathbf{E}h_{\rho}$ is integrable, Lebesgue's dominated convergence theorem implies that the righthand side vanishes as $\delta \downarrow 0$.  Therefore, \eqref{E: square asymptotics} follows.

\medskip

\noindent\textit{$\delta^{3}$ Asymptotics.} Finally, assume that assumption \eqref{E: decay two} holds.  In this case, notice that the main error term from the previous step can be rewritten in the form 
	\begin{align*}
		&\mathbf{E} \int_{0}^{+\infty} \left( \frac{1}{\delta h_{\rho}(s)} \{ G ( x - \rho \delta h_{\rho}(s) ) - G(x) \}  \partial_{x} G(x - \rho \delta s) + \rho |\partial_{x} G(x)|^{2} \right) h_{\rho}(s) \, ds \\
		&\quad =  \rho \partial_{x} G(x) \mathbf{E} \int_{0}^{+\infty} ( \partial_{x} G(x) - \partial_{x} G(x - \rho \delta s) ) h_{\rho}(s) \, ds \\
		&\quad \quad + \mathbf{E} \int_{0}^{+ \infty} \left( \frac{1}{\delta h_{\rho}(s)} \{ G ( x - \rho \delta h_{\rho}(s) )  - G(x) \} + \rho \partial_{x} G(x)  \right) \partial_{x} G(x - \rho \delta s) h_{\rho}(s) \, ds.
	\end{align*}
In the first term, one finds
	\begin{align*}
		&\left|  \frac{ \rho }{ \delta} \mathbf{E} \int_{0}^{+\infty} \partial_{x} G(x) ( \partial_{x} G(x) - \partial_{x} G(x - \rho \delta s) ) h_{\rho}(s) \, ds - \partial_{x} G(x) \partial_{x}^{2} G(x) \int_{0}^{+\infty} s \mathbf{E} h_{\rho}(s) \, ds \right| \\
		&\qquad \leq \| \partial_{x} G \|_{\text{sup}} \int_{0}^{+\infty} \omega_{G}(\delta s) s \mathbf{E} h_{\rho}(s) \, ds.
	\end{align*}
Similarly, in the second term, again using $\omega_{G}(\delta h_{\rho}(s)) \leq \omega_{G}(\delta s)$,
	\begin{align*}
		&\left| \frac{1}{\delta} \mathbf{E} \int_{0}^{+ \infty} \left( \frac{1}{\delta h_{\rho}(s)} \{ G ( x - \rho \delta h_{\rho}(s) ) - G(x) \} + \rho \partial_{x} G(x)  \right) \partial_{x} G(x - \rho \delta s) h_{\rho}(s) \, ds \right.\\
		&\qquad \left. - \frac{1}{2} \partial_{x} G(x) \partial_{x}^{2} G(x) \int_{0}^{+\infty} \mathbf{E} h_{\rho}(s)^{2} \, ds \right| \\
		&\qquad \leq \frac{1}{2} \left( \|\partial_{x} G\|_{\text{sup}} + \| \partial_{x}^{2} G \|_{\text{sup}} \right) \int_{0}^{+\infty} \omega_{G}(\delta s) \mathbf{E} h_{\rho}(s)^{2} \, ds.
	\end{align*}
Combining these estimates with the identity \eqref{E: key identity integral} and estimate \eqref{E: key integral inequality}, one finds
	\begin{align*}
		& \sup \left\{ \left| \delta^{-3} ( \mathscr{L}_{\rho} G_{\delta} ) ( \delta^{-1} x ) - \delta^{-1} \sigma_{\rho} | \partial_{x} G ( x ) |^{2} - a_{\rho} |\partial_{x} G(x)| \partial_{x}^{2} G(x) \right| \, \mid \, x \in \mathbb{R} \right\} \\
		&\qquad \leq \left( 1 + \|\partial_{x} G\|_{\text{sup}} + \| \partial_{x}^{2} G \|_{\text{sup}} \right) \int_{0}^{+\infty} \omega_{G}(\delta s) s \mathbf{E} h_{\rho}(s) \, ds.
	\end{align*}
Since the righthand side vanishes as $\delta \downarrow 0$ by dominated convergence, \eqref{E: cubic asymptotics} follows. \qed

\subsection{Counterexample} \label{S: counterexample} Here is a counterexample, showing that $N^{-1/2} X^{(N)}$ can blow up if \eqref{E: decay one} fails to hold.  Consider the law $\mathbf{P}$ under which $(f_{+},f_{-})$ are given by
	\begin{align*}
		f_{+}(u) = ( Z - u )_{+}, \quad f_{-}(u) = 0,
	\end{align*}
where $Z$ is a nonnegative random variable such that $\mathbb{E} Z^{2} = + \infty$.  With this choice of forcing, the RDE \eqref{E: main RDE} with $p = 1/2$ is superdiffusive in the sense that if $\{ X^{(n)} \}$ is a solution started from any real-valued initial datum, then
	\begin{align*}
		N^{-\frac{1}{2}} X^{(N)} \overset{d}{\to} + \infty \quad \text{as} \, \, N \to + \infty.
	\end{align*}

To see this, it is convenient to define $f^{(m)}_{+}$ for $m \in \mathbb{N}$ by 
	\begin{align*}
		f^{(m)}_{+}(u) = ( \max\{ Z, m \} - u )_{+}.
	\end{align*}
Let $\{X^{(m,n)}\}$ be a solution of the RDE \eqref{E: main RDE} started from the same initial distribution as $\{X^{(n)}\}$, but with the forcing $(f_{+},0)$ replaced by $(f_{+}^{(m)},0)$.  Note that the function $g_{+}^{(m)}$ associated to $f_{+}^{(m)}$ via \eqref{E: right inverse} is given by 
	\begin{align*}
		g_{+}^{(m)}(s) = 0 \quad \text{if} \, \, s < \max\{Z,m\}, \quad g_{+}^{(m)}(s) = s, \quad \text{otherwise,}
	\end{align*}
and the function $g_{-}$ associated to $f_{-} = 0$ is $g_{-}(s) = s$.  Therefore, the RDE with $(f^{(m)}_{+},0)$ satisfies \eqref{E: decay one} and the constant $\sigma^{(m)}$ determining the asymptotics is given by
	\begin{align*}
		\sigma^{(m)} = - \mathbf{E} \int_{0}^{ \max\{Z, m \} } s \, ds = - \frac{1}{2} \mathbf{E} \max\{Z^{2}, m^{2} \}.
	\end{align*}
Note, in particular, that $| \sigma^{(m)} | \to + \infty$ as $m \to + \infty$ by the choice of $Z$.

Since $f_{+} \geq f^{(m)}_{+}$, monotonicity with respect to the data $f_{+}$ (Proposition \ref{P: monotone in data}) implies that, for any $x \in \mathbb{R}$,
	\begin{align*}
		\mathbb{P} \{ N^{-1/2} X^{(N)} \leq x \} \leq \mathbb{P} \{ N^{-1/2} X^{(m,N)} \leq x\}.
	\end{align*}
Therefore, by Corollary \ref{C: distributional limit}, 
	\begin{align*}
		\limsup_{N \to \infty} \mathbb{P} \{ N^{-1/2} X^{(m,N)} \leq x \} \leq F_{\text{Beta}(2,1)} ( 2^{-1} | \sigma^{(m)} |^{- \frac{1}{2} }  x  ).
	\end{align*}
where $F_{\text{Beta}(2,1)}$ is the CDF of the $\text{Beta}(2,1)$ random variable.  Sending $m \to +\infty$, this implies
	\begin{align*}
		\limsup_{ N \to \infty } \mathbb{P} \{ N^{-1/2} X^{(N)} \leq x \} \leq F_{\text{Beta}(2,1)}(0) = 0 \quad \text{for each} \, \, x \in \mathbb{R}.
	\end{align*}


\appendix




\section{Viscosity Solutions} \label{A: viscosity solutions}

This section reviews relevant definitions and results from the theory of viscosity solutions.  General references for this material are \cite{user,barles}.  The interest here is in the initial-value problem
	\begin{align} \label{E: initial value problem appendix}
		\left\{ \begin{array}{r l}
			\partial_{t} F - \mathcal{F} ( \partial_{x} F, \partial^{2}_{x} F ) + Q(F) = 0 & \text{in} \, \, \mathbb{R} \times (0,\infty), \\
			F(x,0) = F_{\text{in}}(x),
		\end{array} \right.
	\end{align}
where $F_{\text{in}} \in CDF(\overline{\mathbb{R}})$ and the diffusion term $\mathcal{F} : \mathbb{R} \times \mathbb{R} \to \mathbb{R}$ and the reaction term $Q : \mathbb{R} \to \mathbb{R}$ are continuous functions, which satisfy the assumptions, for some $L \geq 0$,
	\begin{gather}
		\mathcal{F}( v, w_{1} ) \leq \mathcal{F} ( v, w_{2} ) \quad \text{for each} \, \, v,w_{1},w_{2} \in \mathbb{R} \, \, \text{with} \, \, w_{1} < w_{2}, \label{A: elliptic} \\
		|Q ( q ) - Q( q' ) | \leq L | q - q' | \, \, \text{for each} \, \, q,q' \in \mathbb{R}, \label{A: lipschitz} \\
		\mathcal{F}(0,0) = 0, \quad Q(0) \leq 0, \quad \text{and} \quad Q(1) \geq 0. \label{A: wellposed}
	\end{gather}
Since the problem at hand naturally leads to consideration of possibly discontinuous initial data, some care will be needed in defining precisely what is meant by a solution, subsolution, or supersolution.

The next definition recalls what it means for a function to satisfy a partial differential equation in the viscosity sense.  

	\begin{definition} An upper semicontinuous function $F : \mathbb{R} \times (0,+\infty) \to \mathbb{R}$ is said to satisfy the differential inequality $\partial_{t} F - \mathcal{F} ( \partial_{x} F, \partial_{x}^{2} F ) + Q(F) \leq 0$ in the viscosity sense in $\mathbb{R} \times (0,+\infty)$ if it has the following property: If $\varphi$ is a smooth function defined in some open set $U \subseteq \mathbb{R} \times (0,+\infty)$ and if there is a point $(x_{0},t_{0}) \in U$ at which $\varphi$ touches $F$ from above (that is, $F(x,t) \leq \varphi(x,t)$ for all $(x,t) \in U$ with equality when $(x,t) = (x_{0},t_{0})$), then
		\begin{align*}
			\partial_{t}\varphi(x_{0},t_{0}) - \mathcal{F} ( \partial_{x} \varphi(x_{0},t_{0}), \partial_{x}^{2} \varphi(x_{0},t_{0}) ) + Q( \varphi(x_{0},t_{0}) ) \leq 0.
		\end{align*}

	A lower semicontinuous function $G : \mathbb{R} \times (0,+\infty) \to \mathbb{R}$ is said to satisfy the differential inequality $\partial_{t} G - \mathcal{F} ( \partial_{x} G, \partial_{x}^{2} G ) + Q(G) \geq 0$ in the viscosity sense in $\mathbb{R} \times (0,+\infty)$ if it has the following property: If $\varphi$ is a smooth function defined in some open set $U \subseteq \mathbb{R} \times (0,+\infty)$ and if there is a point $(x_{0},t_{0}) \in U$ at which $\varphi$ touches $F$ from below (that is, $F(x,t) \geq \varphi(x,t)$ for all $(x,t) \in U$ with equality when $(x,t) = (x_{0},t_{0})$), then
		\begin{align*}
			\partial_{t} \varphi(x_{0},t_{0}) - \mathcal{F} ( \partial_{x} \varphi(x_{0},t_{0}), \partial_{x}^{2}\varphi(x_{0},t_{0}) ) + Q( \varphi(x_{0},t_{0}) ) \geq 0.
		\end{align*}
		
	A continuous function $F : \mathbb{R} \times (0,+\infty) \to \mathbb{R}$ is said to satisfy the differential equation $\partial_{t} F - \mathcal{F} ( \partial_{x} F, \partial^{2}_{x} F ) + Q(F) = 0$ in the viscosity sense in $\mathbb{R} \times (0,+\infty)$ if it satisfies both $\partial_{t} F - \mathcal{F} ( \partial_{x} F, \partial_{x}^{2} F ) + Q(F) \leq 0$ and $\partial_{t} F - \mathcal{F} ( \partial_{x} F, \partial^{2}_{x} F ) + Q(F) \geq 0$ in the viscosity sense in $\mathbb{R} \times (0,+\infty)$. 
	\end{definition}
	
The next definition specifies what is meant by a solution, subsolution, or supersolution of the initial-value problem \eqref{E: initial value problem appendix}.  

	\begin{definition} An upper semicontinuous function $F : \mathbb{R} \times [0,+\infty) \to \mathbb{R}$ is said to be a \emph{viscosity subsolution} of \eqref{E: initial value problem appendix} if it satisfies $\partial_{t} F - \mathcal{F} ( \partial_{x} F, \partial_{x}^{2} F ) +Q(F)  \leq 0$ in the viscosity sense in $\mathbb{R} \times (0,+\infty)$ and 
		\begin{align} \label{E: initial condition thing}
			\lim_{\delta \downarrow 0} \sup \left\{ F(y,0) - F_{\text{in}}(x) \, \mid \, x,y \in \mathbb{R} \, \, \text{such that} \, \, |y - x| \leq \delta \right\} \leq 0.
		\end{align}
	
	A lower semicontinuous function $G : \mathbb{R} \times [0,+\infty) \to \mathbb{R}$ is said to be a \emph{viscosity supersolution} of \eqref{E: initial value problem appendix} if it satisfies $\partial_{t} G - \mathcal{F} ( \partial_{x} G, \partial^{2}_{x} G ) + Q(G) \geq 0$ in the viscosity sense in $\mathbb{R} \times (0,+\infty)$ and 
		\begin{align*}
			\lim_{\delta \downarrow 0} \inf \left\{ G(y,0) - F_{\text{in}}(x) \, \mid \, x,y \in \mathbb{R} \, \, \text{such that} \, \, |y - x| \leq \delta \right\} \geq 0.
		\end{align*}
	
	A continuous function $F : \mathbb{R} \times [0,+\infty) \to \mathbb{R}$ is said to be a \emph{(continuous) viscosity solution} of \eqref{E: initial value problem appendix} if it is both a viscosity subsolution and a viscosity supersolution. 
	\end{definition}

Since the main results of this work, Theorems \ref{T: main diffusive} and \ref{T: main subdiffusive}, consider the problem \eqref{E: initial value problem appendix} with initial data that may be discontinuous, a different definition of solution is needed.  In order to state it, it is necessary to recall the notions of \emph{upper} and \emph{lower semicontinuous envelopes.}  Specifically, for a function $F : \mathbb{R} \times [0,+\infty) \to \mathbb{R}$, the upper and lower semicontinuous envelopes $F^{*}$ and $F_{*}$ are the functions defined in $\mathbb{R} \times [0,+\infty)$ via the formulae
	\begin{align*}
		F^{*}(x,t) = \lim_{\delta \downarrow 0} \sup \left\{ F(y,s) \, \mid \, y,s \in \mathbb{R} \times [0,\infty) \, \, \text{such that} \, \, |x - y| + |t - s| \leq \delta \right\}, \\
		F_{*}(x,t) = \lim_{\delta \downarrow 0} \inf \left\{ F(y,s) \, \mid \, y,s \in \mathbb{R} \times [0,\infty) \, \, \text{such that} \, \, |x - y| + |t - s| \leq \delta \right\}.
	\end{align*}
Notice that $F^{*}$ and $F_{*}$ are always upper and lower semicontinuous, respectively.

	\begin{definition} A function $F : \mathbb{R} \times [0,+\infty) \to \mathbb{R}$ is said to be a \emph{discontinuous viscosity solution of \eqref{E: initial value problem appendix}} if $\partial_{t} F^{*} - \mathcal{F}( \partial_{x} F^{*}, \partial_{x}^{2} F^{*} ) +Q(F^{*}) \leq 0$ in the viscosity sense in $\mathbb{R} \times (0,+\infty)$; $\partial_{t} F_{*} - \mathcal{F}( \partial_{x} F_{*}, \partial_{x}^{2} F_{*} ) + Q(F_{*}) \geq 0$ in the viscosity sense in $\mathbb{R} \times (0,+\infty)$; and $F^{*}(\cdot,0) \leq (F_{\text{in}})^{*}$ and $F_{*}(\cdot,0) \geq (F_{\text{in}})_{*}$ in $\mathbb{R}$. \end{definition}

\subsection{Comparison Principle} \label{A: comparison} This subsection begins by recalling the comparison principle for bounded sub- and supersolutions of the initial-value problem \eqref{E: initial value problem appendix}, which is standard.  An existence and uniqueness result is then stated in terms of the maximal subsolution and minimal supersolution.

	\begin{theorem}[Comparison Principle] \label{T: comparison} Assume that $\mathcal{F}$ and $Q$ satisfy assumptions \eqref{A: elliptic} and \eqref{A: lipschitz}.  If $F, G : \mathbb{R} \times [0,+\infty) \to \mathbb{R}$ are respectively a bounded upper and a bounded lower semicontinuous function such that $\partial_{t} F - \mathcal{F}( \partial_{x} F, \partial_{x}^{2} F ) + Q(F) \leq 0$ and $\partial_{t} G - \mathcal{F} ( \partial_{x} G, \partial^{2}_{x} G ) + Q(G) \geq 0$ in the viscosity sense in $\mathbb{R} \times (0,+\infty)$ and if $M$ is the constant defined by
		\begin{align*}
			M = \lim_{\delta \downarrow 0} \sup \left\{ ( F(x,0) - G(y,0) )_{+} \, \mid \, x, y \in \mathbb{R} \, \, \text{such that} \, \, |x - y| \leq \delta \right\},
		\end{align*}
	then, for any $t > 0$,
		\begin{align} \label{E: lipschitz bound}
			\sup \left\{ F(x,t) - G(x,t) \, \mid \, x \in \mathbb{R} \right\} \leq M e^{ L t}.
		\end{align}
	
	In particular, if $F$ and $G$ are respectively bounded viscosity sub- and supersolutions of \eqref{E: initial value problem appendix} for some fixed initial datum $F_{\text{in}}$, then $F \leq G$ pointwise in $\mathbb{R} \times [0,+\infty)$. \end{theorem}
	
		\begin{proof} The reader can find a proof of this version of the theorem with $L = 0$, $M= 0$ in \cite{giga_goto_ishii_sato}.  The extension to the case $L > 0$, $M = 0$ is routine (see also \cite[Section 5.2]{barles}).  The case $L > 0$, $M > 0$ can be reduced to the case $L > 0$, $M = 0$ using an analogue of the argument appearing in the proof of Lemma \ref{L: pulling down} below, namely, using that the function $\tilde{F} = F - M e^{Lt}$ satisfies $\partial_{t} \tilde{F} - \mathcal{F} ( \partial_{x} \tilde{F}, \partial_{x}^{2} \tilde{F} ) + Q ( \tilde{F} ) \leq 0$ in $\mathbb{R} \times (0,+\infty)$. \end{proof}
		
		The comparison principle implies, in particular, that there is at most one viscosity solution of \eqref{E: initial value problem appendix} if $F_{\text{in}} \in CDF(\overline{\mathbb{R}})$ is continuous.  It is known that, indeed, a solution exists in this case, and the assumption \eqref{A: wellposed} ensures that, for each $t > 0$, $F(\cdot,t) \in CDF(\overline{\mathbb{R}})$.
		
		\begin{prop} \label{P: perron} Assume that $\mathcal{F}$ and $Q$ satisfy \eqref{A: elliptic}, \eqref{A: lipschitz}, and \eqref{A: wellposed}.  For any continuous $F_{\text{in}} \in CDF(\overline{\mathbb{R}})$, there is a unique bounded (continuous) viscosity solution $F$ of \eqref{E: initial value problem appendix}.  Further, for any $t > 0$, the function $F(\cdot,t)$ is in $CDF(\overline{\mathbb{R}})$.   \end{prop}	
		
			\begin{proof} The existence of a bounded solution follows from Perron's Method, see \cite{user} and the references therein.  By Theorem \ref{T: comparison}, this is the unique bounded solution.  By \eqref{A: wellposed}, since $0 \leq F_{\text{in}} \leq 1$, the constant functions $G(x,t) \equiv 1$ and $G(x,t) \equiv 0$ are respectively viscosity super- and subsolutions of \eqref{E: initial value problem appendix}, which implies that $0 \leq F(x,t) \leq 1$ for any $(x,t)$ by comparison.  Finally, for any $y \in \mathbb{R}$, the function $(x,t) \mapsto F(x + y,t)$ is a viscosity solution of \eqref{E: initial value problem appendix} with the initial datum $x \mapsto F_{\text{in}}(x + y)$.  Thus, since $F_{\text{in}}$ is nondecreasing, the comparison principle implies that $F(x + y,t) \leq F(x,t)$ for all $t \geq 0$ if $y \geq 0$.  Altogether, this proves that $F(\cdot,t) \in CDF(\overline{\mathbb{R}})$ for each $t > 0$.    \end{proof}
		
In general, let $\underline{F}$ and $\overline{F}$ be the maximal sub- and minimal supersolution of \eqref{E: initial value problem appendix}:
	\begin{align*}
		\underline{F}(x,t) &= \sup \left\{ F(x,t) \, \mid \, F \, \, \text{viscosity subsolution of} \, \, \eqref{E: initial value problem appendix} \right\}, \\
		\overline{F}(x,t) &= \inf \left\{ G(x,t) \, \mid \, G \, \, \text{viscosity supersolution of} \, \, \eqref{E: initial value problem appendix} \right\}.
	\end{align*}
Notice that the comparison principle implies that $\underline{F} \leq \overline{F}$ and any viscosity solution (continuous or discontinuous) lies in between the two.  This immediately implies a uniqueness criterion.

	\begin{prop} \label{P: uniqueness criterion} Assume that $\mathcal{F}$ and $Q$ satisfy assumptions \eqref{A: elliptic} and \eqref{A: lipschitz}.  Let $\underline{F}$ and $\overline{F}$ be the maximal subsolution and minimal supersolution of \eqref{E: initial value problem appendix}, respectively, for some given $F_{\text{in}} \in CDF(\overline{\mathbb{R}})$.  If $\underline{F} = \overline{F}$ in $\mathbb{R} \times (0,+\infty)$, then any discontinuous viscosity solution $F$ of  \eqref{E: initial value problem appendix} must also satisfy $F = \underline{F}$ in $\mathbb{R} \times (0,+\infty)$.  In particular, in this case, modulo the freedom to choose the value at discontinuities of $F_{\text{in}}$ at $t = 0$, there is a unique discontinuous viscosity solution of \eqref{E: initial value problem appendix}, which is continuous in $\mathbb{R} \times (0,+\infty)$. \end{prop}
	
		\begin{proof} As in the discussion preceding the statement, if $F$ is any discontinuous viscosity solution, then $\underline{F} \leq F \leq \overline{F}$ in $\mathbb{R} \times [0,+\infty)$ in general.  Therefore, in the present setting, $F = \underline{F} = \overline{F}$ in $ \mathbb{R} \times (0,+\infty) $, implying that there is at most one discontinuous viscosity solution (up to arbitrary choices at $t = 0$ at points where $F_{\text{in}}$ jumps).  
		
		The fact that at least one exists is almost immediate: Since $\underline{F}$ is a supremum of subsolutions, the upper semicontinuous envelope $ \underline{F}^{*} $ satisfies $ \partial_{t} \underline{F}^{*} - \mathcal{ F } ( \partial_{x} \underline{F}^{*} , \partial_{x}^{2} \underline{F}^{*} ) + Q ( \underline{F}^{*} ) \leq 0$ in $\mathbb{R} \times (0,+\infty)$, and solutions of \eqref{E: initial value problem appendix} with initial data larger than $F_{\text{in}}$ can be used to prove  $ \underline{ F }^{*}(\cdot,0) = F_{\text{in}} $ at $ t = 0 $.  Therefore, by comparison, $ \underline{F}^{*} \leq \overline{F} $ in $\mathbb{R} \times [0,+\infty)$.  A similar argument shows that $\partial_{t} \overline{F}_{*} - \mathcal{F} ( \partial_{x} \overline{F}_{*}, \partial^{2}_{x} \overline{F}_{*} ) + Q ( \overline{F}_{*} ) \geq 0$ in $\mathbb{R} \times (0,+\infty) $ and $\overline{F}_{*}(\cdot,0)$ equals the lower semicontinuous envelope of $F_{\text{in}}$, so $\overline{F}_{*} \geq \underline{F}$.  Since $\overline{F} = \underline{F}$ in $\mathbb{R} \times (0,+\infty)$ by assumption and the upper (resp.\ lower) semicontinuous envelope of a function is always larger (resp.\ smaller) than the function itself, this proves $\underline{F}$ is continuous in $\mathbb{R} \times (0,+\infty)$ and a discontinuous viscosity solution of \eqref{E: initial value problem appendix}. \end{proof}
		
The following lemma will be used below to verify the identity $\underline{F} = \overline{F}$.

	\begin{lemma} \label{L: squeeze lemma} Assume that $\mathcal{F}$ and $Q$ satisfy assumptions \eqref{A: elliptic} and \eqref{A: lipschitz} and fix $F_{\text{in}} \in CDF(\overline{\mathbb{R}})$.  Suppose that $\{F^{(N),+}_{\text{in}}\}$ and $\{F^{(N),-}_{\text{in}}\}$ are two sequences of continuous functions in $CDF(\overline{\mathbb{R}})$ such that
		\begin{gather}
			F^{(N-1),-}_{\text{in}} \leq F^{(N),-}_{\text{in}} \leq F_{\text{in}} \leq F^{(N),+}_{\text{in}} \leq F^{(N-1),+}_{\text{in}} \quad \text{pointwise for each} \, \, N, \label{E: lemma construction 1} \\
			\lim_{N \to +\infty} F^{(N),+}_{\text{in}} = \lim_{N \to +\infty} F^{(N),-}_{\text{in}} = F_{\text{in}} \quad \text{vaguely.} \label{E: lemma construction 2}
		\end{gather}
	Let $\{F^{(N),+}\}$ and $\{F^{(N),-}\}$ be the (continuous) viscosity solutions of \eqref{E: initial value problem appendix} with respective initial data $\{F^{(N),+}_{\text{in}}\}$ and $\{F^{(N),-}_{\text{in}}\}$.  If the limit of the two sequences coincides, that is, if 
		\begin{align*}
			F := \lim_{N \to +\infty} F^{(N),+} = \lim_{N \to +\infty} F^{(N),-} \quad \text{pointwise in} \, \, \mathbb{R} \times (0,+\infty),
		\end{align*}
	then the maximal subsolution and minimal supersolution of \eqref{E: initial value problem appendix} with initial datum $F_{\text{in}}$ are equal in $\mathbb{R} \times (0,+\infty)$ and, in particular, they both equal the limit $F$.    \end{lemma}
	
		\begin{proof} This is immediate from the fact that $F^{(N),-} \leq \underline{F} \leq \overline{F} \leq F^{(N),+}$ for any $N$ by comparison. \end{proof}

\subsection{Uniqueness for Nondecreasing Data} In this section, the PDE's of interest here are shown to satisfy the uniqueness criterion of Proposition \ref{P: uniqueness criterion}.  For the Hamilton-Jacobi equations arising in the $\text{Beta}(2,1)$ case, the fact that almost-everywhere-continuous data leads to a unique solution at positive times is known from \cite{chen_su}.  In the case of the second-order equation in the $\text{Beta}(2,2)$ case, the uniqueness of solutions follows, for instance, from the characterization of nonnegative solutions of the porous medium equation.  Short proofs are presented here for the reader's convenience.

\medskip

\noindent\textit{Case: $Q = 0$.} The discussion begins with the classical case when the reaction term $Q = 0$.  

First, consider the case when $\mathcal{F}$ is a first-order operator of the form
	\begin{align*}
		\mathcal{F}(v,w) = -H(v)
	\end{align*}
for some convex or concave function $H : \mathbb{R} \to \mathbb{R}$ exhibiting superlinear growth:
	\begin{align} \label{E: superlinear}
		\lim_{|v| \to \infty} |v|^{-1} | H(v) | = + \infty.
	\end{align} 
	
At this stage, it is worth noting that there is no loss of generality in assuming that $H$ is convex.  Indeed, in general, if $\partial_{t} F + H(\partial_{x} F) = 0$ in the viscosity sense in $\mathbb{R} \times (0,+\infty)$, then the function $G(x,t) = 1 - F(-x,t)$ satisfies
	\begin{align*}
		\partial_{t} G - H( \partial_{x} G )  = 0 \quad \text{in the viscosity sense in} \, \, \mathbb{R} \times (0,+\infty)
	\end{align*}
and $H$ is convex if and only if the function $v \mapsto -H(v)$ is concave.  

For superlinear convex/concave $H$, the Hopf-Lax formula furnishes the unique discontinuous viscosity solution of \eqref{E: initial value problem appendix} when $Q = 0$.

	\begin{prop} \label{P: uniqueness HJ} Let $F_{\text{in}} \in CDF(\overline{\mathbb{R}})$ and assume $Q = 0$.  If $\mathcal{F}(v,w) = -H(v)$ for a convex $H : \mathbb{R} \to \mathbb{R}$ satisfying $H(0) = 0$ and \eqref{E: superlinear}, and if $\underline{F}$ and $\overline{F}$ are the maximal subsolution and minimal supersolution, respectively, of \eqref{E: initial value problem appendix}, then, for any $(x,t) \in \mathbb{R} \times (0,+\infty)$,
		\begin{align} \label{E: hopf lax}
			\underline{F}(x,t) = \overline{F}(x,t) = \inf \left\{ F_{\text{in}}( y ) + t L \left( \frac{x - y}{t} \right) \, \mid \, y \in \mathbb{R} \right\}.
		\end{align}
	where $L$ is the Legendre-Fenchel transform of $H$. 
	
	Furthermore, the following statements hold:
		\begin{itemize}
			\item[(i)] For any $t > 0$, the functions $\underline{F}(\cdot,t)$ and $F_{\text{in}}$ coincide on the set $\{-\infty,+\infty\}$.
			\item[(ii)] If $\{ F^{(n)}_{\text{in}}\}$ is a sequence in $CDF(\overline{\mathbb{R}})$ converging vaguely to $F_{\text{in}}$, and if $\{ \underline{F}^{(n)} \}$ are the corresponding maximal subsolutions of \eqref{E: initial value problem appendix}, then, for any compact set $K \subseteq \mathbb{R} \times (0,+\infty)$,
				\begin{align*}
					\lim_{ n \to \infty } \sup \left\{ | \underline{F}^{(n)}(x,t) - \underline{F}^{(n)}(x,t) | \, \mid \, (x,t) \in K \right\} = 0.
				\end{align*}
		\end{itemize}
	\end{prop}
	
Some regularity of $F_{\text{in}}$ is necessary to ensure that $\underline{F}$ and $\overline{F}$ coincide at positive times.  For instance, $F_{\text{in}} = \mathbf{1}_{\mathbb{Q}}$ provides a counterexample.  An optimal version of the result can be found in \cite{chen_su}.
	
		\begin{proof}  Let $F$ be the function defined in the righthand side of \eqref{E: hopf lax}.  When $F_{\text{in}}$ is continuous, the fact that $F$ is the (continuous) viscosity solution is classical (see, e.g., \cite[Section 10.3.4]{evans}).  Otherwise, let $\{F^{(N),+}_{\text{in}}\}$ and $\{F_{\text{in}}^{(N),-}\}$ be two sequences converging to $F_{\text{in}}$ as in Lemma \ref{L: squeeze lemma}, and let $\{F^{(N),+}\}$ and $\{F^{(N),-}\}$ be the corresponding solutions given by \eqref{E: hopf lax}.  For any point of continuity $y$ of $F_{\text{in}}$, any $x \in \mathbb{R}$, and any $t > 0$,
		\begin{align*}
			F_{\text{in}}(y) + t L \left( \frac{x - y}{t} \right) = \lim_{N \to \infty} F_{\text{in}}^{(N),-}(y) + t L \left( \frac{x-y}{t} \right) \geq \lim_{N\to \infty} F^{(N),+}(x,t).
		\end{align*}
	Therefore, since $F_{\text{in}}$ is right-continuous, this proves $ \lim_{N \to \infty} F^{(N),+} \leq F $ pointwise in $\mathbb{R} \times (0,+\infty)$.  Next, since the superlinearity of $H$ implies that $L$ is superlinear (see \cite[Theorem 3, Section 3.3.2]{evans}), for any $(x,t)$, it is possible to write $F^{(N),-}(x,t) = F^{(N),-}_{\text{in}}(y_{N}) + t L ( ( x - y_{N} ) / t )$ for some bounded sequence $\{ y_{N} \}$.  If $y_{N} \to y$ as $N \to \infty$ (say, along a subsequence), then
			\begin{equation*}
				F_{\text{in}}(y - \delta) = \lim_{N \to \infty} F^{(N),-}(y-\delta) \leq \liminf_{N \to \infty} F^{(N),-}_{\text{in}}(y_{N}) \quad \text{for almost every} \, \, \delta > 0,
			\end{equation*}
		from which it follows that $F (x,t) \leq \lim_{N \to +\infty} F^{(N),-}(x,t)$.  Therefore, by Lemma \ref{L: squeeze lemma}, $F = \underline{F} = \overline{F}$.

	Item (i) follows directly from the representation \eqref{E: hopf lax}, the superlinearity of $L$, and the fact that $0 \leq F_{\text{in}} \leq 1$.  Item (ii) can be proved using the comparison principle, Theorem \ref{T: comparison}, in conjunction with the half-relaxed limit method; the interested reader can compare with the proof of Corollary \ref{C: arbitrary convergence} below. \end{proof}

Next, let $m > 1$ and consider the case of the nonlinearity
	\begin{align} \label{E: porous medium nonlinearity}
		\mathcal{F} ( v , w ) = a |v|^{m - 1} w
	\end{align}
Notice that if $F$ were a smooth strictly increasing solution of $\partial_{t} F - a |\partial_{x} F|^{m-1} \partial_{x}^{2}F = 0$, then the derivative $\rho = \partial_{x} F$ would solve the porous medium equation (PME) $\partial_{t} \rho = \frac{a}{m} \partial_{x}^{2} ( \rho^{m} )$.  To prove the uniqueness of solutions with possibly discontinuous, nondecreasing data, this reasoning will be inverted.

	\begin{prop} \label{P: uniqueness PME} Let $F_{\text{in}} \in CDF(\overline{\mathbb{R}})$ and assume $Q = 0$.  If $\mathcal{F}$ has the form \eqref{E: porous medium nonlinearity} for some $m > 1$, then the maximal subsolution $\underline{F}$ and minimal supersolution $\overline{F}$ of \eqref{E: initial value problem appendix} are given for $(x,t) \in \mathbb{R} \times (0,+\infty)$ by the formula
		\begin{align} \label{E: PME formula}
			\underline{F}(x,t) = \overline{F}(x,t) = F_{\text{in}}(-\infty) + \int_{-\infty}^{x} \rho(y,t) \, dy,
		\end{align}
	where $\rho \geq 0$ is the unique continuous distributional solution of the initial-value problem
		\begin{align} \label{E: PME appendix}
			\left\{ \begin{array}{r l}
				\partial_{t} \rho - \frac{a}{m} \partial_{x}^{2} ( \rho^{m} ) = 0 & \text{in} \, \, \mathbb{R} \times (0,+\infty), \\
				\rho(0) = \partial_{x} F_{\text{in}}.
			\end{array} \right.
		\end{align}

	Furthermore, the following statements hold:
		\begin{itemize}
			\item[(i)] For any $t > 0$, the functions $\underline{F}(\cdot, t)$ and $F_{\text{in}}$ coincide on the set $\{-\infty,+\infty\}$.  
			\item[(ii)] If $\{ F_{\text{in}}^{(n)} \}$ is some sequence in $CDF(\overline{\mathbb{R}})$ converging vaguely to $F_{\text{in}}$, and if $\{ \underline{F}^{(n)} \}$ are the corresponding maximal subsolutions of \eqref{E: initial value problem appendix}, then, for any compact set $K \subseteq \mathbb{R} \times (0,+\infty)$,
			\begin{align*}
				\lim_{n \to + \infty} \sup \left\{ | \underline{F}^{(n)}(x,t) - \underline{F}(x,t) | \, \mid \, (x,t) \in K \right\} = 0.
			\end{align*}
		\end{itemize}
		
	\end{prop} 
	
When $F_{\text{in}}$ is uniformly Lipschitz, the representation formula obtained above is a special case of \cite[Theorem 3.1]{addario-berry_beckman_lin-symmetric}.
	
		\begin{proof} The uniqueness and claimed representation formula for the solution follows from the classical uniqueness theorem for the PME by Dahlberg and Kenig \cite{dahlberg_kenig}.  To see this, first, note that if $\partial_{x} F_{\text{in}}$ is smooth and positive, then it is well-known that the solution $\rho$ of the PME is smooth, hence the antiderivative given by \eqref{E: PME formula} is a smooth solution of \eqref{E: initial value problem appendix}.  It is classical that smooth solutions are viscosity solutions.  It is also classical in this setting that $\int_{-\infty}^{\infty} \rho(y,t) \, dy = F_{\text{in}}(+\infty) - F_{\text{in}}(-\infty)$ for each $t > 0$, as in item (i).
		
		It will be convenient later in the proof to note that, in this one-dimensional setting, since the initial data is a subprobability measure, the pressure $\rho^{m-1}$ associated with any of the smooth solutions above satisfies the following estimates
			\begin{align} \label{E: pressure bound}
				\tau^{ \frac{ m - 1 }{ m + 1 } } |\rho^{m-1}| + \tau^{ \frac{ m }{ m + 1 } } | \partial_{x} ( \rho^{m-1} ) | + \tau^{ \frac{ 2 m }{ m + 1 } } | \partial_{t} ( \rho^{m-1} ) |   \leq C \quad \text{in} \, \, \mathbb{R} \times (\tau,+\infty),
			\end{align}
		where $C > 0$ is a universal constant and $\tau > 0$ is arbitrary (see \cite[Chapter 15]{vazquez}).
		
		If $F_{\text{in}}$ is absolutely continuous with $\partial_{x} F_{\text{in}} \in L^{1}(\mathbb{R})$, then it can be approximated by smooth, strictly increasing functions $\{F_{\text{in}}^{(N)}\}$, with $\{\partial_{x} F_{\text{in}}^{(N)}\}$ converging in $L^{1}$, and the corresponding solutions $\{\rho^{(N)}\}$ of the PME then converge in $L^{\infty}([0,+\infty);L^{1}(\mathbb{R}))$ by the $L^{1}$ contractivity property of the equation (see \cite[Chapter 9]{vazquez}) to the solution $\rho$ started from $\partial_{x} F_{\text{in}}$.  This convergence implies uniform convergence of the antiderivatives: Since the uniform limit of viscosity solutions is a viscosity solution, once again \eqref{E: PME formula} provides the viscosity solution in this case.  Note also that item (i) remains true and the bounds \eqref{E: pressure bound} are preserved in the limit.  
		
		Finally, if $\partial_{x} F_{\text{in}}$ is merely a finite measure, consider sequences of smooth functions $\{F^{(N),+}_{\text{in}}\}$ and $\{F^{(N),-}_{\text{in}}\}$ as in Lemma \ref{L: squeeze lemma}, chosen in such a way that $F^{(N),\pm}_{\text{in}}(\bar{x}) \to F_{\text{in}}(\bar{x})$ as $N \to +\infty$ for each $\bar{x} \in \{-\infty,+\infty\}$.  By the bounds \eqref{E: pressure bound}, the corresponding solutions of the PME $\{\rho^{(N),+}\}$ and $\{\rho^{(N),-}\}$ started from the data $\{ \partial_{x} F^{(N),+}_{\text{in}}\}$ and $\{ \partial_{x} F^{(N),-}\}$ are uniformly bounded and equicontinuous in $\mathbb{R} \times (\tau,+\infty)$ for any $\tau > 0$.  Therefore, the antiderivatives $\{F^{(N),+}\}$ and $\{F^{(N),-}\}$ become equicontinuous after a positive amount of time.  Since they are, respectively, monotone decreasing and monotone increasing in $N$, the limits $F^{+} = \lim_{N \to \infty} F^{(N),+}$ and $F^{-} = \lim_{N \to \infty} F^{(N),-}$ exist.  Since $F^{(N),-} \leq F^{-} \leq F^{+} \leq F^{(N),+}$ holds in $\mathbb{R} \times [0,+\infty)$ for any $N$, $F^{-}(\cdot,t)$ and $F^{+}(\cdot,t)$ both converge vaguely to $F_{\text{in}}$ as $t \downarrow 0$.  By Arzel\`{a}-Ascoli, the derivatives $\partial_{x} F^{-}$ and $\partial_{x}F^{+}$ are continuous distributional solutions of the PME in $\mathbb{R} \times (0,+\infty)$, and the vague convergence of $F^{-}$ and $F^{+}$ at time $t = 0$ implies that the initial traces of $\partial_{x} F^{-}$ and $\partial_{x}F^{+}$ both equal $\partial_{x} F_{\text{in}}$.  Therefore, by Dahlberg-Kenig, $\partial_{x}F^{-} = \partial_{x}F^{+}$ in $\mathbb{R} \times (0,+\infty)$, which implies the difference $F^{+} - F^{-}$ depends only on time.
		
		Since item (i) holds for regular solutions, and, in the construction above, $F^{(N),\pm}_{\text{in}}(\bar{x}) \to F_{\text{in}}(\bar{x})$ for each $\bar{x} \in \{-\infty,+\infty\}$, it follows that $F^{+}(\bar{x},t) = F^{-}(\bar{x},t) = F_{\text{in}}(\bar{x})$ also holds for any $t$.  Therefore, $F^{+} = F^{-}$, and so \eqref{E: PME formula} follows by Lemma \ref{L: squeeze lemma} and item (i) holds.
		
		Item (ii) follows as in the previous proposition.\end{proof}
		
\medskip

\noindent\textit{Case $Q \neq 0$.} Next, using only the comparison principle, uniqueness in the case $Q = 0$ is upgraded to the case when $Q$ is nonzero.

	\begin{prop} \label{P: unique with reaction} The identity $\underline{F} = \overline{F}$ in $\mathbb{R} \times (0,+\infty)$ still holds for the initial-value problem \eqref{E: initial value problem appendix} with arbitrary data in $CDF(\overline{\mathbb{R}})$ if $\mathcal{F}$ satisfies the assumptions of Propositions \ref{P: uniqueness HJ} or \ref{P: uniqueness PME} and $Q$ satisfies \eqref{A: lipschitz} and \eqref{A: wellposed}. \end{prop} 
	
		\begin{proof} Fix $F_{\text{in}} \in CDF(\overline{\mathbb{R}})$.  If $F_{\text{in}}$ is continuous, then Proposition \ref{P: perron} implies there is a unique viscosity solution $F$ of \eqref{E: initial value problem appendix} and $\underline{F} = \overline{F} = F$.  Otherwise, if it has jumps, fix sequences $\{ F_{\text{in}}^{(N),+} \}$ and $\{ F_{\text{in}}^{(N),-} \}$ of continuous functions in $CDF(\overline{\mathbb{R}})$ satisfying \eqref{E: lemma construction 1} and \eqref{E: lemma construction 2} as in Lemma \ref{L: squeeze lemma} and, in addition,
			\begin{gather} \label{E: convergence at infinity appendix part}
				\lim_{N \to \infty} F_{\text{in}}^{(N),-}( \pm \infty) = \lim_{N \to \infty} F_{\text{in}}^{(N),+}( \pm \infty) = F_{\text{in}}( \pm \infty).
			\end{gather}
		For each $N$, let $F^{(N),+}$ and $F^{(N),-}$ denote the unique viscosity solutions of \eqref{E: initial value problem appendix} with initial data $F_{\text{in}}^{(N),+}$ and $F_{\text{in}}^{(N),-}$, respectively.
		
		By construction and the comparison principle, the bounds $F^{(N),-} \leq \underline{F} \leq \overline{F} \leq F^{(N),+}$  hold pointwise in $\mathbb{R} \times [0,+\infty)$.  Therefore, it only remains to show that, for any $t > 0$, 
			\begin{align} \label{E: final uniqueness argument}
				\lim_{N \to \infty} \sup \left\{ F^{(N),+}(x,t) - F^{(N),-}(x,t)  \, \mid \, x \in \mathbb{R} \right\} = 0.
			\end{align}
		To this end, since $Q$ is bounded in $[0,1]$, notice that the modified functions $\tilde{F}^{(N),+}$ and $\tilde{F}^{(N),-}$ given by 
			\begin{align*}
				\tilde{F}^{(N),\pm}(x,t) = F^{(N),\pm}(x,t) \mp C(Q) t, \quad \text{where} \quad C(Q) = \sup \left\{ |Q(q)| \, \mid \, 0 \leq q \leq 1 \right\},
			\end{align*}
		are, respectively, sub- and supersolutions of the PDE $\partial_{t} F - \mathcal{F} ( \partial_{x} F, \partial^{2}_{x} F ) = 0$ with $Q$ removed.  Therefore, if $G^{(N),+}$ and $G^{(N),-}$ are the viscosity solutions of this PDE corresponding to initial data $F^{(N),+}_{\text{in}}$ and $F^{(N),-}_{\text{in}}$, respectively, then the comparison principle implies
			\begin{align} \label{E: key comparison}
				F^{(N),+} - C(Q) t = \tilde{F}^{(N),+} \leq G^{(N),+} \quad \text{and} \quad G^{(N),-} \leq \tilde{F}^{(N),-} = F^{(N),-} + C(Q) t,
			\end{align}
		where all the inequalities above hold pointwise in $\mathbb{R} \times [0,+\infty)$.  On the other hand, by Propositions \ref{P: uniqueness HJ} or \ref{P: uniqueness PME}, if $G$ is the discontinuous viscosity solution of the PDE $\partial_{t} F - \mathcal{F} ( \partial_{x} F, \partial^{2}_{x} F ) = 0$ with initial datum $F_{\text{in}}$, then, for any $\delta > 0$, the functions $G^{(N),+}(\cdot,\delta)$ and $G^{(N),-}(\cdot,\delta)$ converge pointwise in $\mathbb{R}$ to $G(\cdot,\delta)$.  Since the constants $G^{(N),+}(\pm \infty,\delta)$ and $G^{(N),-}(\pm \infty, \delta)$ also converge to $G(\pm \infty, \delta)$ by \eqref{E: convergence at infinity appendix part} and item (i) in the aforementioned propositions, Dini's Theorem (applied to the compact metric space $\mathbb{R} \cup \{-\infty,+\infty\}$) implies that this convergence is actually uniform:
			\begin{align*}
				\lim_{N \to \infty} \sup \left\{ G^{(N),+}(x,\delta) - G^{(N),-}(x,\delta)  \, \mid \, x \in \mathbb{R} \right\} = 0.
			\end{align*}   
		Therefore, by \eqref{E: key comparison},
			\begin{align} \label{E: key observation}
				\lim_{N \to \infty} \sup \left\{ F^{(N),+}(x,\delta) - F^{(N),-}(x,\delta) \, \mid \, x \in \mathbb{R} \right\} \leq 2C(Q) \delta.
			\end{align}
			
		Finally, moving back to the original problem (i.e., with the reaction term $Q$), the estimate \eqref{E: lipschitz bound} in Theorem \ref{T: comparison} implies that, for any $t > \delta$,
			\begin{align*}
				\limsup_{N \to \infty} \sup \left\{ F^{(N),+}(x,t) - F^{(N),-}(x,t) \, \mid \, x \in \mathbb{R} \right\} \leq 2C(Q) \delta e^{ L (t - \delta)}.
			\end{align*}
		After sending $\delta \downarrow 0$, this becomes \eqref{E: final uniqueness argument}, as desired.
		\end{proof}
		
\subsection{Viscosity Solutions in $CDF(\overline{\mathbb{R}})$} Recall that throughout this work, the functions of interest belong to the space $CDF(\overline{\mathbb{R}})$ of CDF's of extended real-valued random variables.  (Recall that this is nothing other than the space of nondecreasing, right-continuous functions taking values in $[0,1]$.)  This section covers some specific simplifications of viscosity solutions theory that are applicable as long as the sub- or supersolutions in question are nondecreasing in the spatial variable.  Specifically, it shows that, in this case, there is no loss of generality assuming that the test functions involved are both globally defined and nondecreasing in the spatial variable.

The first lemma concerns the case when a subsolution is touched above by a test function that is increasing close to the touching point.

\begin{lemma} \label{L: test function increasing} Fix $T,A > 0$ and let $F : \mathbb{R} \times (0,T) \to [0,A)$ be an upper semicontinuous function that is nondecreasing in the spatial variable.  Fix $(x_{0},t_{0}) \in \mathbb{R} \times (0,+\infty)$ and let $\varphi$ be a smooth function defined in some neighborhood of $(x_{0},t_{0})$ such that $\varphi(x_{0},t_{0}) < A$ and $\partial_{x} \varphi(x_{0},t_{0}) > 0$.  If $\varphi$ touches $F$ from above at $(x_{0},t_{0})$, then there is an $r > 0$, a nondecreasing $C^{2}$ function $\psi_{1} : \mathbb{R} \to [0,1]$ with $\psi_{1}(x) = \psi_{1}(x_{0} - r)$ for each $x \leq x_{0} - r$ and $\psi_{1}(x) = A$ for each $x \geq x_{0} + r$, and a $C^{2}$ function $\psi_{2} : \mathbb{R} \to \mathbb{R}$ with $\psi_{2}(t_{0}) = 0$ such that if $\psi$ is the function
	\begin{align*}
		\psi(x,t) = \psi_{1}(x) + \psi_{2}(t),
	\end{align*}
then
	\begin{align*}
		F(x,t) \leq \psi(x,t) \quad \text{for each} \, \, (x,t) \in \mathbb{R} \times [t_{0}-r,t_{0}+r]
	\end{align*}
and equality holds if and only if $(x,t) = (x_{0},t_{0})$.  Further, for any $\delta > 0$, it is possible to construct $\psi$ in such a way that 
	\begin{gather*}
		\psi(x_{0},t_{0}) = \varphi(x_{0},t_{0}), \quad \partial_{x} \psi(x_{0},t_{0}) = \partial_{x} \varphi(x_{0},t_{0}), \quad \partial_{t} \varphi(x_{0},t_{0}) = \partial_{t} \varphi(x_{0},t_{0}), \\
		\left| \partial_{x}^{2} \psi(x_{0},t_{0}) - \partial_{x}^{2} \varphi(x_{0},t_{0}) \right| \leq \delta.
	\end{gather*}
\end{lemma}

	\begin{proof} Fix $\delta > 0$.  By Taylor's theorem, there is an $r > 0$ and second-degree polynomials $R$ and $S$ such that $\partial^{m}_{x} R(x_{0}) = \partial_{x}^{m} \varphi(x_{0},t_{0})$ for $m \in \{0,1\}$, $S(0) = 0$, $\partial_{t}S(0) = \partial_{t}\varphi(x_{0},t_{0})$, and 
		\begin{align*}
			F(x,t) \leq \varphi(x,t) \leq R(x)+ S(t - t_{0}) \quad \text{for each} \, \, (x,t) \in [x_{0}-r,x_{0}+r] \times [t_{0}-r,t_{0}+r].
		\end{align*}
	Further, $R$ can be chosen in such a way that $|\partial_{x}^{2} R(x_{0}) - \partial_{x}^{2} \varphi(x_{0},t_{0})| \leq \delta/2$.  
	Adding a term of the form $\delta (x - x_{0})^{2}/4 + t^{2}$ to $R + S$, there is no loss of generality in assuming in addition that $\varphi(x,t) = R(x) + S(t - t_{0})$ if and only if $(x,t) = (x_{0},t_{0})$.  Since $0 < \partial_{x}\varphi(x_{0},t_{0}) = \partial_{x}R(x_{0})$, after possibly making $r$ smaller, it can similarly be assumed that $\partial_{x}R(x) > 0$ for each $x \in [x_{0} - 2r, x_{0} + 2r]$.  
	
	Given $a \in (0,r)$, let $\rho : \mathbb{R} \to \mathbb{R}$ be a smooth function such that $\rho(x) = 0$ if $|x - x_{0}| \leq \frac{a}{2}$ and $\rho(x) = 1$ if $|x - x_{0}| \geq a$.  Define $\psi_{1}$ by 
		\begin{align*}
			\psi_{1}(x) = (1 - \rho(x)) R(x) + A \rho(x) \mathbf{1}_{[x_{0},+\infty)}(x) + R \left( x_{0} - \frac{a}{2} \right) \rho(x) \mathbf{1}_{(-\infty,x_{0}]}(x).
		\end{align*}
	Since $R(x_{0}) = \varphi(x_{0},t_{0}) < A$, $\psi_{1}$ is a nondecreasing function provided $a$ is small enough.  Fix such a small $a$ henceforth and let $\psi_{2}(t) = S(t - t_{0})$.  Since $F$ is nondecreasing in the spatial variable and $F(x,t) \leq \varphi(x,t)$ for each $(x,t) \in [x_{0}-r,x_{0}+r] \times [t_{0}-r,t_{0}+r]$, the function $\psi(x,t) = \psi_{1}(x) + \psi_{2}(t)$ has the desired properties.    \end{proof}

The next lemma concerns the case when a subsolution is touched above by a test function that is flat close to the touching point.

\begin{lemma} \label{L: test function flat} Fix $T,A > 0$ and let $F : \mathbb{R} \times (0,T) \to [0,A)$ be an upper semicontinuous function that is nondecreasing in the spatial variable.  Assume that $(x_{0},t_{0}) \in \mathbb{R} \times (0,+\infty)$ and $\varphi$ is a smooth function defined in some neighborhood of $(x_{0},t_{0})$ such that $\varphi(x_{0},t_{0}) < A$ and $\partial_{x} \varphi(x_{0},t_{0}) = \partial_{x}^{2} \varphi(x_{0},t_{0}) = 0$.  If $\varphi$ touches $F$ from above at $(x_{0},t_{0})$, then there is an $r > 0$, a nondecreasing $C^{2}$ function $\psi_{1} : \mathbb{R} \to [0,1]$ with $\psi(x) = \psi(x_{0})$ for each $x \leq x_{0}$ and $\psi(x) = A$ for each $x \geq x_{0} + r$, and a $C^{2}$ function $\psi_{2} : [t_{0}-r,t_{0}+r] \to \mathbb{R}$ with $\psi_{2}(t_{0}) = 0$ such that if $\psi$ is the function
		\begin{equation*}
			\psi(x,t) = \psi_{1}(x) + \psi_{2}(t),
		\end{equation*} 
then $\psi(x_{0},t_{0}) = \varphi(x_{0},t_{0})$, $\partial_{t}\psi(x_{0},t_{0}) = \partial_{t}\varphi(x_{0},t_{0})$,
		\begin{gather*}
			F(x,t) \leq \psi(x,t) \quad \text{for each} \, \, (x,t) \in \mathbb{R} \times [t_{0} - r, t_{0} + r], \quad \text{and} \\
			(x_{0},t_{0}) \in \{ (x,t) \in \mathbb{R} \times [t_{0} - r, t_{0} + r] \, \mid \, F(x,t) = \psi(x,t) \} \subseteq (-\infty,x_{0}] \times \{t_{0}\}.
		\end{gather*}
	\end{lemma}
	
\begin{proof} For concreteness, choose $r > 0$ such that $F(x,t) \leq \varphi(x,t)$ for each $(x,t) \in [x_{0}-r,x_{0}+r] \times [t_{0}-r,t_{0}+r]$ with equality when $(x,t) = (x_{0},t_{0})$.  Expanding $\varphi$ in a fourth-order Taylor expansion around $(x_{0},t_{0})$, one finds
		\begin{align*}
			\varphi(x,t) &= \varphi(x_{0},t_{0}) + S(t - t_{0}) + \frac{1}{3!} \partial^{3}_{x} \varphi(x_{0},t_{0}) (x - x_{0})^{3} \\
				&\qquad + \frac{1}{4!} \partial^{4}_{x} \varphi(x_{0},t_{0}) (x - x_{0})^{4} + o( |x - x_{0}|^{4} + |t - t_{0}|^{4} ),
		\end{align*}
	where $S$ is some polynomial of degree four satisfying $S(0) = 0$.  Since $F$ is nondecreasing in the $x$ variable; $F(x_{0},t_{0}) = \varphi(x_{0},t_{0})$; and $F(x,t) \leq \varphi(x,t)$ for $(x,t)$ close enough to $(x_{0},t_{0})$, it follows that
		\begin{align*}
			\partial^{3} \varphi(x_{0},t_{0}) \geq 0.
		\end{align*}
	
	Let $\alpha = \max\{ 1 + \partial_{x}^{4} \varphi(x_{0},t_{0}), 0 \}$.  Making $r$ smaller if necessary, there is no loss of generality in assuming that, for each $(x,t) \in [x_{0}-r,x_{0}+r] \times [t_{0}-r,t_{0}+r]$,
		\begin{align*}
			\varphi(x,t) &\leq \varphi(x_{0},t_{0}) + S(t - t_{0}) + \frac{ (t - t_{0})^{4} }{ 4! } + \frac{1}{3!} \partial_{x}^{3} \varphi(x_{0},t_{0}) ( x - x_{0} )^{3} + \frac{1}{4!} \alpha ( x - x_{0} )^{4} \\
				&=: P(x,t).
		\end{align*}
	with equality if and only if $(x,t) = (x_{0},t_{0})$.  $P$ has the desired form, but it is not nondecreasing in $x$.  This will be rectified in the remainder of the proof.  
	
	Define the nondecreasing function $\eta$ by
		\begin{align*}
			\eta(x) = \left\{ \begin{array}{r l}
						\varphi(x_{0},t_{0}) + \frac{1}{3!} \partial_{x}^{3} \varphi(x_{0},t_{0}) ( x - x_{0} )^{3} + \frac{1}{4!} \alpha (x - x_{0})^{4}, & \text{if} \, \, x_{0} \leq x \, \\
					\varphi(x_{0},t_{0}), & \text{otherwise.}
				\end{array} \right.
		\end{align*}
	Notice that $\eta$ is $C^{2}$.  Given $a \leq r$, let $\rho : \mathbb{R} \to [0,1]$ be a smooth, nondecreasing function such that $\rho(x) = 0$ if $x \leq x_{0} + \frac{a}{2}$ and $\rho(x) = 1$ if $x \geq x_{0} + a$.  Define $h$ by $h(x) = (1 - \rho(x) ) \eta(x) + A \rho(x)$.  Since $\eta(x_{0}) = \varphi(x_{0},t_{0}) < A$, $h$ is a  nondecreasing function provided $a$ is small enough.  Fix such a small $a$ henceforth and let $\psi$ be the modified version of $P$ given by
		\begin{align*}
			\psi(x,t) = h(x) + S(t - t_{0}) + \frac{(t - t_{0})^{4}}{4!} .
		\end{align*}
	Since $F \leq \varphi \leq P$ in $[x_{0} - r, x_{0} + r] \times [t_{0}-r,t_{0}+r]$, $F$ is nondecreasing in the $x$ variable, and $F < A$ globally, it follows that $F \leq \psi$ in $\mathbb{R} \times [t_{0}-r,t_{0}+r]$ with equality only if $x \leq x_{0}$ and $t = t_{0}$.  Therefore, the desired conclusion follows with $\psi_{1}(x) = h(x)$ and $\psi_{2}(t) = S(t-t_{0}) + (t - t_{0})^{4}/4!$. \end{proof}




\section{Scaling Limits of Some Monotone Semigroups in $CDF(\overline{\mathbb{R}})$} \label{A: monotone schemes}

This appendix proves Theorem \ref{T: monotone scheme}, an extension of the convergence framework of \cite{barles_souganidis} to approximations of PDE's posed in $CDF(\overline{\mathbb{R}})$.  As described already briefly in the introduction, technical issues arise stemming from the fact that $CDF(\overline{\mathbb{R}})$ is not a vector space, and the smooth test functions used in the definition of viscosity solution are not necessarily nondecreasing.  These issues are not fatal and, indeed, similar ones have been overcome in other settings, such as geometric flows (see \cite{barles_souganidis-interface}).

In the proof of the theorem, there is one technicality that needs to be overcome.  When one checks that, for instance, the limit $F$ satisfies $\partial_{t} F \leq \mathcal{F}( \partial_{x} F, \partial^{2}_{x} F ) - Q(F)$, it is desirable to invoke the asymptotics of $\mathscr{L}_{\epsilon} \varphi$ for an arbitrary smooth test function $\varphi$ that touches $F$ from above.  If $\partial_{x} \varphi > 0$ at the contact point, then Lemma \ref{L: test function increasing} above shows that $\varphi$ can be assumed to be nondecreasing in $\mathbb{R}$ and then $\mathscr{L}_{\epsilon} \varphi$ is well-defined.  Otherwise, if $\partial_{x} \varphi = 0$ at the contact point, of course, $\varphi$ generally will not be monotone at all, and some work is required to proceed.

At the same time, \emph{smooth} nondecreasing functions have a particular property that is suggestive here: If $G$ is a smooth nondecreasing function in $\mathbb{R}$, and if $\partial_{x} G(x_{0}) = 0$ at some point $x_{0}$, then $\partial_{x}^{2} G(x_{0}) = 0$.  This motivates the following lemma (which applies even if the subsolution is not monotone in the spatial variable).

	\begin{lemma} \label{L: reduced definition} Let $\mathcal{F} : \mathbb{R} \times \mathbb{R} \to \mathbb{R}$ be a continuous function that is nondecreasing in the second argument and $Q : \mathbb{R} \to \mathbb{R}$ be continuous.  An upper semicontinuous function $F : \mathbb{R} \times (0,+\infty) \to \mathbb{R}$ satisfies $\partial_{t} F - \mathcal{F} ( \partial_{x} F, \partial^{2}_{x} F ) + Q(F) \leq 0$ in the viscosity sense if and only if it has the following property: If $\varphi$ is a smooth test function that touches $F$ from above at some point $(x_{0},t_{0})$ and
		\begin{align*}
			\text{either (i)} \quad \partial_{x} \varphi(x_{0},t_{0}) \neq 0, \quad \text{or (ii)} \quad  \partial_{x} \varphi(x_{0},t_{0}) = \partial_{x}^{2} \varphi(x_{0},t_{0}) = 0,
		\end{align*}
	then 
		\begin{equation*}
			\partial_{t} \varphi(x_{0},t_{0}) \leq \mathcal{F} ( \partial_{x} \varphi(x_{0},t_{0}), \partial_{x}^{2} \varphi(x_{0},t_{0}) ).
		\end{equation*}
	\end{lemma}
	
		\begin{proof} The fact that this weaker definition of viscosity subsolution is equivalent to the usual one can be proved exactly as in \cite[Proposition 2.2]{barles_georgelin} (see also \cite[Proof of Theorem 1.2]{morfe_discontinuities}). \end{proof}
		
\begin{remark} A corresponding version of Lemma \ref{L: reduced definition}  holds for supersolutions, as can be readily checked using the transformation $G(x,t) = 1 - F(-x,t)$. \end{remark}
	
Now Lemma \ref{L: reduced definition} can be leveraged in conjunction with Lemma \ref{L: test function flat} above to resolve the issue when $\partial_{x} \varphi = 0$ at the contact point: There is no loss of generality assuming that $\partial_{x}^{2} \varphi = 0$ also holds there, and then Lemma \ref{L: test function flat} shows that $\varphi$ can be taken to be globally nondecreasing in the spatial variable.  
	
\subsection{Proof of Theorem \ref{T: monotone scheme}} \label{A: proof scheme convergence} It is technically convenient to proceed by proving the following statement: For each $(x,t) \in \mathbb{R} \times [0,+\infty)$, 
		\begin{align} \label{E: local uniform convergence}
			\lim_{\delta \downarrow 0} \sup \left\{ | F^{(N)}_{n}(y) - F(x,t) | \, \mid \, | \delta^{(N)} y - x| + | N^{-1} n - t | + N^{-1} \leq \delta \right\} = 0.
		\end{align}
	Notice that this implies that $F_{N} \to F$ locally uniformly in $\mathbb{R} \times [0,+\infty)$ as $N \to +\infty$.

	Since the operators $\{ T^{(N)} \}$ are monotone, and nondecreasing continuous functions can be approximated locally uniformly from above and below by smooth functions, it suffices to consider the case when equality holds $F^{(N)}_{0} = F_{\text{in}}$ for each $N$ and $F_{\text{in}}$ has bounded, uniformly continuous first and second derivatives.  This will be assumed for the remainder of the proof.
	
Throughout what follows, define functions $c_{-}, c_{+} : [0,+\infty) \to \mathbb{R}$ by 
	\begin{align} \label{E: upper and lower limits}
		\frac{d c_{\pm}}{dt} = - Q( c_{\pm} ), \quad c_{-}(0) = F_{\text{in}}(-\infty), \quad c_{+}(0) = F_{\text{in}}(+\infty).
	\end{align}
Assumption (vi) ensures that these ODE's are wellposed.  As discussed in Section \ref{S: abstract framework}, for any $N$, the function $q \mapsto q - Q^{(N)}(q)$ maps $[0,1]$ into itself.  Taken together with assumption (vi), this implies that $Q(0) \leq 0$ and $Q(1) \geq 0$ for each $N$, and, thus,
	\begin{align*}
		0 \leq c_{+}(t), c_{-}(t) \leq 1 \quad \text{for each} \, \, t \geq 0.
	\end{align*}
For each $N$, let $\{c_{-}^{(N)}(n)\}$ and $\{c_{+}^{(N)}(n)\}$ be the sequences of numbers in $[0,1]$ defined analogously by
	\begin{gather*}
		c_{\pm}^{(N)}(n) = T^{(N)} c_{\pm}(n-1) = c_{\pm}^{(N)}(n-1) -  Q^{(N)} ( c_{\pm}^{(N)}(n-1) ), \\
		c_{+}^{(N)}(0) = c_{+}(0), \quad c_{-}^{(N)}(0) = c_{-}(0).
	\end{gather*}
(Assumptions (i) and (iv) imply that $T^{(N)}$ maps constant functions to constant functions, and $\mathcal{L}^{(N)}$ vanishes on constants, as explained in the discussion in Section \ref{S: abstract framework}.)
By assumption (vi) and standard ODE arguments, for any $T > 0$,
	\begin{align} \label{E: convergence of ODE}
		\lim_{\delta \downarrow 0} \sup \left\{ | c_{\pm}^{(N)}( n ) - c_{\pm}(t) | \, \mid \, 0 \leq t \leq T, \, \, | N^{-1} n - t| + N^{-1} \leq \delta \right\} = 0.
	\end{align}
Further, by assumption (iv), for any $N$,
	\begin{align} \label{E: bounds reaction term}
		c_{-}^{(N)}(n) = F^{(N)}_{n}(-\infty) \leq F^{(N)}_{n} \leq F^{(N)}_{n}(+\infty) = c_{+}^{(N)}(n) \quad \text{for each} \, \, n \geq 0.
	\end{align}
	
	\textit{Step 1: Lipschitz Estimate in Time.} Since $F_{\text{in}}$ is currently assumed to have bounded, uniformly continuous first and second derivatives, by the consistency condition (v), for each $y \in \mathbb{R}$,
		\begin{align*}
			\mathscr{L}^{(N)} F_{0}^{(N)} (y) = \mathscr{L}^{(N)} F_{\text{in}}(y) = N^{-1} \mathcal{F} ( \partial_{x} F_{\text{in}} ( \delta^{(N)} y ), \partial^{2}_{x} F_{\text{in}} ( \delta^{(N)} y ) ) + N^{-1} R^{(N)} ( \delta^{(N)} y ) .
		\end{align*}
	where $R^{(N)} \to 0$ uniformly as $N \to +\infty$.  Thus, since $\mathcal{F}$ is bounded on compact sets by continuity and the reaction term $Q^{(N)}$ is of order $N^{-1}$ by assumption (vi), there is a constant $C_{0} > 0$ such that, for sufficiently large $N$,
		\begin{align*}
			\| T^{(N)} F_{0}^{(N)} - F_{0}^{(N)} \|_{\text{sup}} \leq C_{0} N^{-1}.
		\end{align*}
	By Proposition \ref{P: contractive} below, the assumptions (i)-(iv) and (vi) imply the following contractivity property of $T^{(N)}$:
		\begin{align*}
			\| T^{(N)} F - T^{(N)} G \|_{\text{sup}} \leq (1 + L N^{-1} ) \| F - G \|_{\text{sup}}.
		\end{align*}
	Thus, the estimate of $F^{(N)}_{1} - F_{0}^{(N)} = T^{(N)} F_{0}^{(N)} - F_{0}^{(N)}$ can be iterated to find
		\begin{align*}
			\| F^{(N)}_{n+1} - F_{n}^{(N)} \|_{\text{sup}} = \| T^{(N)}F^{(N)}_{n} - T^{(N)} F_{n-1}^{(N)} \|_{\text{sup}} \leq C_{0} N^{-1} ( 1 + N^{-1} L )^{n} \quad \text{for each} \, \, n \in \mathbb{N}.
		\end{align*}
	Summing this estimate in $n$, one obtains a constant $C_{0}'' > 0$ such that that $|F^{(N)}_{n}(x) - F_{\text{in}}(x)| \leq C_{0}'' L^{-1} \{ (1 + N^{-1} L )^{n} - 1 \}$ uniformly in $(x,n)$.  In view of the definition of $F_{N}$, this implies that, for each $T > 0$ and any $N \in \mathbb{N}$,
		\begin{align*}
			\sup \left\{ |F_{N}(x,t) - F_{\text{in}}(x) | \, \mid \, (x,t) \in \mathbb{R} \times [0,T] \right\} \leq C_{0}'' L^{-1} ( e^{L T} - 1 ).
		\end{align*}

	\textit{Step 2: Convergence.}  The remainder of the proof uses the classical half-relaxed limit method.  Let $F_{\star}$ and $F^{\star}$ be the half-relaxed limits of $\{F^{(N)}_{n}\}$ defined by  
		\begin{align*}
			F_{\star}(x,t) &= \lim_{\delta \downarrow 0} \inf \left\{ F_{n}^{(N)}(y) \, \mid \, | \delta^{(N)} y - x | + | N^{-1} n - t| + N^{-1} \leq \delta \right\}, \\ 
			F^{\star}(x,t) &= \lim_{\delta \downarrow 0} \sup \left\{ F_{n}^{(N)}(y) \, \mid \, | \delta^{(N)} y - x |+ | N^{-1} n - t| + N^{-1} \leq \delta \right\},
		\end{align*}
	As is well-known, to prove the local convergence statement \eqref{E: local uniform convergence}, it suffices to show that $F^{\star} = F_{\star} = F$ (see \cite[Remark 6.4]{user} or \cite[Lemma 6.2]{barles}).  To do this, since the bound $F_{\star} \leq F^{\star}$ follows immediately from the definitions, it is only necessary to show that $F^{\star} \leq F \leq F_{\star}$.  This will follow from the comparison principle, Theorem \ref{T: comparison}, as soon as it is established that $F^{\star}$ and $F_{\star}$ define respectively a viscosity sub- and a viscosity supersolution of the initial value problem \eqref{E: initial value problem appendix}.

\textit{Step 3: $F^{\star}$ is a Subsolution.}  In view of what was proved in Step 1 and the uniform continuity of $F_{\text{in}}$, it is clear that
	\begin{align*}
		\lim_{ \delta \downarrow 0 } \sup \left\{ | F^{\star}(x,0) - F_{\text{in}}(y) | \, \mid \, x,y \in \mathbb{R} \, \, \text{such that} \, \, |x - y| \leq \delta \right\} = 0,
	\end{align*} 
so $F^{\star}$ satisfies \eqref{E: initial condition thing}.   Therefore, to see that $F^{\star}$ is a viscosity subsolution of \eqref{E: initial value problem appendix}, it only remains to establish that $\partial_{t} F^{\star} - \mathcal{F} ( \partial_{x} F^{\star}, \partial_{x}^{2} F^{\star} ) + Q(F^{\star}) \leq 0$ in the viscosity sense in $\mathbb{R} \times (0,+\infty)$.

In order to avoid boundary issues that can arise if $F^{\star}$ is somewhere equal to one, it will be convenient to pull $F^{\star}$ down: To see that $F^{\star}$ is a subsolution, it suffices to prove that, for any $\delta \in (0,1)$, the function $G^{\star} : \mathbb{R} \times [0,+\infty) \to [0, 1 -\delta]$ defined by $G^{\star}(x,t) = \max\{ F^{\star}(x,t) - \delta e^{L t}, 0 \}$ satisfies
		\begin{align*}
			\partial_{t} G^{\star} - \mathcal{F} ( \partial_{x} G^{\star}, \partial_{x}^{2} G^{\star} ) + Q ( G^{\star} ) \leq 0 \quad \text{in the viscosity sense in} \, \, \mathbb{R} \times (0,+\infty).
		\end{align*}
(By standard arguments, the subsolution property of $F^{\star}$ can then be recovered upon sending $\delta \downarrow 0$.)
For the rest of this step, fix $\delta$ and consider $G^{\star}$.  Note also that if $\{G^{(N)}_{n}\}$ is the sequence defined for a given $N$ by $G^{(N)}_{n} = \max\{ F^{(N)}_{n} - \delta ( 1 + L N^{-1} )^{n} , 0 \}$, then Lemma \ref{L: pulling down} below implies that 
		\begin{align} \label{E: subsolution final}
			G^{(N)}_{n} \leq T^{(N)} G^{(N)}_{n-1} \quad \text{pointwise for each} \, \, n \in \mathbb{N}.
		\end{align}
Further, by definition of $G^{\star}$,
	\begin{align} \label{E: half relaxed limit G}
		G^{\star}(x,t) = \lim_{\delta \downarrow 0} \sup \left\{ G^{(N)}_{n}(y) \, \mid \, | \delta^{(N)} y - x| + | N^{-1} n - t | + N^{-1} \leq \delta \right\}.
	\end{align}
Thus, in passing from $F$'s to $G$'s, nothing has changed except the stricter upper bound $G^{(N)}_{n} \leq 1 - \delta (1 + L N^{-1} )^{ n }$ holds for each $n$ and the evolution equation has been weakened to an inequality in \eqref{E: subsolution final}.

Suppose that $\varphi$ is a smooth function defined in some open subset of space-time that touches $G^{\star}$ from above at some point $(x_{0},t_{0}) \in \mathbb{R} \times (0,+\infty)$ in its domain.  The goal is to show that 
	\begin{equation} \label{E: viscosity subsolution inequality test function} 
		\partial_{t} \varphi(x_{0},t_{0}) - \mathcal{F} ( \partial_{x} \varphi(x_{0},t_{0}), \partial_{x}^{2} \varphi(x_{0},t_{0}) ) + Q ( \varphi(x_{0},t_{0}) ) \leq 0.
	\end{equation}  
Since $G^{\star}(x,t) \geq \max\{ c_{-}(t) - \delta e^{Lt}, 0 \}$ pointwise in $\mathbb{R} \times [0,+\infty)$ by \eqref{E: bounds reaction term} and the definition of $G^{\star}$, it is convenient to consider two cases separately:
	\begin{equation*}
		\text{(I)} \quad \varphi(x_{0},t_{0}) = \max\{ c_{-}(t_{0}) - \delta e^{Lt_{0}}, 0 \} \quad \text{and} \quad \text{(II)} \quad \varphi(x_{0},t_{0}) > \max\{ c_{-}(t_{0}) - \delta e^{Lt_{0}},0\}.
	\end{equation*}
	
\medskip 

\noindent\textit{Step 3, Case (I).} First, consider case (I).  Since $\varphi \geq G^{\star}$ holds in a neighborhood of $(x_{0},t_{0})$, $\varphi(x_{0},t_{0}) = G^{\star}(x_{0},t_{0})$, and $G^{\star}(x,t) \geq \max\{ c_{-}(t) - \delta e^{Lt}, 0 \}$ globally, in this case, $(x_{0},t_{0})$ is a local minimum of the function $(x,t) \mapsto \varphi(x,t) - \max\{ c_{-}(t) - \delta e^{Lt},0\}$.  In particular, $x_{0}$ is a local minimum of $x \mapsto \varphi(x,t_{0})$, hence
	\begin{align*}
		\partial_{x} \varphi(x_{0},t_{0}) = 0, \quad \partial_{x}^{2} \varphi(x_{0},t_{0}) \geq 0.
	\end{align*}
Therefore, since $\mathcal{F}(0,0) = 0$ by \eqref{E: operator vanishes at zero} and $\mathcal{F}$ is nondecreasing in the second argument,
	\begin{align} \label{E: diffusion nonnegative test function}
		\mathcal{F}(\partial_{x} \varphi(x_{0},t_{0}), \partial_{x}^{2} \varphi(x_{0},t_{0}) ) = \mathcal{F} ( 0 , \partial_{x}^{2} \varphi(x_{0},t_{0}) ) \geq 0.
	\end{align}
Restricting attention this time to the slice $x = x_{0}$, the function $t \mapsto \varphi(x_{0},t) - \max\{c_{-}(t) - \delta e^{Lt},0\}$ has a local minimum at $t = t_{0}$, from which it readily follows that either
	\begin{align*}
		\partial_{t} \varphi(x_{0},t_{0}) = \frac{d}{dt} \{ c_{-}(t) - \delta e^{Lt} \} = - Q(c_{-}(t_{0})) - L \delta e^{Lt_{0}} \leq -Q(\varphi(x_{0},t_{0}))
	\end{align*}
if $c_{-}(t_{0}) - \delta e^{Lt_{0}} > 0$, or
	\begin{align*}
		\partial_{t} \varphi(x_{0},t_{0}) = 0 \leq -Q(0) = -Q ( \varphi(x_{0},t_{0}))
	\end{align*}
if $c_{-}(t_{0}) - \delta e^{Lt_{0}} \leq 0$.  Combining the inequality $\partial_{t} \varphi(x_{0},t_{0}) \leq - Q(\varphi(x_{0},t_{0}))$ with \eqref{E: diffusion nonnegative test function}, one obtains the desired inequality \eqref{E: viscosity subsolution inequality test function}.  

\medskip

\noindent\textit{Step 3, Case (II).} It only remains to consider case (II).  That is, from now on, assume that $\varphi(x_{0},t_{0}) > \max\{ c_{-}(t_{0}) - \delta e^{Lt_{0}}, 0\}$.  

According to Lemma \ref{L: reduced definition}, there is no loss of generality in assuming that either 
	\begin{equation*}
		\text{(a)} \quad \partial_{x} \varphi(x_{0},t_{0}) \neq 0 \quad \text{or} \quad \text{(b)}
		 \quad \partial_{x} \varphi(x_{0},t_{0}) = \partial_{x}^{2} \varphi(x_{0},t_{0}) = 0.
	\end{equation*}
In case (a), since $G^{\star}$ is nondecreasing, it follows that $\partial_{x} \varphi(x_{0},t_{0}) > 0$, and Lemma \ref{L: test function increasing} is applicable.  In case (b), Lemma \ref{L: test function flat} applies.  To avoid repetition, only case (b) will be considered in what follows.

Therefore, assume henceforth that (b) holds.  According to Lemma \ref{L: test function flat} with $A =  1 - \frac{\delta}{2}$,  it is possible to fix an $r > 0$ and a function $\psi : \mathbb{R} \times [t_{0} - r, t_{0} + r] \to \mathbb{R}$ of the form
	\begin{equation*}
		\psi(x,t) = \psi_{1}(x) + \psi_{2}(t) \quad \text{for each} \, \, (x,t) \in \mathbb{R} \times [t_{0} - r, t_{0} + r],
	\end{equation*}
where $\psi_{1}$ is a nondecreasing $C^{2}$ function with $\psi_{1}(x) = \psi_{1}(x_{0})$ for $x \leq x_{0}$ and $\psi_{1}(x) = 1 - \frac{\delta}{2}$ for $x \geq x_{0} + r$; $\psi_{2}$ is $C^{2}$ in $[t_{0}-r,t_{0}+r]$ with $\psi_{2}(t_{0}) = 0$; and 
	\begin{gather*}
		\psi(x_{0},t_{0}) = \varphi(x_{0},t_{0}), \quad \partial_{t}\psi(x_{0},t_{0}) = \partial_{t} \varphi(x_{0},t_{0}), \\
		G^{\star}(x,t) \leq \psi(x,t) \quad \text{for each} \, \, (x,t) \in \mathbb{R} \times [t_{0}-r,t_{0} + r], \\
		(x_{0},t_{0}) \in \{ (x,t) \, \mid \, G^{\star}(x,t) = \psi(x,t) \} \subseteq (-\infty,x_{0}] \times \{t_{0}\}.
	\end{gather*}
Since $\psi_{2}$ is continuous in $[t_{0}-r,t_{0}+r]$ and $\psi_{1}(-\infty) = \varphi(x_{0},t_{0}) > \max\{ c_{-}(t_{0}) - \delta e^{Lt_{0}}, 0 \}$, up to shrinking $r$ if necessary, there is no loss of generality in assuming that there are constants $c_{*}, c^{*} \in (0,1)$ such that, for each $t \in [t_{0}-r,t_{0}+r]$,
	\begin{align} \label{E: room}
		\max\{ c_{-}(t) - \delta e^{Lt}, 0 \} < c_{*} \leq \psi(-\infty,t) \quad \text{and} \quad   1 - \delta < c^{*} \leq \psi(+\infty,t) \leq 1 - \frac{\delta}{4}.
	\end{align}
	
Finally, for each $N > 0$, let $\nu(N)$ be the constant
	\begin{align*}
		\nu(N) = \sup \left\{ G^{(N)}_{n}(y) - \psi(\delta^{(N)}y,N^{-1}n) \, \mid \, y \in \mathbb{R}, \, \, N^{-1} n \in [t_{0} -r, t_{0} + r] \right\}.
	\end{align*}
Notice that \eqref{E: room}, \eqref{E: bounds reaction term}, and \eqref{E: convergence of ODE} imply that there is an $R > 0$ such that, for any large enough $N$,
	\begin{align*}
		 \sup \left\{ G^{(N)}_{n}(y) - \psi( \delta^{(N)}y,N^{-1}n) \, \mid \, \delta^{(N)} |y| \geq R, \, \, N^{-1} n \in [t_{0} -r, t_{0} + r] \right\} < 0.
	\end{align*}
Therefore, by the equation \eqref{E: half relaxed limit G} relating $\{G^{(N)}_{n}\}$ to $G^{\star}$, it follows that there is a sequence $\{ ( N_{k},y_{k},n_{k} ) \}$ such that $(N^{-1}_{k},\nu(N_{k})) \to (0,0)$ as $k \to \infty$; $ \delta^{(N_{k})} |y_{k}| \leq R$ and $| N_{k}^{-1} n_{k} - t_{0}| \leq r$ for each $k$; and
	\begin{align} \label{E: useful convergence}
		G^{(N_{k})}_{n_{k}}(y_{k}) - \psi(\delta^{(N_{k})} y_{k}, N_{k}^{-1}n_{k}) = \nu(N_{k}) \quad \text{for each} \, \, k.
	\end{align}
Further, since the set $\{ (x,t) \in \mathbb{R} \times [t_{0} - r, t_{0} + r] \, \mid \, G^{\star}(x,t) = \psi(x,t)\}$ is contained in $(-\infty,x_{0}] \times \{t_{0}\}$, there is no loss of generality in assuming that there is an $x'_{0} \leq x_{0}$ such that $( \delta^{(N_{k})} y_{k}, N_{k}^{-1} n_{k}) \to (x'_{0},t_{0})$ and $G^{\star}(x'_{0},t_{0}) = \psi(x'_{0},t_{0})$.  Note, in particular, that the inequality $x'_{0} \leq x_{0}$ implies, by the constancy of $\psi_{1}$ in $(-\infty,x_{0}]$,
	\begin{align} \label{E: vanishing derivatives in limit}
		\lim_{k \to \infty} (\psi_{1}( \delta^{(N_{k})} y_{k} ), \partial_{x} \psi_{1}( \delta^{(N_{k})} y_{k}), \partial_{x}^{2}\psi_{1}( \delta^{(N_{k})} y_{k})) = (\psi_{1}(x_{0}),0,0).
	\end{align}
	
Finally, for notational ease, let $\Psi_{k} : \mathbb{R} \times [t_{0}-r,t_{0}+r] \to \mathbb{R}$ be the function $\Psi_{k}(y,t) = \psi(\delta^{(N_{k})}y,t) + \nu(N_{k})$.  Since $\nu(N_{k}) \to 0$, the bounds \eqref{E: room} imply that $0 \leq \Psi_{k} \leq 1$ pointwise for all $k$ large enough.  That is, $\Psi_{k} \in CDF(\overline{\mathbb{R}})$ for large enough $k$.  By the choice of $\nu(N_{k})$, 
	\begin{align*}
		G^{(N_{k})}_{n} \leq \Psi_{k}(\cdot,N_{k}^{-1}n)  \quad \text{pointwise for each integer} \, \, n \in  [N(t_{0} - r), N(t_{0} + r)].
	\end{align*}  
Therefore, by monotonicity (assumption (i)),
	\begin{align*}
		T^{(N_{k})}G^{(N_{k})}_{n_{k} - 1} \leq T^{(N_{k})} \Psi_{k}(\cdot,N_{k}^{-1}(n_{k}-1)).
	\end{align*}
By \eqref{E: subsolution final}, \eqref{E: useful convergence}, and the definition of $\Psi_{k}$, this implies
	\begin{align*}
		\Psi_{k}( y_{k}, N_{k}^{-1} n_{k}) = G^{(N_{k})}_{n_{k}}(y_{k}) \leq T^{(N_{k})} G^{(N_{k})}_{n_{k} - 1}(y_{k}) \leq T^{(N_{k})} \Psi_{k}(y_{k},N_{k}^{-1} n_{k} - N_{k}^{-1}).
	\end{align*}
After rearranging terms, this becomes
		\begin{align*}
			&N_{k}\Big( \Psi_{k}(y_{k}, N_{k}^{-1} n_{k}) - \Psi_{k}(y_{k}, N_{k}^{-1} n_{k} - N_{k}^{-1} ) \Big) \\
			&\quad \quad \quad \leq N_{k} ( \mathscr{L}^{(N_{k})} \Psi_{k} )(y_{k},N_{k}^{-1} (n_{k} - 1)) - N_{k} Q^{(N_{k})} (\Psi_{k}(y_{k},N_{k}^{-1} (n_{k} - 1)).
		\end{align*}
	By assumption (iii), the definition of $\Psi_{k}$, and the form of $\psi$, 
		\begin{align*}
			(\mathscr{L}^{(N_{k})} \Psi_{k})(y_{k}, N_{k}^{-1} ( n_{k} - 1 ) ) = (\mathscr{L}^{(N_{k})} (\psi_{1})_{\delta^{(N_{k})}} ) (y_{k}),
		\end{align*}
	hence, by assumptions (v) and (vi) and \eqref{E: vanishing derivatives in limit}, sending $k \to \infty$ above yields
		\begin{align*}
			\partial_{t} \psi(x_{0},t_{0}) - \mathcal{F} ( 0, 0 ) + Q ( \psi ( x_{0}, t_{0}) ) ) &=  \partial_{t}\psi_{2}(t_{0}) - \lim_{k \to \infty} \mathcal{F} ( \partial_{x} \psi_{1}( \delta^{(N_{k})} y_{k} ), \partial_{x}^{2} \psi_{1}( \delta^{(N_{k})} y_{k} ) ) \\
			&\quad + \lim_{k \to \infty} N_{k} Q^{(N_{k})}( \psi ( \delta^{(N_{k})} y_{k}, N_{k}^{-1} n_{k}) ) \\
			&\leq 0.
		\end{align*}
	Since $\partial_{x} \varphi(x_{0},t_{0}) = \partial_{x}^{2} \varphi(x_{0},t_{0}) = 0$ by assumption and $\partial^{m}_{t}\psi(x_{0},t_{0}) = \partial_{t}^{m} \varphi(x_{0},t_{0})$ for $m \in \{0,1\}$, this yields \eqref{E: viscosity subsolution inequality test function} as desired. 
	
\textit{Step 4: $F_{\star}$ is a Supersolution.} The proof that $F_{\star}$ is a supersolution of \eqref{E: initial value problem appendix} can be reduced to the subsolution property of $F^{\star}$ after a suitable transformation.  As in Remark \ref{R: left continuous} above, let $G_{-}$ denote the left-continuous version of a nondecreasing function $G$, and define an involution $G \mapsto G_{\text{rev}}$ on $CDF(\overline{\mathbb{R}})$ via the formula
	\begin{align*}
		G_{\text{rev}}(x) = 1 - G_{-}(-x).
	\end{align*}
Let $\{ \tilde{T}^{(N)} \}$ be the operators on $CDF(\overline{\mathbb{R}})$ obtained from $\{T^{(N)}\}$ by
	\begin{align*}
		\tilde{T}^{(N)} F = ( T^{(N)} (F_{\text{rev}}) )_{\text{rev}}.
	\end{align*}
It is straightforward to check that the family $\{\tilde{T}^{(N)}\}$ satisfies assumptions (i)-(vi), the only difference being $\mathcal{F}$ has to be replaced by the function $\tilde{\mathcal{F}}(v,w) = -\mathcal{F}(v,-w)$ and $Q$ by $\tilde{Q}(q) = -Q( 1 - q)$.

Define $\{\tilde{F}^{(N)}_{n}\}$ by the rule $\tilde{F}^{(N)}_{n} = (F^{(N)}_{n})_{\text{rev}}$.  Notice that the recursion $F^{(N)}_{n} = T^{(N)} F^{(N)}_{n-1}$ implies that $\tilde{F}^{(N)}_{n} = \tilde{T}^{(N)} \tilde{F}^{(N)}_{n-1}$.  Therefore, by Step 3, if $\tilde{F}^{\star}$ is defined by
	\begin{align*}
		\tilde{F}^{\star}(x,t) = \lim_{\delta \downarrow 0} \sup \left\{ \tilde{F}^{(N)}_{n}(y) \, \mid \, |\delta^{(N)} y - x| + |N^{-1} n - t| + N^{-1} \leq \delta \right\},
	\end{align*}
then $\tilde{F}^{\star}$ is a subsolution of the initial value problem \eqref{E: initial value problem appendix} with $\mathcal{F}$ replaced by $\tilde{\mathcal{F}}$, $Q$ by $\tilde{Q}$, and $F_{\text{in}}$ by $(F_{\text{in}})_{\text{rev}}$.  Further, by definition, $\tilde{F}^{\star}(x,t) = 1 - F_{\star}(-x,t)$.  From these last two observations, one readily concludes that $F_{\star}$ is a viscosity supersolution of \eqref{E: initial value problem appendix}.\qed

\subsection{Proof of Corollary \ref{C: arbitrary convergence}} In the proof that follows, it it convenient to work with the half-relaxed limits $\liminf_{*} F_{N}$ and $\limsup^{*} F_{N}$ of the rescaled CDF $F_{N}$ defined by
	\begin{align*}
		(\liminf\nolimits_{*} F_{N})(x,t) &= \lim_{\delta \downarrow 0} \inf \left\{ F_{N}(x',s) \, \mid \, |x' - x| + |s - t| + N^{-1} \leq \delta \right\}, \\
		(\limsup\nolimits^{*} F_{N})(x,t) &= \lim_{\delta \downarrow 0} \sup \left\{ F_{\epsilon}(x',s) \, \mid \, |x' - x| + |s - t| + N^{-1} \leq \delta \right\}.
	\end{align*}
	
\begin{proof}[Proof of Corollary \ref{C: arbitrary convergence}] Fix $F_{\text{in}} \in CDF(\overline{\mathbb{R}})$.  Suppose that, at time $t = 0$, $F_{N}(\cdot,0) \to F_{\text{in}}$ vaguely as $N \to +\infty$.  Let $\underline{F}$ and $\overline{F}$ be the maximal subsolution and the minimal supersolution of the problem \eqref{E: initial value problem appendix}.  In view of Proposition \ref{P: uniqueness criterion} and well-known properties of half-relaxed limits (see \cite[Remark 6.4]{user} or \cite[Lemma 6.2]{barles}), to see that $F_{N}$ converges locally uniformly in $\mathbb{R} \times (0,+\infty)$ to the discontinuous viscosity solution as $N \to +\infty$, it suffices to establish that
			\begin{align*}
				\underline{F} \leq \liminf\nolimits_{*} F_{N} \leq \limsup\nolimits^{*} F_{N} \leq \overline{F} \quad \text{in} \, \, \mathbb{R} \times [0,+\infty).
			\end{align*}
		The details will be provided only for the lower bound as the upper bound follows similarly.
		
		For each $M \in \mathbb{N}$, generate a rescaled CDF $F_{N,M}$ like $F_{N}$, albeit instead starting from the initial datum
			\begin{align*}
				F_{N,M}(x,0) = 2^{M} \int_{0}^{\infty} F_{N}(x - y,0) \rho(2^{M}y) \, dy,
			\end{align*} 
		where $\rho : (0,+\infty) \to (0,+\infty)$ is a smooth function supported in $[1,2]$ with $\int_{0}^{\infty} \rho(y) \, dy = 1$.  From the fact that $\rho$ is supported in $[1,2]$ and $F_{N}$ is nondecreasing in $x$, it follows that $F_{N,M}(x,0) \leq F_{N,M+1}(x,0) \leq F_{N}(x,0)$ for every $x \in \mathbb{R}$.  Due to the monotonicity assumption (i), this implies that $F_{N,M} \leq F_{N,M+1} \leq F_{N}$ in $\mathbb{R} \times [0,\infty)$ for each $N$ and $M$.  By Theorem \ref{T: monotone scheme}, $F_{N,M} \to F_{M}$ locally uniformly in $\mathbb{R} \times [0,+\infty)$ as $N \to +\infty$, where $F_{M}$ is the solution of \eqref{E: initial value problem appendix} with initial datum $F_{\text{in},M}$ given by $F_{\text{in},M}(x) = 2^{M} \int_{0}^{\infty} F_{\text{in}}(x - y) \rho(2^{M}y) \, dy$.  Further, by construction, the bound $F_{M} \leq F_{M+1} \leq \liminf\nolimits_{*} F_{N}$ holds pointwise in $\mathbb{R} \times [0,+\infty)$ for any fixed $M$.  
		
		Standard arguments (see \cite[Theorem 6.2]{barles}) show that the function $F_{\infty}$ defined by
			\begin{align*}
				F_{\infty}(x,t) = \lim_{\delta \downarrow 0} \inf \left\{ F_{M}(x',s) \, \mid \, |x' - x| + |t - s| + M^{-1} \leq \delta \right\}
			\end{align*}
		satisfies $\partial_{t} F_{\infty} - \mathcal{F}( \partial_{x} F_{\infty}, \partial_{x}^{2} F_{\infty} ) + Q(F_{\infty}) \geq 0$ in the viscosity sense in $\mathbb{R} \times (0,+\infty)$ and $F_{\infty}(\cdot,0) \geq F_{M}(\cdot,0)$ for any $M$.  In the limit $M \to +\infty$, this last bound becomes
			\begin{align*}
				F_{\infty}(x,0) \geq \lim_{ \delta \downarrow 0 } F_{\text{in}}(x - \delta).
			\end{align*} 
		Therefore, by the comparison principle (Theorem \ref{T: comparison}), $F_{\infty}$ is at least as large as any viscosity subsolution of the initial-value problem \eqref{E: initial value problem appendix}.  Therefore, $F_{\infty} \geq \underline{F}$ in $\mathbb{R} \times [0,+\infty)$, and this implies $\liminf\nolimits_{*} F_{N} \geq F_{\infty} \geq \underline{F}$ as desired.   \end{proof}

\subsection{Lipschitz Estimate} \label{A: contractivity} This subsection establishes that operators such as the family $\{ T^{(N)} \}$ are uniformly Lipschitz on $CDF(\overline{\mathbb{R}})$ with the supremum norm.

	\begin{prop} \label{P: contractive} Suppose that $T$ is an operator on $CDF(\overline{\mathbb{R}})$ satisfying the following two properties:
		\begin{itemize}
			\item[(i)] \emph{Monotonicity:} If $F \leq G$ pointwise in $\mathbb{R}$, then $TF\leq TG$ also holds.
		\item[(ii)] \emph{Discrete-Time Reaction-Diffusion Equation:} There is a function $\mathcal{L}$ defined on $CDF(\overline{\mathbb{R}})$ and a function $Q$ defined on $[0,1]$ such that
			\begin{align*}
				TF - F = \mathcal{L}F - Q(F) \quad \text{for each} \, \, F \in CDF(\overline{\mathbb{R}}).
			\end{align*}
		\item[(iii)] \emph{Commutation with Constants:} Given $F,G \in CDF(\overline{\mathbb{R}})$, if the difference $F - G$ is a constant function, then $\mathcal{L}F = \mathcal{L}G$.
		\end{itemize}
		If there is a constant $L > 0$ such that $|Q(q) - Q(q')| \leq L|q - q'|$ for each $q,q' \in [0,1]$, then
		\begin{align*}
			\| T^{n}F - T^{n}G \|_{\sup} \leq (1 + L )^{n} \| F - G \|_{\sup} \quad \text{for any} \, \, F,G \in CDF(\overline{\mathbb{R}}), \, \, n \in \mathbb{N}.
		\end{align*}\end{prop}
		
The proposition follows immediately from the following lemma, which was needed in the proof of Theorem \ref{T: monotone scheme} above:

	\begin{lemma} \label{L: pulling down} Suppose that $T$ is a monotone operator on $CDF(\overline{\mathbb{R}})$ satisfying the assumptions of Proposition \ref{P: contractive}.  If $\{F_{n}\}$ is any sequence with $F_{n} \leq TF_{n-1}$ for each $n$, then, for any $\delta > 0$, the sequence $\{ \tilde{F}_{n} \}$ given by $\tilde{F}_{n} = \max\{ F_{n} - \delta ( 1 + L )^{n}, 0 \}$ satisfies 				\begin{align}  \label{E: exponential comparison argument}
			\tilde{F}_{n} \leq T \tilde{F}_{n-1} \quad \text{for each} \, \, n.
		\end{align}
	
	In particular, if $\{F_{n}\}$ and $\{G_{n}\}$ are any two sequences in $CDF(\overline{\mathbb{R}})$ such that 
		
		\begin{align*}
			F_{n} \leq TF_{n - 1} \quad \text{and} \quad G_{n} \geq TG_{n - 1} \quad \text{ for each} \, \, n \in \mathbb{N} ,
		\end{align*}
	then
		\begin{align*}
			F_{n} \leq G_{n} + ( 1 + L )^{n} \sup \left\{ ( F_{0}(x) - G_{0}(x) )_{+} \, \mid \, x \in \mathbb{R} \right\}  \quad \text{for each} \, \, n.
		\end{align*}
	\end{lemma}
	
		\begin{proof} First, consider the sequence $\{ \tilde{F}_{n} \}$.  Let $\{c_{n}\}_{n \in \mathbb{N}}$ be the sequence of nonnegative numbers determined by the recursion
			\begin{align*}
				c_{n} = (1 + L ) c_{n - 1}, \quad c_{0} = \delta,
			\end{align*}
		so that $\tilde{F}_{n} = \max\{ F_{n} - c_{n}, 0 \}$.

		Fix $n \in \mathbb{N}$.  If $c_{n - 1} \geq 1$, then \eqref{E: exponential comparison argument} holds trivially as $\tilde{F}_{n} = \tilde{F}_{n-1} = 0$ in this case.  Thus, assume $c_{n-1} \in (0,1)$.  By monotonicity and the assumption on $\mathcal{L}$,
			\begin{align*}
				TF_{n-1} - c_{n} &\leq T ( \max\{ F_{n-1}, c_{n-1} \} ) - c_{n} \\
										&= \max\{ F_{n-1}, c_{n-1} \} + \mathcal{L} ( \max\{ F_{n-1} - c_{n-1}, 0 \} ) - Q ( \max\{ F_{n-1} -c_{n-1}, 0 \} ) \\
										&\qquad - Q ( \max\{ F_{n-1}, c_{n-1} \} ) + Q ( \max\{ F_{n-1} - c_{n-1}, 0 \} ) - c_{n} \\
										&\leq T( \max\{ F_{n-1} - c_{n-1}, 0 \} ) + c_{n-1} - c_{n} + Lc_{n-1}.
			\end{align*}
		In view of the formula for $\{c_{n}\}$ and the property of $\{F_{n}\}$, this implies
			\begin{align*}
				F_{n} - c_{n} \leq T ( \max\{ F_{n-1} - c_{n-1},0\} ) = T \tilde{F}_{n-1}.
			\end{align*}
		In view of the fact that $TG \geq 0$ for \emph{any} $G \in CDF(\overline{\mathbb{R}})$, this last bound implies \eqref{E: exponential comparison argument}.
		
		Finally, suppose that, in addition, $G_{n} \geq TG_{n-1}$ for any $n$, and let $\delta = \sup \{ ( F_{0}(x) - G_{0}(x) )_{+} \, \mid \, x \in \mathbb{R} \}$.  Notice that $\tilde{F}_{0} = \max\{ F_{0} - c_{0}, 0 \} \leq G_{0}$ by definition.  Thus, by the monotonicity of $T$,
			\begin{align*}
				\max\{ F_{n} - c_{n}, 0 \} =  \tilde{F}_{n} \leq G_{n} \quad \text{for each} \, \, n.
			\end{align*}
		In particular, $F_{n} \leq G_{n} + c_{n}$ for any $n$, as desired. \end{proof}




\section{Effective Coefficients in the Series-Parallel Graph Case} \label{A: polylog}

This appendix computes the constants $\sigma_{\mathcal{D}}$ and $a_{R}$ in the case of the logarithm of the distance and the resistance of the series-parallel graph.   Recall that, in this setting, the relevant forcing $f$ takes the form
	\begin{align*}
		f(u) = \log ( 1 + e^{-u} ).
	\end{align*}	
The corresponding function $g$ from Section \ref{S: main results} is given by
	\begin{align*}
		g(s) = \log ( e^{s} - 1 ) \quad \text{if} \, \, s \geq \log 2, \quad g(s) = 0, \quad \text{otherwise.}
	\end{align*}
Since $g(s) + f(g(s)) = s$ for $s \geq \log 2$, the formulae for $\sigma_{\mathcal{D}}$ and $a_{R}$ become
	\begin{align*}
		\sigma_{\mathcal{D}} &= - \int_{0}^{\log 2} s \, ds - \int_{\log 2}^{\infty} \log \left( \frac{1}{1 - e^{-s}} \right) \, ds, \\ 
		a_{R} &= 3 \int_{0}^{\log 2} s^{2} \, ds + \int_{\log 2}^{\infty} \log^{2} \left( \frac{1}{1 - e^{-s}} \right) \, ds + 2 \int_{\log 2}^{\infty} s \log \left( \frac{1}{1 - e^{-s}} \right) \, ds
	\end{align*}

The integrals above can be computed using polylogarithms, a family of functions related to the Riemann zeta function $\zeta$.  The dilogarithm and trilogarithm $Li_{2}$ and $Li_{3}$ are defined for $t \in [0,1]$ via the integrals 
	\begin{equation*}
		Li_{2}(t) = -\int_{0}^{t} \frac{1}{x} \log ( 1 - x ) \, dt, \quad Li_{3}(x) = \int_{0}^{t} \frac{1}{x} Li_{2}(x) \, dx.
	\end{equation*}
These functions have the power series representation $Li_{k}(t) = \sum_{n = 1}^{\infty} n^{-k} t^{n}$, hence, in particular, $Li_{k}(1) = \zeta(k)$.  The computations that follow seem to be related to Euler's formula for $Li_{2}( \frac{1}{2} )$ and Landen's formula for $Li_{3}( \frac{1}{2} )$ (see \cite[Section 1.1]{lewin}).

To compute the integrals, use the change-of-variables $x = e^{-s}$, $ds = \frac{1}{x} d x$.  In the case of $\sigma_{\mathcal{D}}$, this leads to
	\begin{align*}
		\sigma_{\mathcal{D}} &= \int_{ \frac{1}{2} }^{ 1 } \log ( x ) \frac{1}{ x } \, d x + \int_{ 0 }^{ \frac{1}{2} } \log ( 1 - x ) \frac{1}{x} \, d x \\
			&= - \frac{1}{2} \log^{2}(\frac{1}{2}) + \frac{1}{2} \int_{ \frac{1}{2} }^{ 1 } \log ( x ) \frac{ 1 }{ 1 - x } \, d x + \frac{1}{2} \int_{0}^{\frac{1}{2}} \log ( 1 - x ) \frac{1}{x} \, d x \\
			&= - \frac{1}{2} \log^{2}(\frac{1}{2}) + \frac{1}{2} \int_{ \frac{1}{2} }^{ 1 } \log ( 1 - x ) \frac{ 1 }{ x } \, d x + \frac{1}{2} \log^{2} ( \frac{1}{2} ) + \frac{1}{2} \int_{ 0 }^{ \frac{1}{2} } \log ( 1 - x ) \frac{1}{x} \, d x \\
			&= - \frac{1}{2} Li_{2}(1) = - \frac{1}{2} \zeta ( 2 ).
	\end{align*}  
In the case of $a_{R}$, this change-of-variables yields
	\begin{align*}
		a_{R} &= 3 \int_{ \frac{1}{2} }^{ 1 } \log^{2} ( x ) \frac{1}{x} \, d x + \int_{ 0 }^{ \frac{1}{2} } \log^{2} ( 1 - x ) \frac{1}{x} \, d x + 2 \int_{ 0 }^{ \frac{1}{2} } \log ( x ) \log ( 1 - x ) \frac{1}{x} \, d x
	\end{align*}
After integrating-by-parts and changing variables in the second integral, one has
	\begin{align*}
		\int_{ 0 }^{ \frac{1}{2} } \log^{2} ( 1 - x ) \frac{1}{ x } \, d x &= 2 \int_{ 0 }^{ \frac{1}{2} } \log ( x ) \log ( 1 - x ) \frac{1}{ 1 - x } \, d x + \log^{3} ( \frac{1}{2} ) \\
			&= 2 \int_{ \frac{1}{2} }^{ 1 } \log ( x ) \log ( 1 - x ) \frac{ 1 }{ x } \, d x + \log^{3} ( \frac{1}{2} ).
	\end{align*}
This leads to
	\begin{align*}
		a_{R} &= - \log^{3}(\frac{1}{2}) + 2 \int_{ \frac{1}{2} }^{ 1 } \log ( x ) \log ( 1 - x ) \frac{1}{ x } \, d x + \log^{3} ( \frac{1}{2} ) + 2 \int_{ 0 }^{ \frac{1}{2} } \log ( x ) \log ( 1 - x ) \frac{1}{ x } \, d x \\
			&= 2 \int_{ 0 }^{ 1 } \log ( x ) \log ( 1 - x ) \frac{ 1 }{ x } \, d x.
	\end{align*}
Finally,  integrating by parts and invoking the definitions of $Li_{2}$ and $Li_{3}$ yields
	\begin{align*}
		a_{R} &= 2 \int_{ 0 }^{ 1 } Li_{2}( x ) \frac{1}{x} \, d x - 2 Li_{2}( 1 ) \log ( 1 ) = 2 \int_{ 0 }^{ 1 } Li_{2} ( x ) \frac{ 1 }{ x } \, d x = 2 Li_{3}(1) = 2 \zeta(3).
	\end{align*}

\bibliographystyle{plain}
\bibliography{series_parallel}
  
\end{document}